\documentclass[10pt,reqno,oneside]{amsproc} \usepackage{amsfonts} \usepackage{amsmath, amsthm, amssymb, enumerate} \DeclareSymbolFont{bbold}{U}{bbold}{m}{n} \DeclareSymbolFontAlphabet{\mathbbold}{bbold} \def\indic{\mathbbold{1}} \textwidth 16truecm \textheight 8.4in\oddsidemargin0.2truecm\evensidemargin0.7truecm\voffset-.9truecm \usepackage{color} \usepackage{marginnote}    \chardef\forshowkeys=0    \chardef\refcheck=0    \chardef\showllabel=0    \chardef\sketches=0 \usepackage[unicode,breaklinks=true,colorlinks=true,linkcolor=blue,urlcolor=blue,citecolor=blue]{hyperref} \ifnum\forshowkeys=1      \usepackage[notref,notcite,color]{showkeys} \fi \ifnum\showllabel=1 \definecolor{colorcccc}{rgb}{0.7,0.7,0.7}   \def\llabel#1{\marginnote{\color{colorcccc}\rm\small(#1)}[-0.0cm]\notag} \def\llabel{\label} \else  \def\llabel#1{\notag} \fi  \ifnum\refcheck=1   \usepackage{refcheck} \fi   \newtheorem{Theorem}{Theorem}[section]   \newtheorem{Lemma}[Theorem]{Lemma}   \newtheorem{definition}{Definition}[section]    \def\minn#1{\bar{#1}}        \def\PP{{\mathbb P}}        \def\dd{d}        \def\un{u^{(n)}}    \def\um{u^{(m)}}                                    \def\wnm{w^{(n,m)}}        \def\snm{{\mathcal S}_{nm}}                                                \def\ff{{f}}        \def\uu{{{u}}}            \def\startnewsection#1#2{\section{#1}\label{#2}\setcounter{equation}{0}}                            \def\NNp{{\mathbb N}}        \def\RR{{\mathbb R}}     \def\TT{{\mathbb T}}    \def\WW{{\mathbb W}}    \def\EE{{\mathbb E}}    \def\comma{ {\rm ,\qquad{}} }             \def\commaone{ {\rm ,\qquad{}} }                                                \def\indeq{\qquad{}}                
                    \def\colb{\color{black}} \definecolor{colorpppp}{rgb}{0.6,0.0,0.1} \definecolor{colorgggg}{rgb}{.0,0.4,0.0} \definecolor{colorgray}{rgb}{0.6,0.6,0.6}   \def\cole{\color{colorpppp}}                   \definecolor{colororange}{rgb}{0.8,0.2,0} \definecolor{colorpurple}{rgb}{0.6,0.0,0.6}  \def\cole{\color{black}} 
\begin{document} \title[Stochastic NSE with $L^p$ data]{Local existence of strong solutions to the stochastic Navier-Stokes equations with $L^{p}$~data} \author[I.~Kukavica]{Igor Kukavica} \address{Department of Mathematics, University of Southern California, Los Angeles, CA 90089} \email{kukavica@usc.edu} \author[F.~Xu]{Fanhui Xu} \address{Department of Mathematical Sciences, Carnegie Mellon University, Pittsburgh, PA 15213} \email{fanhuix@andrew.cmu.edu} \maketitle \tableofcontents \date{} \begin{abstract} For the stochastic Navier-Stokes equations  with a multiplicative white noise on ${\mathbb T}^{3}$, we prove that there exists a unique strong solution locally in time when the initial datum belongs to $L^p(\Omega; L^p)$ for~$p>3$. \end{abstract} \startnewsection{Introduction}{sec1} We consider the stochastic Navier-Stokes equation (SNSE)    \begin{align}     &d\uu( t,x)       =\Delta \uu( t,x)\,dt             -\mathcal{P}\bigl(( \uu( t,x)\cdot \nabla)\uu( t,x)\bigr)\,dt+\sigma(\uu( t,x))\,d\WW(t),     \label{FGHVBNFJHERTWERDFSDFGSDFGSDFGSDFGDFGWERTWERT01}      \\&     \nabla\cdot \uu = 0     \label{FGHVBNFJHERTWERDFSDFGSDFGSDFGSDFGDFGWERTWERT02}    \\&     \uu( 0,x)= \uu_0 ( x)\comma  ( t,x)\in( 0,\infty)\times\TT^3    \label{FGHVBNFJHERTWERDFSDFGSDFGSDFGSDFGDFGWERTWERT03}   \end{align} on the 3D torus $\TT^3=[0,1]^3$, where we assume that the initial datum $u_0$ satisfies $\nabla\cdot u_0=0$ and $\int_{\TT^3} u_0=0$.  The variable $u$ denotes the velocity and $\mathcal{P}$ stands for the Leray projector onto the mean zero divergence-free fields. The stochastic term $\sigma(\uu)d\WW(t)$ represents a \def\DFGDSFSDFGHSFGH{\int} \def\DFGDSFSDFGHSFGI{\partial} possibly infinite-dimensional multiplicative white noise and is interpreted in the It\^{o} sense. Physically, it accounts for random velocity-dependent perturbations during the flow evolution.  \par In this paper, we obtain a pathwise unique strong solution of \eqref{FGHVBNFJHERTWERDFSDFGSDFGSDFGSDFGDFGWERTWERT01}--\eqref{FGHVBNFJHERTWERDFSDFGSDFGSDFGSDFGDFGWERTWERT03} in $L^p(\TT^3)$ for  the full range of exponents $p>3$.  This problem was  previously addressed  in the paper \cite{KXZ}, where we used the  fixed point technique utilizing a multiplicative cut-off of the nonlinear term, which in turn led to the restriction $p>5$ for the initial data. In this paper, we change the approach to using a spectral Galerkin type approximation.  These approximations are generally well-suited for the Hilbert space setting but are not considered well-adapted for $L^p$ type approximations. However, using a square cut-off in Fourier space, we show that the solutions of the approximation converge in suitable  spaces to the sought-after local $L^{p}$ solution of the stochastic Navier-Stokes system for~$p>3$. \par When considering the approximating system and  assuming an additive noise,  one can write the equation as the difference of the stochastic Stokes equation  and an auxiliary deterministic equation; cf.~\cite{F, MS} for applications of this approach  to additive white noise and additive L\'{e}vy noise.  When having a general multiplicative noise, we lose such an advantage. Here, we reduce the finite-dimensional approximations of \eqref{FGHVBNFJHERTWERDFSDFGSDFGSDFGSDFGDFGWERTWERT01} to equations with an additive noise by a fixed point argument, and then we utilize stopping times to linearize $( \uu\cdot \nabla)\uu$ to derive energy estimates. However, introducing stopping times is known for bringing the potential issue of them being degenerate, which prevents one from claiming a limit on a nontrivial time interval for the sequence of approximating solutions. We circumvent this obstacle by employing the estimates in probability subspaces that monotonically expand to the whole probability space. In each of these subspaces, the sequence of stopping times has a positive limit almost surely, and thus we can extract a pathwise limit of approximating solutions in that space. Then, we extend the result to the whole probability space by using indicator functions and showing convergence in an appropriate sense. There are also some challenges when studying SNSE in non-Hilbert trajectory spaces. First, the usual Galerkin scheme does not converge in $L^p$ when $p\neq 2$. In this paper, we use the rectangular partial sums instead (cf.~\eqref{FGHVBNFJHERTWERDFSDFGSDFGSDFGSDFGDFGWERTWERT40} below) and prove their continuity in $L^p(l^2)$ by applying a vector analog of the Calder\'{o}n-Zygmund theorem. Yet as a consequence, there is no available result ensuring the existence of approximating solutions for this finite scheme. Hence, we first construct approximating solutions in $H^1(\TT^3)$. Then we derive estimates in $L^p(\TT^3)$, where  the one of the important ingredients is Lemma~\ref{T01}, which is proven in~\cite{KXZ} using \cite{R}. \par A study of the SNSE dates back to the work of Bensoussan and Temam (cf.~\cite{BeT}) in the 1970's. Early investigations were usually conducted in Hilbert settings. Existing results on this include~\cite{MeS}, where the SNSE with additive white noise was considered in 2D domains, and a global strong solution was shown to exist, and \cite{F}, where the consideration was in 3D. in~\cite{FS}, the authors proved the same results for the SNSE in 2D unbounded domains with multiplicative noise. in~\cite{GZ}, the authors constructed a maximal strong solution for the equation in 3D bounded domains assuming $H^1$ regularity of the initial data. If the SNSE is equipped with a non-degenerate noise and  sufficiently small initial data in $H^s$, then by \cite{Ki}, a global strong solution exists with large probability. Development of the theory naturally leads to considerations in Banach spaces. An effort in this direction includes \cite{AgV}, where the authors studied the equation with multiplicative noise in critical Besov spaces, and \cite{KXZ}, where the SNSE with $L^p$ initial data was considered. Under the condition that $p>5$, a unique strong solution exists globally with large probability. There are also investigations on various notions of solutions to the SNSE or to relevant equations with different types of noise; cf.~\cite{BF, FG, MR} for results on martingale solutions of the SNSE, \cite{BCF, CC, DZ, MS} for the mild formulation subject to white noise, and \cite{BT, ZBL1, ZBL2, FRS} for mild formulation with L\'{e}vy noise. in~\cite{GV}, the stochastic Euler equation with linear multiplicative noise was addressed in $W^{s,p}$, and in~\cite{BR}, the vorticity equation of SNSE with a convolution-type noise was considered. \par The paper is organized as follows. Section~\ref{sec2} is an introduction to the approximating finite dimensional scheme, while it also contains preliminaries on stochastic integration.  In Section~\ref{sec3}, we state our assumptions and the main result.  In Section~\ref{sec4}, we establish the local existence and the pathwise uniqueness of strong solutions. Furthermore,  we prove that \[   \EE\left[ \sup_{0\leq s\leq \tau}\Vert\uu(s,\cdot)\Vert_p^p \right] \leq  C\EE\Bigl[\Vert\uu_0\Vert_p^p +1 \bigr] \] and  \[   \EE\left[ \DFGDSFSDFGHSFGH_0^{\tau}  \sum_{j}    \DFGDSFSDFGHSFGH_{\TT^3} | \nabla (|\uu_j(s,x)|^{p/2})|^2 \,dx\,ds \right] \leq  C\EE\bigl[\Vert\uu_0\Vert_p^p +1 \bigr] , \] up to the maximal time of existence $\tau$. \par \startnewsection{Notation and preliminaries}{sec2} \subsection{Basic Notation}\colb The Fourier coefficients of an integrable function $f$ are denoted by   \begin{equation}    \mathcal{F}f(m)=\hat{f}(m)=\DFGDSFSDFGHSFGH_{\TT^3} f(x)e^{- 2\pi i m\cdot x}\,dx    \comma m\in \mathbb{Z}^d    ,    \llabel{8Th sw ELzX U3X7 Ebd1Kd Z7 v 1rN 3Gi irR XG KWK0 99ov BM0FDJ Cv k opY NQ2 aN9 4Z 7k0U nUKa mE3OjU 8D F YFF okb SI2 J9 V9gV lM8A LWThDP nP u 3EL 7HP D2V Da ZTgg zcCC mbvc70 qq P cC9 mt6 0og cr TiA3 HEjw TK8ymK eu J Mc4 q6d Vz2 00 XnYU tLR9 GYjPXv FO V r6W 1zU K1W bP ToaW JJuK nxBLnd 0f t DEb Mmj 4lo HY yhZy MjM9 1zQS4p 7z 8 eKa 9h0 Jrb ac ekci rexG 0z4n3x z0 Q OWS vFj 3jL hW XUIU 21iI AwJtI3 Rb W a90 I7r zAI qI 3UEl UJG7 tLtUXz w4 K QNE TvX zqW au jEMe nYlN IzLGxg B3 A uJ8 6VS 6Rc PJ 8OXW w8im tcKZEz Ho p 84G 1gS As0 PC owMI 2fLK TdD60y nH g 7lk NFj JLq Oo Qvfk fZBN G3o1Dg Cn 9 hyU h5V SP5 z6 1qvQ wceU dVJJsB vX D G4E LHQ HIa PT bMTr sLsm tXGyOB 7p 2 Os4 3US bq5 ik 4Lin 769O TkUxmp I8 u GYn fBK bYI 9A QzCF w3h0 geJftZ ZK U 74r Yle ajm km ZJdi TGHO OaSt1N nl B 7Y7 h0y oWJ ry rVrT zHO8 2S7oub QA W x9d z2X YWB e5 Kf3A LsUF vqgtM2 O2 I dim rjZ 7RN 28 4KGY trVa WW4nTZ XV b RVo Q77 hVL X6 K2kq FWFm aZnsF9 Ch p 8Kx rsc SGP iS tVXB J3xZ cD5IP4 Fu 9 Lcd TR2 Vwb cL DlGK 1ro3 EEyqEA zw 6 sKe Eg2 sFf jz MtrZ 9kbd xNw66c xf t lzD GZh xQA WQ KkSX jqmm rEpNuG 6P y loq 8hH lSf Ma LXm5 RzEX W4Y1Bq ibEQ04}   \end{equation} where $d$ is the space dimension, while the Fourier inversion formula reads   \begin{equation}    (\mathcal{F}^{-1}g)(x)       =\sum_{m\in  \mathbb{Z}^d} g(m)e^{2\pi i m\cdot x}      .    \llabel{ 3 UOh Yw9 5h6 f6 o8kw 6frZ wg6fIy XP n ae1 TQJ Mt2 TT fWWf jJrX ilpYGr Ul Q 4uM 7Ds p0r Vg 3gIE mQOz TFh9LA KO 8 csQ u6m h25 r8 WqRI DZWg SYkWDu lL 8 Gpt ZW1 0Gd SY FUXL zyQZ hVZMn9 am P 9aE Wzk au0 6d ZghM ym3R jfdePG ln 8 s7x HYC IV9 Hw Ka6v EjH5 J8Ipr7 Nk C xWR 84T Wnq s0 fsiP qGgs Id1fs5 3A T 71q RIc zPX 77 Si23 GirL 9MQZ4F pi g dru NYt h1K 4M Zilv rRk6 B4W5B8 Id 3 Xq9 nhx EN4 P6 ipZl a2UQ Qx8mda g7 r VD3 zdD rhB vk LDJo tKyV 5IrmyJ R5 e txS 1cv EsY xG zj2T rfSR myZo4L m5 D mqN iZd acg GQ 0KRw QKGX g9o8v8 wm B fUu tCO cKc zz kx4U fhuA a8pYzW Vq 9 Sp6 CmA cZL Mx ceBX Dwug sjWuii Gl v JDb 08h BOV C1 pni6 4TTq Opzezq ZB J y5o KS8 BhH sd nKkH gnZl UCm7j0 Iv Y jQE 7JN 9fd ED ddys 3y1x 52pbiG Lc a 71j G3e uli Ce uzv2 R40Q 50JZUB uK d U3m May 0uo S7 ulWD h7qG 2FKw2T JX z BES 2Jk Q4U Dy 4aJ2 IXs4 RNH41s py T GNh hk0 w5Z C8 B3nU Bp9p 8eLKh8 UO 4 fMq Y6w lcA GM xCHt vlOx MqAJoQ QU 1 e8a 2aX 9Y6 2r lIS6 dejK Y3KCUm 25 7 oCl VeE e8p 1z UJSv bmLd Fy7ObQ FN l J6F RdF kEm qM N0Fd NZJ0 8DYuq2 pL X JNz 4rO ZkZ X2 IjTD 1fVt z4BmFI Pi 0 GKD R2W PhO zH zTLP lbAE OT9XW0 gb T Lb3 XRQ qGG 8o 4TPE 6WREQ05}   \end{equation} We adopt the same notation for the Fourier transform of a distribution $f\in {\mathcal D}'=(C^{\infty}({\mathbb T}^{d}))'$. As usual, $W^{s,p}( \TT^d)$, where $p>1$, denotes the class of functions  $f\in\mathcal{D}'( \TT^d)$ for which $     \Vert  f\Vert_{s,p}      =     \Vert J^s f\Vert_p     <\infty $    , where   \begin{equation}     J^s f (x) :=\sum_{m\in  \mathbb{Z}^d}(1+4\pi^2| m|^{2})^{s/2}\hat{f}(m) e^{2\pi i m\cdot x}\comma x\in \mathbb{T}^d    \comma s\in{\mathbb R}     ,    \llabel{c uMqMXh s6 x Ofv 8st jDi u8 rtJt TKSK jlGkGw t8 n FDx jA9 fCm iu FqMW jeox 5Akw3w Sd 8 1vK 8c4 C0O dj CHIs eHUO hyqGx3 Kw O lDq l1Y 4NY 4I vI7X DE4c FeXdFV bC F HaJ sb4 OC0 hu Mj65 J4fa vgGo7q Y5 X tLy izY DvH TR zd9x SRVg 0Pl6Z8 9X z fLh GlH IYB x9 OELo 5loZ x4wag4 cn F aCE KfA 0uz fw HMUV M9Qy eARFe3 Py 6 kQG GFx rPf 6T ZBQR la1a 6Aeker Xg k blz nSm mhY jc z3io WYjz h33sxR JM k Dos EAA hUO Oz aQfK Z0cn 5kqYPn W7 1 vCT 69a EC9 LD EQ5S BK4J fVFLAo Qp N dzZ HAl JaL Mn vRqH 7pBB qOr7fv oa e BSA 8TE btx y3 jwK3 v244 dlfwRL Dc g X14 vTp Wd8 zy YWjw eQmF yD5y5l DN l ZbA Jac cld kx Yn3V QYIV v6fwmH z1 9 w3y D4Y ezR M9 BduE L7D9 2wTHHc Do g ZxZ WRW Jxi pv fz48 ZVB7 FZtgK0 Y1 w oCo hLA i70 NO Ta06 u2sY GlmspV l2 x y0X B37 x43 k5 kaoZ deyE sDglRF Xi 9 6b6 w9B dId Ko gSUM NLLb CRzeQL UZ m i9O 2qv VzD hz v1r6 spSl jwNhG6 s6 i SdX hob hbp 2u sEdl 95LP AtrBBi bP C wSh pFC CUa yz xYS5 78ro f3UwDP sC I pES HB1 qFP SW 5tt0 I7oz jXun6c z4 c QLB J4M NmI 6F 08S2 Il8C 0JQYiU lI 1 YkK oiu bVt fG uOeg Sllv b4HGn3 bS Z LlX efa eN6 v1 B6m3 Ek3J SXUIjX 8P d NKI UFN JvP Ha Vr4T eARP dXEV7B xM 0 A7w 7je p8M EQ06}   \end{equation} with $\Vert \cdot \Vert_p$ denoting the $L^{p}$ norm. When $p=2$, we also write $H^s( \TT^d)=W^{s,2}( \TT^d)$.  The Leray projector   \begin{equation}     ( \mathcal{P} \uu)_j( x)=\sum_{k=1}^{d}( \delta_{jk}+R_j R_k) \uu_k( x)\comma j=1,2,\ldots, d    \llabel{4Q ahOi hEVo Pxbi1V uG e tOt HbP tsO 5r 363R ez9n A5EJ55 pc L lQQ Hg6 X1J EW K8Cf 9kZm 14A5li rN 7 kKZ rY0 K10 It eJd3 kMGw opVnfY EG 2 orG fj0 TTA Xt ecJK eTM0 x1N9f0 lR p QkP M37 3r0 iA 6EFs 1F6f 4mjOB5 zu 5 GGT Ncl Bmk b5 jOOK 4yny My04oz 6m 6 Akz NnP JXh Bn PHRu N5Ly qSguz5 Nn W 2lU Yx3 fX4 hu LieH L30w g93Xwc gj 1 I9d O9b EPC R0 vc6A 005Q VFy1ly K7 o VRV pbJ zZn xY dcld XgQa DXY3gz x3 6 8OR JFK 9Uh XT e3xY bVHG oYqdHg Vy f 5kK Qzm mK4 9x xiAp jVkw gzJOdE 4v g hAv 9bV IHe wc Vqcb SUcF 1pHzol Nj T l1B urc Sam IP zkUS 8wwS a7wVWR 4D L VGf 1RF r59 9H tyGq hDT0 TDlooa mg j 9am png aWe nG XU2T zXLh IYOW5v 2d A rCG sLk s53 pW AuAy DQlF 6spKyd HT 9 Z1X n2s U1g 0D Llao YuLP PB6YKo D1 M 0fi qHU l4A Ia joiV Q6af VT6wvY Md 0 pCY BZp 7RX Hd xTb0 sjJ0 Beqpkc 8b N OgZ 0Tr 0wq h1 C2Hn YQXM 8nJ0Pf uG J Be2 vuq Duk LV AJwv 2tYc JOM1uK h7 p cgo iiK t0b 3e URec DVM7 ivRMh1 T6 p AWl upj kEj UL R3xN VAu5 kEbnrV HE 1 OrJ 2bx dUP yD vyVi x6sC BpGDSx jB C n9P Fiu xkF vw 0QPo fRjy 2OFItV eD B tDz lc9 xVy A0 de9Y 5h8c 7dYCFk Fl v WPD SuN VI6 MZ 72u9 MBtK 9BGLNs Yp l X2y b5U HgH AD bW8X Rzkv UJZShW QH G oEQ12}   \end{equation}  is defined for distributions that have mean zero over ${\mathbb T}^{d}$. Here,   \begin{equation}      R_j=-\frac{\DFGDSFSDFGHSFGI}{\DFGDSFSDFGHSFGI x_j}( -\Delta)^{-\frac{1}{2}}     \comma j=1,2,\ldots,d    \label{FGHVBNFJHERTWERDFSDFGSDFGSDFGSDFGDFGWERTWERT11}   \end{equation} are Riesz transforms. For convenience, we write   \begin{equation}      W_{\rm sol}^{s,p}=\{\mathcal{P}f: f\in W^{s,p}\}.    \llabel{KX yVA rsH TQ 1Vbd dK2M IxmTf6 wE T 9cX Fbu uVx Cb SBBp 0v2J MQ5Z8z 3p M EGp TU6 KCc YN 2BlW dp2t mliPDH JQ W jIR Rgq i5l AP gikl c8ru HnvYFM AI r Ih7 Ths 9tE hA AYgS swZZ fws19P 5w e JvM imb sFH Th CnSZ HORm yt98w3 U3 z ant zAy Twq 0C jgDI Etkb h98V4u o5 2 jjA Zz1 kLo C8 oHGv Z5Ru Gwv3kK 4W B 50T oMt q7Q WG 9mtb SIlc 87ruZf Kw Z Ph3 1ZA Osq 8l jVQJ LTXC gyQn0v KE S iSq Bpa wtH xc IJe4 SiE1 izzxim ke P Y3s 7SX 5DA SG XHqC r38V YP3Hxv OI R ZtM fqN oLF oU 7vNd txzw UkX32t 94 n Fdq qTR QOv Yq Ebig jrSZ kTN7Xw tP F gNs O7M 1mb DA btVB 3LGC pgE9hV FK Y LcS GmF 863 7a ZDiz 4CuJ bLnpE7 yl 8 5jg Many Thanks, POL OG EPOe Mru1 v25XLJ Fz h wgE lnu Ymq rX 1YKV Kvgm MK7gI4 6h 5 kZB OoJ tfC 5g VvA1 kNJr 2o7om1 XN p Uwt CWX fFT SW DjsI wuxO JxLU1S xA 5 ObG 3IO UdL qJ cCAr gzKM 08DvX2 mu i 13T t71 Iwq oF UI0E Ef5S V2vxcy SY I QGr qrB HID TJ v1OB 1CzD IDdW4E 4j J mv6 Ktx oBO s9 ADWB q218 BJJzRy UQ i 2Gp weE T8L aO 4ho9 5g4v WQmoiq jS w MA9 Cvn Gqx l1 LrYu MjGb oUpuvY Q2 C dBl AB9 7ew jc 5RJE SFGs ORedoM 0b B k25 VEK B8V A9 ytAE Oyof G8QIj2 7a I 3jy Rmz yET Kx pgUq 4Bvb cD1b1g KB y oE3 azg elV Nu 8iZEQ14}   \end{equation} \colb
When constructing solutions below, we use a finite approximation
scheme to~\eqref{FGHVBNFJHERTWERDFSDFGSDFGSDFGSDFGDFGWERTWERT01}.  For this purpose, we introduce the  rectangular partial sums   \begin{equation}    \mathcal{T}_n f(x)=\sum_{|k_1|\leq n_1}\cdots \sum_{|k_d|\leq n_d}\hat{f}(k)e^{2\pi i k\cdot x}=\DFGDSFSDFGHSFGH_{\TT^d} f(x-u)D_n(u)\,du,    \label{FGHVBNFJHERTWERDFSDFGSDFGSDFGSDFGDFGWERTWERT40}   \end{equation} for $f\in L^{1}(\mathbb{T}^{d})$, where $n=(n_1, \ldots, n_d)$ is a multi-index and  $D_n(u)=\sum_{|k_1|\leq n_1}\cdots \sum_{|k_d|\leq n_d}e^{2\pi i k\cdot u}$ is the rectangular Dirichlet kernel.  Denote   \begin{equation}    \minn{n}=\min\{n_1,\ldots,n_d\}    \comma n=(n_1,\ldots,n_d)\in\mathbb{N}_0^{d}    .    \llabel{1 w1tq twKx8C LN 2 8yn jdo jUW vN H9qy HaXZ GhjUgm uL I 87i Y7Q 9MQ Wa iFFS Gzt8 4mSQq2 5O N ltT gbl 8YD QS AzXq pJEK 7bGL1U Jn 0 f59 vPr wdt d6 sDLj Loo1 8tQXf5 5u p mTa dJD sEL pH 2vqY uTAm YzDg95 1P K FP6 pEi zIJ Qd 8Ngn HTND 6z6ExR XV 0 ouU jWT kAK AB eAC9 Rfja c43Ajk Xn H dgS y3v 5cB et s3VX qfpP BqiGf9 0a w g4d W9U kvR iJ y46G bH3U cJ86hW Va C Mje dsU cqD SZ 1DlP 2mfB hzu5dv u1 i 6eW 2YN LhM 3f WOdz KS6Q ov14wx YY d 8sa S38 hIl cP tS4l 9B7h FC3JXJ Gp s tll 7a7 WNr VM wunm nmDc 5duVpZ xT C l8F I01 jhn 5B l4Jz aEV7 CKMThL ji 1 gyZ uXc Iv4 03 3NqZ LITG Ux3ClP CB K O3v RUi mJq l5 blI9 GrWy irWHof lH 7 3ZT eZX kop eq 8XL1 RQ3a Uj6Ess nj 2 0MA 3As rSV ft 3F9w zB1q DQVOnH Cm m P3d WSb jst oj 3oGj advz qcMB6Y 6k D 9sZ 0bd Mjt UT hULG TWU9 Nmr3E4 CN b zUO vTh hqL 1p xAxT ezrH dVMgLY TT r Sfx LUX CMr WA bE69 K6XH i5re1f x4 G DKk iB7 f2D Xz Xez2 k2Yc Yc4QjU yM Y R1o DeY NWf 74 hByF dsWk 4cUbCR DX a q4e DWd 7qb Ot 7GOu oklg jJ00J9 Il O Jxn tzF VBC Ft pABp VLEE 2y5Qcg b3 5 DU4 igj 4dz zW soNF wvqj bNFma0 am F Kiv Aap pzM zr VqYf OulM HafaBk 6J r eOQ BaT EsJ BB tHXj n2EU CNleWp cv W JIg gWXEQ60}   \end{equation} It was shown in~\cite{G} that $\mathcal{T}_n$ is a continuous operator on all  $L^q(\TT^d)$ spaces ($1<q<\infty$), i.e., \begin{equation} \Vert\mathcal{T}_n f\Vert_q\leq C_q\Vert  f\Vert_q\comma f\in L^q(\TT^d), \label{FGHVBNFJHERTWERDFSDFGSDFGSDFGSDFGDFGWERTWERT42} \end{equation} and  \begin{equation} \Vert\mathcal{T}_n f-f\Vert_q\to 0\qquad\mbox{as }  \minn{n}\to\infty \label{FGHVBNFJHERTWERDFSDFGSDFGSDFGSDFGDFGWERTWERT53} \end{equation} \par The following lemma is one of our main devices for estimating the nonlinear terms. \par \cole \begin{Lemma} \label{T02} For all $q\in (1,\infty)$ and all $f\in W^{1,q}(\TT^d)$ with $x$-zero mean, there exists $\alpha(q)\in (0,1]$ such that   \begin{equation}\label{FGHVBNFJHERTWERDFSDFGSDFGSDFGSDFGDFGWERTWERT56}   \Vert (\mathcal{T}_n-\mathcal{T}_m   )f\Vert_q    \leq \frac{{C_{q,d}}}{ \minn{m}^{\alpha}\wedge \minn{n}^{\alpha}} \Vert  \nabla f\Vert_q,   \end{equation} for all multiindices $n=(n_1, \ldots, n_d)$ and $m=(m_1, \cdots, m_d)$. \end{Lemma} \colb   \par   \begin{proof}[Proof of Lemma~\ref{T02}] Let ${u}\in (C^{\infty}(\TT^d))^d$, and denote   \begin{equation}     T^{(n,m)}    =\mathcal{T}_{nm} \Delta^{-1}\mbox{div}     ,    \llabel{ Ksn B3 wvmo WK49 Nl492o gR 6 fvc 8ff jJm sW Jr0j zI9p CBsIUV of D kKH Ub7 vxp uQ UXA6 hMUr yvxEpc Tq l Tkz z0q HbX pO 8jFu h6nw zVPPzp A8 9 61V 78c O2W aw 0yGn CHVq BVjTUH lk p 6dG HOd voE E8 cw7Q DL1o 1qg5TX qo V 720 hhQ TyF tp TJDg 9E8D nsp1Qi X9 8 ZVQ N3s duZ qc n9IX ozWh Fd16IB 0K 9 JeB Hvi 364 kQ lFMM JOn0 OUBrnv pY y jUB Ofs Pzx l4 zcMn JHdq OjSi6N Mn 8 bR6 kPe klT Fd VlwD SrhT 8Qr0sC hN h 88j 8ZA vvW VD 03wt ETKK NUdr7W EK 1 jKS IHF Kh2 sr 1RRV Ra8J mBtkWI 1u k uZT F2B 4p8 E7 Y3p0 DX20 JM3XzQ tZ 3 bMC vM4 DEA wB Fp8q YKpL So1a5s dR P fTg 5R6 7v1 T4 eCJ1 qg14 CTK7u7 ag j Q0A tZ1 Nh6 hk Sys5 CWon IOqgCL 3u 7 feR BHz odS Jp 7JH8 u6Rw sYE0mc P4 r LaW Atl yRw kH F3ei UyhI iA19ZB u8 m ywf 42n uyX 0e ljCt 3Lkd 1eUQEZ oO Z rA2 Oqf oQ5 Ca hrBy KzFg DOseim 0j Y BmX csL Ayc cC JBTZ PEjy zPb5hZ KW O xT6 dyt u82 Ia htpD m75Y DktQvd Nj W jIQ H1B Ace SZ KVVP 136v L8XhMm 1O H Kn2 gUy kFU wN 8JML Bqmn vGuwGR oW U oNZ Y2P nmS 5g QMcR YHxL yHuDo8 ba w aqM NYt onW u2 YIOz eB6R wHuGcn fi o 47U PM5 tOj sz QBNq 7mco fCNjou 83 e mcY 81s vsI 2Y DS3S yloB Nx5FBV Bc 9 6HZ EOX UO3 W1 fIF5 jtEM W6KW7D 6EQ10}   \end{equation} where $    \mathcal{T}_{nm}=\mathcal{T}_n-\mathcal{T}_m    $ and   \begin{equation}   \Delta^{-1}\mbox{div}~{u}   =   \sum_{k=(k_1,\ldots,k_d)\neq 0}   \left(\frac{-1}{4\pi^2 |k|^2}\sum_{l=1}^d     2\pi i k_l    \hat{u}_l(k)   \right)e^{2\pi i k\cdot x}.    \llabel{3 t H0F CVT Zup Pl A9aI oN2s f1Bw31 gg L FoD O0M x18 oo heEd KgZB Cqdqpa sa H Fhx BrE aRg Au I5dq mWWB MuHfv9 0y S PtG hFF dYJ JL f3Ap k5Ck Szr0Kb Vd i sQk uSA JEn DT YkjP AEMu a0VCtC Ff z 9R6 Vht 8Ua cB e7op AnGa 7AbLWj Hc s nAR GMb n7a 9n paMf lftM 7jvb20 0T W xUC 4lt e92 9j oZrA IuIa o1Zqdr oC L 55L T4Q 8kN yv sIzP x4i5 9lKTq2 JB B sZb QCE Ctw ar VBMT H1QR 6v5srW hR r D4r wf8 ik7 KH Egee rFVT ErONml Q5 L R8v XNZ LB3 9U DzRH ZbH9 fTBhRw kA 2 n3p g4I grH xd fEFu z6RE tDqPdw N7 H TVt cE1 8hW 6y n4Gn nCE3 MEQ51i Ps G Z2G Lbt CSt hu zvPF eE28 MM23ug TC d j7z 7Av TLa 1A GLiJ 5JwW CiDPyM qa 8 tAK QZ9 cfP 42 kuUz V3h6 GsGFoW m9 h cfj 51d GtW yZ zC5D aVt2 Wi5IIs gD B 0cX LM1 FtE xE RIZI Z0Rt QUtWcU Cm F mSj xvW pZc gl dopk 0D7a EouRku Id O ZdW FOR uqb PY 6HkW OVi7 FuVMLW nx p SaN omk rC5 uI ZK9C jpJy UIeO6k gb 7 tr2 SCY x5F 11 S6Xq OImr s7vv0u vA g rb9 hGP Fnk RM j92H gczJ 660kHb BB l QSI OY7 FcX 0c uyDl LjbU 3F6vZk Gb a KaM ufj uxp n4 Mi45 7MoL NW3eIm cj 6 OOS e59 afA hg lt9S BOiF cYQipj 5u N 19N KZ5 Czc 23 1wxG x1ut gJB4ue Mx x 5lr s8g VbZ s1 NEfI 02Rb pkfEOZ E4 e seo 9te NRU Ai nujf eJEQ33}   \end{equation} (Below, we shall choose ${u}=\nabla f$.) Note that $\Delta^{-1}\mbox{div}~\nabla f=f$ since $f$ has zero mean.  By the orthogonality of $\{e^{2\pi i k\cdot x}\}$ and Parseval's identity,  we have   \begin{equation}\llabel{Ya Ehns0Y 6X R UF1 PCf 5eE AL 9DL6 a2vm BAU5Au DD t yQN 5YL LWw PW GjMt 4hu4 FIoLCZ Lx e BVY 5lZ DCD 5Y yBwO IJeH VQsKob Yd q fCX 1to mCb Ej 5m1p Nx9p nLn5A3 g7 U v77 7YU gBR lN rTyj shaq BZXeAF tj y FlW jfc 57t 2f abx5 Ns4d clCMJc Tl q kfq uFD iSd DP eX6m YLQz JzUmH0 43 M lgF edN mXQ Pj Aoba 07MY wBaC4C nj I 4dw KCZ PO9 wx 3en8 AoqX 7JjN8K lq j Q5c bMS dhR Fs tQ8Q r2ve 2HT0uO 5W j TAi iIW n1C Wr U1BH BMvJ 3ywmAd qN D LY8 lbx XMx 0D Dvco 3RL9 Qz5eqy wV Y qEN nO8 MH0 PY zeVN i3yb 2msNYY Wz G 2DC PoG 1Vb Bx e9oZ GcTU 3AZuEK bk p 6rN eTX 0DS Mc zd91 nbSV DKEkVa zI q NKU Qap NBP 5B 32Ey prwP FLvuPi wR P l1G TdQ BZE Aw 3d90 v8P5 CPAnX4 Yo 2 q7s yr5 BW8 Hc T7tM ioha BW9U4q rb u mEQ 6Xz MKR 2B REFX k3ZO MVMYSw 9S F 5ek q0m yNK Gn H0qi vlRA 18CbEz id O iuy ZZ6 kRo oJ kLQ0 Ewmz sKlld6 Kr K JmR xls 12K G2 bv8v LxfJ wrIcU6 Hx p q6p Fy7 Oim mo dXYt Kt0V VH22OC Aj f deT BAP vPl oK QzLE OQlq dpzxJ6 JI z Ujn TqY sQ4 BD QPW6 784x NUfsk0 aM 7 8qz MuL 9Mr Ac uVVK Y55n M7WqnB 2R C pGZ vHh WUN g9 3F2e RT8U umC62V H3 Z dJX LMS cca 1m xoOO 6oOL OVzfpO BO X 5Ev KuL z5s EW 8a9y otqk cKbDJN Us l pYM JpJ jOWEQ43}   \Vert T^{(n,m)} {u}\Vert_2^2\leq \sum_{|k|\geq \minn{m}\wedge \minn{n}}   \biggl|\frac{1}{4\pi^2 |k|^2}   \sum_{l=1}^d   2\pi i k_l   \hat{u}_l(k)   \biggr|^2   \leq    \frac{C_d}{\minn{m}^2\wedge \minn{n}^{2}} \Vert  {u}\Vert_2^2,    \end{equation} from where   \begin{equation}\label{FGHVBNFJHERTWERDFSDFGSDFGSDFGSDFGDFGWERTWERT62}   \Vert T^{(n,m)} {u}\Vert_2\leq \frac{C_d}{\minn{m}\wedge \minn{n}} \Vert  {u}\Vert_2.    \end{equation} Substituting ${u}$ by $\nabla f$ yields    \begin{equation}\llabel{ Uy 2U4Y VKH6 kVC1Vx 1u v ykO yDs zo5 bz d36q WH1k J7Jtkg V1 J xqr Fnq mcU yZ JTp9 oFIc FAk0IT A9 3 SrL axO 9oU Z3 jG6f BRL1 iZ7ZE6 zj 8 G3M Hu8 6Ay jt 3flY cmTk jiTSYv CF t JLq cJP tN7 E3 POqG OKe0 3K3WV0 ep W XDQ C97 YSb AD ZUNp 81GF fCPbj3 iq E t0E NXy pLv fo Iz6z oFoF 9lkIun Xj Y yYL 52U bRB jx kQUS U9mm XtzIHO Cz 1 KH4 9ez 6Pz qW F223 C0Iz 3CsvuT R9 s VtQ CcM 1eo pD Py2l EEzL U0USJt Jb 9 zgy Gyf iQ4 fo Cx26 k4jL E0ula6 aS I rZQ HER 5HV CE BL55 WCtB 2LCmve TD z Vcp 7UR gI7 Qu FbFw 9VTx JwGrzs VW M 9sM JeJ Nd2 VG GFsi WuqC 3YxXoJ GK w Io7 1fg sGm 0P YFBz X8eX 7pf9GJ b1 o XUs 1q0 6KP Ls MucN ytQb L0Z0Qq m1 l SPj 9MT etk L6 KfsC 6Zob Yhc2qu Xy 9 GPm ZYj 1Go ei feJ3 pRAf n6Ypy6 jN s 4Y5 nSE pqN 4m Rmam AGfY HhSaBr Ls D THC SEl UyR Mh 66XU 7hNz pZVC5V nV 7 VjL 7kv WKf 7P 5hj6 t1vu gkLGdN X8 b gOX HWm 6W4 YE mxFG 4WaN EbGKsv 0p 4 OG0 Nrd uTe Za xNXq V4Bp mOdXIq 9a b PeD PbU Z4N Xt ohbY egCf xBNttE wc D YSD 637 jJ2 ms 6Ta1 J2xZ PtKnPw AX A tJA Rc8 n5d 93 TZi7 q6Wo nEDLwW Sz e Sue YFX 8cM hm Y6is 15pX aOYBbV fS C haL kBR Ks6 UO qG4j DVab fbdtny fi D BFI 7uh B39 FJ 6mYr CUUT f2X38J 43 K EQ41}   \Vert \mathcal{T}_{nm} f\Vert_2\leq \frac{C_2}{ \minn{m}\wedge \minn{n}} \Vert  \nabla f\Vert_2   ,   \end{equation} where $C_2$ is a constant depending only on the dimension $d$. Now, let $q\in(1,2)$ be arbitrary and choose $r=(1+q)/2\in (1,q)$. Then we have   \begin{equation}\label{FGHVBNFJHERTWERDFSDFGSDFGSDFGSDFGDFGWERTWERT35}   \Vert T^{(n,m)} {u}\Vert_r   \leq C_r \Vert  \Delta^{-1}\mbox{div}~{u}\Vert_r   \leq C_r \Vert \nabla \Delta^{-1}\mbox{div}~{u}\Vert_r   \leq C_r \Vert  {u}\Vert_r   ,   \end{equation} where the constant $C_r$ depends on $r$ and thus on $q$; in the second inequality in \eqref{FGHVBNFJHERTWERDFSDFGSDFGSDFGSDFGDFGWERTWERT35}, we used the Poincar\'e inequality $\Vert f\Vert_{r}\leq C_r \Vert \nabla f\Vert_{r}$, which holds for $r\geq 1$ and $f\in W^{1,r}$ such that $\DFGDSFSDFGHSFGH_{{\mathbb T}^{d}} f=0$. Using the Marcinkiewicz interpolation theorem on the inequalities \eqref{FGHVBNFJHERTWERDFSDFGSDFGSDFGSDFGDFGWERTWERT62} and \eqref{FGHVBNFJHERTWERDFSDFGSDFGSDFGSDFGDFGWERTWERT35} yields   \begin{equation}   \Vert T^{(n,m)} {u}\Vert_q\leq C\frac{C_d^{\alpha}C_r^{1-\alpha}}{\minn{m}^{\alpha}\wedge \minn{n}^{\alpha}} \Vert  {u}\Vert_q   ,    \label{FGHVBNFJHERTWERDFSDFGSDFGSDFGSDFGDFGWERTWERT13}   \end{equation} where $\alpha\in(0,1)$ is such that $1/q=\alpha/2+(1-\alpha)/r$, i.e., $\alpha=(1/r-1/q)/(1/r-1/2)$.  Substituting ${u}=\nabla f$  in \eqref{FGHVBNFJHERTWERDFSDFGSDFGSDFGSDFGDFGWERTWERT13} yields \eqref{FGHVBNFJHERTWERDFSDFGSDFGSDFGSDFGDFGWERTWERT56} for $q\in(1,2)$. The proof for $q\in(2,\infty)$ is the same except that we take $r=2q$. \end{proof} \par   Above and in the rest of this paper, $C$ denotes a generic positive constant,  with additional dependence indicated when necessary. \par \subsection{Preliminaries on stochastic analysis} \colb We denote by $\mathcal{H}$ a real separable Hilbert space with a complete orthonormal basis $\{\mathbf{e}_k\}_{k\geq 1}$, and by $(\Omega, \mathcal{F},(\mathcal{F}_t)_{t\geq 0},\mathbb{P})$ a complete probability space with an augmented filtration $(\mathcal{F}_t)_{t\geq 0}$.  With $\{W_k: k\in\NNp\}$  a family of independent $\mathcal{F}_t$-adapted Brownian motions, $\WW( t,\omega):=\sum_{k\geq 1} W_k( t,\omega) \mathbf{e}_k$ is an $\mathcal{F}_t$-adapted and $\mathcal{H}$-valued cylindrical Wiener process. \par For a real separable Hilbert space $\mathcal{Y}$, we define $l^2( \mathcal{H},\mathcal{Y})$ to be the set of Hilbert-Schmidt operators from $\mathcal{H}$ to $\mathcal{Y}$ with the norm  defined by   \begin{equation}    \Vert G\Vert_{l^2( \mathcal{H},\mathcal{Y})}^2:= \sum_{k=1}^{\dim \mathcal{H}} | G \mathbf{e}_k|_{\mathcal{Y}}^2<\infty \comma G\in l^2( \mathcal{H},\mathcal{Y}).    \llabel{yZg 87i gFR 5R z1t3 jH9x lOg1h7 P7 W w8w jMJ qH3 l5 J5wU 8eH0 OogRCv L7 f JJg 1ug RfM XI GSuE Efbh 3hdNY3 x1 9 7jR qeP cdu sb fkuJ hEpw MvNBZV zL u qxJ 9b1 BTf Yk RJLj Oo1a EPIXvZ Aj v Xne fhK GsJ Ga wqjt U7r6 MPoydE H2 6 203 mGi JhF nT NCDB YlnP oKO6Pu XU 3 uu9 mSg 41v ma kk0E WUpS UtGBtD e6 d Kdx ZNT FuT i1 fMcM hq7P Ovf0hg Hl 8 fqv I3R K39 fn 9MaC Zgow 6e1iXj KC 5 lHO lpG pkK Xd Dxtz 0HxE fSMjXY L8 F vh7 dmJ kE8 QA KDo1 FqML HOZ2iL 9i I m3L Kva YiN K9 sb48 NxwY NR0nx2 t5 b WCk x2a 31k a8 fUIa RGzr 7oigRX 5s m 9PQ 7Sr 5St ZE Ymp8 VIWS hdzgDI 9v R F5J 81x 33n Ne fjBT VvGP vGsxQh Al G Fbe 1bQ i6J ap OJJa ceGq 1vvb8r F2 F 3M6 8eD lzG tX tVm5 y14v mwIXa2 OG Y hxU sXJ 0qg l5 ZGAt HPZd oDWrSb BS u NKi 6KW gr3 9s 9tc7 WM4A ws1PzI 5c C O7Z 8y9 lMT LA dwhz Mxz9 hjlWHj bJ 5 CqM jht y9l Mn 4rc7 6Amk KJimvH 9r O tbc tCK rsi B0 4cFV Dl1g cvfWh6 5n x y9Z S4W Pyo QB yr3v fBkj TZKtEZ 7r U fdM icd yCV qn D036 HJWM tYfL9f yX x O7m IcF E1O uL QsAQ NfWv 6kV8Im 7Q 6 GsX NCV 0YP oC jnWn 6L25 qUMTe7 1v a hnH DAo XAb Tc zhPc fjrj W5M5G0 nz N M5T nlJ WOP Lh M6U2 ZFxw pg4Nej P8 U Q09 JX9 n7S kE WixE Rwgy FEQ15}   \end{equation} In this paper, we either regard \eqref{FGHVBNFJHERTWERDFSDFGSDFGSDFGSDFGDFGWERTWERT01} as a vector-valued equation or consider it componentwise. Correspondingly, $\mathcal{Y}={\mathbb R}$ or ${\mathbb R}^{d}$. Let $G=(G_1, \cdots, G_d)$ and $G\mathbf{e}_k:=(G_1\mathbf{e}_k, \cdots, G_d\mathbf{e}_k)$. Then $G\in l^2( \mathcal{H},{\mathbb R}^{d})$ if and only if $G_i\in l^2( \mathcal{H},{\mathbb R})$ for all $i\in\{1, \cdots, d\}$. The Burkholder-Davis-Gundy (BDG) inequality   \begin{equation}    \EE \left[ \sup_{s\in[0,t]}\left| \DFGDSFSDFGHSFGH_0^s G \,d\WW_r \right|_{\mathcal{Y}}^p\right]      \leq         C \EE\left[ \left(\DFGDSFSDFGHSFGH_0^t \Vert G\Vert^2_{ l^2( \mathcal{H},\mathcal{Y})}\, dr \right)^{p/2}\right]    \llabel{vttzp 4A s v5F Tnn MzL Vh FUn5 6tFY CxZ1Bz Q3 E TfD lCa d7V fo MwPm ngrD HPfZV0 aY k Ojr ZUw 799 et oYuB MIC4 ovEY8D OL N URV Q5l ti1 iS NZAd wWr6 Q8oPFf ae 5 lAR 9gD RSi HO eJOW wxLv 20GoMt 2H z 7Yc aly PZx eR uFM0 7gaV 9UIz7S 43 k 5Tr ZiD Mt7 pE NCYi uHL7 gac7Gq yN 6 Z1u x56 YZh 2d yJVx 9MeU OMWBQf l0 E mIc 5Zr yfy 3i rahC y9Pi MJ7ofo Op d enn sLi xZx Jt CjC9 M71v O0fxiR 51 m FIB QRo 1oW Iq 3gDP stD2 ntfoX7 YU o S5k GuV IGM cf HZe3 7ZoG A1dDmk XO 2 KYR LpJ jII om M6Nu u8O0 jO5Nab Ub R nZn 15k hG9 4S 21V4 Ip45 7ooaiP u2 j hIz osW FDu O5 HdGr djvv tTLBjo vL L iCo 6L5 Lwa Pm vD6Z pal6 9Ljn11 re T 2CP mvj rL3 xH mDYK uv5T npC1fM oU R RTo Loi lk0 FE ghak m5M9 cOIPdQ lG D LnX erC ykJ C1 0FHh vvnY aTGuqU rf T QPv wEq iHO vO hD6A nXuv GlzVAv pz d Ok3 6ym yUo Fb AcAA BItO es52Vq d0 Y c7U 2gB t0W fF VQZh rJHr lBLdCx 8I o dWp AlD S8C HB rNLz xWp6 ypjuwW mg X toy 1vP bra uH yMNb kUrZ D6Ee2f zI D tkZ Eti Lmg re 1woD juLB BSdasY Vc F Uhy ViC xB1 5y Ltql qoUh gL3bZN YV k orz wa3 650 qW hF22 epiX cAjA4Z V4 b cXx uB3 NQN p0 GxW2 Vs1z jtqe2p LE B iS3 0E0 NKH gY N50v XaK6 pNpwdB X2 Y v7V 0Ud dTc Pi EQ34}   \end{equation} holds for $p\in [1,\infty)$ and all $G\in l^2( \mathcal{H},\mathcal{Y})$ such that the right hand side above is finite.  For $s\geq0$, introduce   \begin{equation}     \mathbb{W}^{s,p}:=\left\{f\colon\TT^d\to l^2( \mathcal{H},\mathcal{Y})            :              f\mathbf{e}_k\in W^{s,p}(\TT^d) \mbox{ for each }k, \mbox{ and } \DFGDSFSDFGHSFGH_{\TT^d} \Vert J^s f\Vert_{l^2( \mathcal{H},\mathcal{Y})}^p \,dx<\infty                       \right\},    \llabel{dRNN CLG4 7Fc3PL Bx K 3Be x1X zyX cj 0Z6a Jk0H KuQnwd Dh P Q1Q rwA 05v 9c 3pnz ttzt x2IirW CZ B oS5 xlO KCi D3 WFh4 dvCL QANAQJ Gg y vOD NTD FKj Mc 0RJP m4HU SQkLnT Q4 Y 6CC MvN jAR Zb lir7 RFsI NzHiJl cg f xSC Hts ZOG 1V uOzk 5G1C LtmRYI eD 3 5BB uxZ JdY LO CwS9 lokS NasDLj 5h 8 yni u7h u3c di zYh1 PdwE l3m8Xt yX Q RCA bwe aLi N8 qA9N 6DRE wy6gZe xs A 4fG EKH KQP PP KMbk sY1j M4h3Jj gS U One p1w RqN GA grL4 c18W v4kchD gR x 7Gj jIB zcK QV f7gA TrZx Oy6FF7 y9 3 iuu AQt 9TK Rx S5GO TFGx 4Xx1U3 R4 s 7U1 mpa bpD Hg kicx aCjk hnobr0 p4 c ody xTC kVj 8t W4iP 2OhT RF6kU2 k2 o oZJ Fsq Y4B FS NI3u W2fj OMFf7x Jv e ilb UVT ArC Tv qWLi vbRp g2wpAJ On l RUE PKh j9h dG M0Mi gcqQ wkyunB Jr T LDc Pgn OSC HO sSgQ sR35 MB7Bgk Pk 6 nJh 01P Cxd Ds w514 O648 VD8iJ5 4F W 6rs 6Sy qGz MK fXop oe4e o52UNB 4Q 8 f8N Uz8 u2n GO AXHW gKtG AtGGJs bm z 2qj vSv GBu 5e 4JgL Aqrm gMmS08 ZF s xQm 28M 3z4 Ho 1xxj j8Uk bMbm8M 0c L PL5 TS2 kIQ jZ Kb9Q Ux2U i5Aflw 1S L DGI uWU dCP jy wVVM 2ct8 cmgOBS 7d Q ViX R8F bta 1m tEFj TO0k owcK2d 6M Z iW8 PrK PI1 sX WJNB cREV Y4H5QQ GH b plP bwd Txp OI 5OQZ AKyi ix7Qey YI 9 1Ea EQ17}   \end{equation} which are Banach spaces with respect to the norm    \begin{equation}    \Vert f\Vert_{\mathbb{W}^{s,p}}:=\left( \DFGDSFSDFGHSFGH_{\TT^d} \Vert J^s f\Vert_{l^2( \mathcal{H},\mathcal{Y})}^p \,dx\right)^{1/p}.     \llabel{16r KXK L2 ifQX QPdP NL6EJi Hc K rBs 2qG tQb aq edOj Lixj GiNWr1 Pb Y SZe Sxx Fin aK 9Eki CHV2 a13f7G 3G 3 oDK K0i bKV y4 53E2 nFQS 8Hnqg0 E3 2 ADd dEV nmJ 7H Bc1t 2K2i hCzZuy 9k p sHn 8Ko uAR kv sHKP y8Yo dOOqBi hF 1 Z3C vUF hmj gB muZq 7ggW Lg5dQB 1k p Fxk k35 GFo dk 00YD 13qI qqbLwy QC c yZR wHA fp7 9o imtC c5CV 8cEuwU w7 k 8Q7 nCq WkM gY rtVR IySM tZUGCH XV 9 mr9 GHZ ol0 VE eIjQ vwgw 17pDhX JS F UcY bqU gnG V8 IFWb S1GX az0ZTt 81 w 7En IhF F72 v2 PkWO Xlkr w6IPu5 67 9 vcW 1f6 z99 lM 2LI1 Y6Na axfl18 gT 0 gDp tVl CN4 jf GSbC ro5D v78Cxa uk Y iUI WWy YDR w8 z7Kj Px7C hC7zJv b1 b 0rF d7n Mxk 09 1wHv y4u5 vLLsJ8 Nm A kWt xuf 4P5 Nw P23b 06sF NQ6xgD hu R GbK 7j2 O4g y4 p4BL top3 h2kfyI 9w O 4Aa EWb 36Y yH YiI1 S3CO J7aN1r 0s Q OrC AC4 vL7 yr CGkI RlNu GbOuuk 1a w LDK 2zl Ka4 0h yJnD V4iF xsqO00 1r q CeO AO2 es7 DR aCpU G54F 2i97xS Qr c bPZ 6K8 Kud n9 e6SY o396 Fr8LUx yX O jdF sMr l54 Eh T8vr xxF2 phKPbs zr l pMA ubE RMG QA aCBu 2Lqw Gasprf IZ O iKV Vbu Vae 6a bauf y9Kc Fk6cBl Z5 r KUj htW E1C nt 9Rmd whJR ySGVSO VT v 9FY 4uz yAH Sp 6yT9 s6R6 oOi3aq Zl L 7bI vWZ 18c Fa iwpt C1nd Fyp4oEQ18}   \end{equation} Above, we denoted $(J^s f) \mathbf{e}_k=J^s (f \mathbf{e}_k)$. Also, $\mathbb{W}^{0,p}$ is abbreviated as $\mathbb{L}^{p}$. If $f\in \mathbb{L}^{2}$, then $\DFGDSFSDFGHSFGH_0^t f\,d\WW_t$ is an $L^{2}(\TT^d)$-valued Wiener process (cf.~\cite{DZ2}).  \par Letting $( \mathcal{P}f)  \mathbf{e}_k=\mathcal{P} (f \mathbf{e}_k)$, where $\mathcal{P}$ is the Leray projector, we have
$\mathcal{P}f\in \mathbb{W}^{s,p}$ if $ f\in \mathbb{W}^{s,p}$. Write   \begin{equation}    \mathbb{W}_{\rm sol}^{s,p}=\{\mathcal{P}f: f\in \mathbb{W}^{s,p}\}.    \llabel{K xD f Qz2 813 6a8 zX wsGl Ysh9 Gp3Tal nr R UKt tBK eFr 45 43qU 2hh3 WbYw09 g2 W LIX zvQ zMk j5 f0xL seH9 dscinG wu P JLP 1gE N5W qY sSoW Peqj MimTyb Hj j cbn 0NO 5hz P9 W40r 2w77 TAoz70 N1 a u09 boc DSx Gc 3tvK LXaC 1dKgw9 H3 o 2kE oul In9 TS PyL2 HXO7 tSZse0 1Z 9 Hds lDq 0tm SO AVqt A1FQ zEMKSb ak z nw8 39w nH1 Dp CjGI k5X3 B6S6UI 7H I gAa f9E V33 Bk kuo3 FyEi 8Ty2AB PY z SWj Pj5 tYZ ET Yzg6 Ix5t ATPMdl Gk e 67X b7F ktE sz yFyc mVhG JZ29aP gz k Yj4 cEr HCd P7 XFHU O9zo y4AZai SR O pIn 0tp 7kZ zU VHQt m3ip 3xEd41 By 7 2ux IiY 8BC Lb OYGo LDwp juza6i Pa k Zdh aD3 xSX yj pdOw oqQq Jl6RFg lO t X67 nm7 s1l ZJ mGUr dIdX Q7jps7 rc d ACY ZMs BKA Nx tkqf Nhkt sbBf2O BN Z 5pf oqS Xtd 3c HFLN tLgR oHrnNl wR n ylZ NWV NfH vO B1nU Ayjt xTWW4o Cq P Rtu Vua nMk Lv qbxp Ni0x YnOkcd FB d rw1 Nu7 cKy bL jCF7 P4dx j0Sbz9 fa V CWk VFo s9t 2a QIPK ORuE jEMtbS Hs Y eG5 Z7u MWW Aw RnR8 FwFC zXVVxn FU f yKL Nk4 eOI ly n3Cl I5HP 8XP6S4 KF f Il6 2Vl bXg ca uth8 61pU WUx2aQ TW g rZw cAx 52T kq oZXV g0QG rBrrpe iw u WyJ td9 ooD 8t UzAd LSnI tarmhP AW B mnm nsb xLI qX 4RQS TyoF DIikpe IL h WZZ 8ic JGa 91 HxRbEQ19}   \end{equation} Letting $( \mathcal{T}_nf)  \mathbf{e}_k=\mathcal{T}_n(f \mathbf{e}_k)$, where $\mathcal{T}_n$ is the rectangular partial sum in \eqref{FGHVBNFJHERTWERDFSDFGSDFGSDFGSDFGDFGWERTWERT40}, we have the vector valued analog of \eqref{FGHVBNFJHERTWERDFSDFGSDFGSDFGSDFGDFGWERTWERT42}--\eqref{FGHVBNFJHERTWERDFSDFGSDFGSDFGSDFGDFGWERTWERT53}, stated next. \par \cole \begin{Lemma} \label{L01} For every $q\in (1,\infty)$, we have   \begin{equation}    \Vert \mathcal{T}_nf\Vert_{\mathbb{L}^{q}}\leq C_q\Vert f\Vert_{\mathbb{L}^{q}}    \comma n \in {\mathbb N}_0^{d}    \label{FGHVBNFJHERTWERDFSDFGSDFGSDFGSDFGDFGWERTWERT07}   \end{equation} and   \begin{equation}    \Vert \mathcal{T}_n f - f\Vert_{\mathbb{L}^{q}}\to 0    \quad \text{as~$\minn{n}\to\infty$}   ,    \label{FGHVBNFJHERTWERDFSDFGSDFGSDFGSDFGDFGWERTWERT28}   \end{equation} for $f\in {\mathbb L}^{q}$  with $x$-zero mean. \end{Lemma} \colb \par \begin{proof}[Proof of Lemma~\ref{L01}] It is sufficient to prove \eqref{FGHVBNFJHERTWERDFSDFGSDFGSDFGSDFGDFGWERTWERT07} when $d=1$, in which case $n\in {\mathbb N}$. Consider the mapping   \begin{equation}    H \sum_{m\in {\mathbb Z}} \hat f(m)e^{2\pi i m\cdot x}    :=     \sum_{m\geq 1} \hat f(m)e^{2\pi i m\cdot x}    ,    \llabel{ 97kn Whp9sA Vz P o85 60p RN2 PS MGMM FK5X W52OnW Iy o Yng xWn o86 8S Kbbu 1Iq1 SyPkHJ VC v seV GWr hUd ew Xw6C SY1b e3hD9P Kh a 1y0 SRw yxi AG zdCM VMmi JaemmP 8x r bJX bKL DYE 1F pXUK ADtF 9ewhNe fd 2 XRu tTl 1HY JV p5cA hM1J fK7UIc pk d TbE ndM 6FW HA 72Pg LHzX lUo39o W9 0 BuD eJS lnV Rv z8VD V48t Id4Dtg FO O a47 LEH 8Qw nR GNBM 0RRU LluASz jx x wGI BHm Vyy Ld kGww 5eEg HFvsFU nz l 0vg OaQ DCV Ez 64r8 UvVH TtDykr Eu F aS3 5p5 yn6 QZ UcX3 mfET Exz1kv qE p OVV EFP IVp zQ lMOI Z2yT TxIUOm 0f W L1W oxC tlX Ws 9HU4 EF0I Z1WDv3 TP 4 2LN 7Tr SuR 8u Mv1t Lepv ZoeoKL xf 9 zMJ 6PU In1 S8 I4KY 13wJ TACh5X l8 O 5g0 ZGw Ddt u6 8wvr vnDC oqYjJ3 nF K WMA K8V OeG o4 DKxn EOyB wgmttc ES 8 dmT oAD 0YB Fl yGRB pBbo 8tQYBw bS X 2lc YnU 0fh At myR3 CKcU AQzzET Ng b ghH T64 KdO fL qFWu k07t DkzfQ1 dg B cw0 LSY lr7 9U 81QP qrdf H1tb8k Kn D l52 FhC j7T Xi P7GF C7HJ KfXgrP 4K O Og1 8BM 001 mJ PTpu bQr6 1JQu6o Gr 4 baj 60k zdX oD gAOX 2DBk LymrtN 6T 7 us2 Cp6 eZm 1a VJTY 8vYP OzMnsA qs 3 RL6 xHu mXN AB 5eXn ZRHa iECOaa MB w Ab1 5iF WGu cZ lU8J niDN KiPGWz q4 1 iBj 1kq bak ZF SvXq vSiR bLTriS y8 Q YOa mQU EQ16}   \end{equation} for $l^{2}$-valued functions $f$ such that $\hat f(0)=0$. Using the vector valued Calder\'on-Zygmund theorem, we have   \begin{equation}    \Vert H f\Vert_{\mathbb{L}^{q}}    \leq C_q    \Vert f\Vert_{\mathbb{L}^{q}}    ,    \llabel{ZhO rG HYHW guPB zlAhua o5 9 RKU trF 5Kb js KseT PXhU qRgnNA LV t aw4 YJB tK9 fN 7bN9 IEwK LTYGtn Cc c 2nf Mcx 7Vo Bt 1IC5 teMH X4g3JK 4J s deo Dl1 Xgb m9 xWDg Z31P chRS1R 8W 1 hap 5Rh 6Jj yT NXSC Uscx K4275D 72 g pRW xcf AbZ Y7 Apto 5SpT zO1dPA Vy Z JiW Clu OjO tE wxUB 7cTt EDqcAb YG d ZQZ fsQ 1At Hy xnPL 5K7D 91u03s 8K 2 0ro fZ9 w7T jx yG7q bCAh ssUZQu PK 7 xUe K7F 4HK fr CEPJ rgWH DZQpvR kO 8 Xve aSB OXS ee XV5j kgzL UTmMbo ma J fxu 8gA rnd zS IB0Y QSXv cZW8vo CO o OHy rEu GnS 2f nGEj jaLz ZIocQe gw H fSF KjW 2Lb KS nIcG 9Wnq Zya6qA YM S h2M mEA sw1 8n sJFY Anbr xZT45Z wB s BvK 9gS Ugy Bk 3dHq dvYU LhWgGK aM f Fk7 8mP 20m eV aQp2 NWIb 6hVBSe SV w nEq bq6 ucn X8 JLkI RJbJ EbwEYw nv L BgM 94G plc lu 2s3U m15E YAjs1G Ln h zG8 vmh ghs Qc EDE1 KnaH wtuxOg UD L BE5 9FL xIp vu KfJE UTQS EaZ6hu BC a KXr lni r1X mL KH3h VPrq ixmTkR zh 0 OGp Obo N6K LC E0Ga Udta nZ9Lvt 1K Z eN5 GQc LQL L0 P9GX uakH m6kqk7 qm X UVH 2bU Hga v0 Wp6Q 8JyI TzlpqW 0Y k 1fX 8gj Gci bR arme Si8l w03Win NX w 1gv vcD eDP Sa bsVw Zu4h aO1V2D qw k JoR Shj MBg ry glA9 3DBd S0mYAc El 5 aEd pII DT5 mb SVuX o8Nl Y24WCA 6dEQ23}   \end{equation} where $C_q$ is a constant, depending on $q$. Note that the same inequality is satisfied by the operators $    H_n \sum_{m\in {\mathbb Z}} \hat f(m)e^{2\pi i m\cdot x}    =     \sum_{m\leq n} \hat f(m)e^{2\pi i m\cdot x} $.  The proof of \eqref{FGHVBNFJHERTWERDFSDFGSDFGSDFGSDFGDFGWERTWERT07} is then concluded by observing that $   \mathcal{T}_n f    = (I-H_{-(n+1)})H_n  $, for all $n\in {\mathbb N}$. \par First note that $l^{2}$-valued trigonometric polynomials  ${\mathcal T}$ are dense in ${\mathbb L}^{q}$. This holds since the Fej\'er kernel constitutes a partition of unity, implying the convergence of the $l^2$-valued Fourier series in the mean. Assume that $f\in {\mathbb L}^{q}$, and let $\epsilon>0$. Then there exists $P\in {\mathcal T}$ such that $\Vert f-P\Vert_{{\mathbb L}^{q}}\leq \epsilon$, whence   \begin{align}    \begin{split}     \Vert f-\mathcal{T}_n f\Vert_{{\mathbb L}^{q}}     &\leq     \Vert f-P\Vert_{{\mathbb L}^{q}}     +     \Vert P- \mathcal{T}_n P\Vert_{{\mathbb L}^{q}}     +     \Vert \mathcal{T}_n P- \mathcal{T}_n f\Vert_{{\mathbb L}^{q}}     \leq     \epsilon     +     \Vert P- \mathcal{T}_n P\Vert_{{\mathbb L}^{q}}     +   C_q \epsilon     .    \end{split}    \llabel{ f CVF 6Al a6i Ns 7GCh OvFA hbxw9Q 71 Z RC8 yRi 1zZ dM rpt7 3dou ogkAkG GE 4 87V ii4 Ofw Je sXUR dzVL HU0zms 8W 2 Ztz iY5 mw9 aB ZIwk 5WNm vNM2Hd jn e wMR 8qp 2Vv up cV4P cjOG eu35u5 cQ X NTy kfT ZXA JH UnSs 4zxf Hwf10r it J Yox Rto 5OM FP hakR gzDY Pm02mG 18 v mfV 11N n87 zS X59D E0cN 99uEUz 2r T h1F P8x jrm q2 Z7ut pdRJ 2DdYkj y9 J Yko c38 Kdu Z9 vydO wkO0 djhXSx Sv H wJo XE7 9f8 qh iBr8 KYTx OfcYYF sM y j0H vK3 ayU wt 4nA5 H76b wUqyJQ od O u8U Gjb t6v lc xYZt 6AUx wpYr18 uO v 62v jnw FrC rf Z4nl vJuh 2SpVLO vp O lZn PTG 07V Re ixBm XBxO BzpFW5 iB I O7R Vmo GnJ u8 Axol YAxl JUrYKV Kk p aIk VCu PiD O8 IHPU ndze LPTILB P5 B qYy DLZ DZa db jcJA T644 Vp6byb 1g 4 dE7 Ydz keO YL hCRe Ommx F9zsu0 rp 8 Ajz d2v Heo 7L 5zVn L8IQ WnYATK KV 1 f14 s2J geC b3 v9UJ djNN VBINix 1q 5 oyr SBM 2Xt gr v8RQ MaXk a4AN9i Ni n zfH xGp A57 uA E4jM fg6S 6eNGKv JL 3 tyH 3qw dPr x2 jFXW 2Wih pSSxDr aA 7 PXg jK6 GGl Og 5PkR d2n5 3eEx4N yG h d8Z RkO NMQ qL q4sE RG0C ssQkdZ Ua O vWr pla BOW rS wSG1 SM8I z9qkpd v0 C RMs GcZ LAz 4G k70e O7k6 df4uYn R6 T 5Du KOT say 0D awWQ vn2U OOPNqQ T7 H 4Hf iKY Jcl Rq M2g9 lcQEQ29}   \end{align} Since, by $P\in {\mathcal T}$, the middle term vanishes for $n$ sufficiently large, and we obtain $\limsup_{n\to \infty}     \Vert f-\mathcal{T}_n f\Vert_{{\mathbb L}^{q}} \leq C_q \epsilon$. The proof of \eqref{FGHVBNFJHERTWERDFSDFGSDFGSDFGSDFGDFGWERTWERT28} is then concluded by sending $\epsilon\to 0$. \end{proof}
\par Lemma~\ref{L01} confirms that $\DFGDSFSDFGHSFGH_0^t \mathcal{T}_nf\,d\WW_t$ is an $\mathcal{T}_nL^{2}(\TT^d)$-valued Wiener process if $f\in \mathbb{L}^{2}$.  \par Let $A$ be an operator, for instance a differential operator, and let $\sigma$ and $g$ be $(l^2( \mathcal{H},\RR))^d$-valued operators. Given a cylindrical Wiener process $\WW$ relative to a prescribed stochastic basis $(\Omega, \mathcal{F},(\mathcal{F}_t)_{t\geq 0},\mathbb{P})$, we define the local strong solution for the $d$-dimensional stochastic evolution partial differential equation   \begin{equation}    \uu(t,x)=\uu_0(x)      +\DFGDSFSDFGHSFGH_0^t (A\uu(r,x)+\ff(r,x))\,dr      +\DFGDSFSDFGHSFGH_0^t \bigl(\sigma(\uu(r,x))+g(r,x)\bigr) \,d\WW(r)    \label{FGHVBNFJHERTWERDFSDFGSDFGSDFGSDFGDFGWERTWERT20}   \end{equation} in the following way. \par \begin{definition}[Local strong solution] \label{D01} {\rm A pair $(\uu,\tau)$ is called a local strong solution of \eqref{FGHVBNFJHERTWERDFSDFGSDFGSDFGSDFGDFGWERTWERT20} on $(\Omega, \mathcal{F},(\mathcal{F}_t)_{t\geq 0},\mathbb{P})$ if $\tau$ is a stopping time with $\PP(\tau>0)=1$ and if $u\in L^p(\Omega; C([0,\tau\wedge T], L^p)$ is progressively measurable and it satisfies   \begin{align}    \begin{split}      (\uu( t),\phi)      &= (\uu_0,\phi)           +\DFGDSFSDFGHSFGH_0^t (A\uu(r)+\ff(r),\phi)\,dr           +\DFGDSFSDFGHSFGH_0^t \bigl(\sigma(\uu(r))+g(r),\phi\bigr)\,d\WW(r)\quad \PP\mbox{-a.s.}    ,    \end{split}    \llabel{Z cvCNBP 2B b tjv VYj ojr rh 78tW R886 ANdxeA SV P hK3 uPr QRs 6O SW1B wWM0 yNG9iB RI 7 opG CXk hZp Eo 2JNt kyYO pCY9HL 3o 7 Zu0 J9F Tz6 tZ GLn8 HAes o9umpy uc s 4l3 CA6 DCQ 0m 0llF Pbc8 z5Ad2l GN w SgA XeN HTN pw dS6e 3ila 2tlbXN 7c 1 itX aDZ Fak df Jkz7 TzaO 4kbVhn YH f Tda 9C3 WCb tw MXHW xoCC c4Ws2C UH B sNL FEf jS4 SG I4I4 hqHh 2nCaQ4 nM p nzY oYE 5fD sX hCHJ zTQO cbKmvE pl W Und VUo rrq iJ zRqT dIWS QBL96D FU d 64k 5gv Qh0 dj rGlw 795x V6KzhT l5 Y FtC rpy bHH 86 h3qn Lyzy ycGoqm Cb f h9h prB CQp Fe CxhU Z2oJ F3aKgQ H8 R yIm F9t Eks gP FMMJ TAIy z3ohWj Hx M R86 KJO NKT c3 uyRN nSKH lhb11Q 9C w rf8 iiX qyY L4 zh9s 8NTE ve539G zL g vhD N7F eXo 5k AWAT 6Vrw htDQwy tu H Oa5 UIO Exb Mp V2AH puuC HWItfO ru x YfF qsa P8u fH F16C EBXK tj6ohs uv T 8BB PDN gGf KQ g6MB K2x9 jqRbHm jI U EKB Im0 bbK ac wqIX ijrF uq9906 Vy m 3Ve 1gB dMy 9i hnbA 3gBo 5aBKK5 gf J SmN eCW wOM t9 xutz wDkX IY7nNh Wd D ppZ UOq 2Ae 0a W7A6 XoIc TSLNDZ yf 2 XjB cUw eQT Zt cuXI DYsD hdAu3V MB B BKW IcF NWQ dO u3Fb c6F8 VN77Da IH E 3MZ luL YvB mN Z2wE auXX DGpeKR nw o UVB 2oM VVe hW 0ejG gbgz Iw9FwQ hN Y rFI 4pT lqr EQ21}   \end{align} } for all $\phi\in C^{\infty}(\TT^d)$ and all $t\in [0,\tau\wedge T]$. \end{definition} \par \colb If $A$ is a differential operator, we interpret $(A u(r),\phi)$ using integration by parts. \par \begin{definition}[Pathwise uniqueness] 	\label{D02} 	{\rm The equation \eqref{FGHVBNFJHERTWERDFSDFGSDFGSDFGSDFGDFGWERTWERT20} is said to have a pathwise unique strong solution if for any pair of local strong solutions $(\uu,\tau)$ and $(v, \eta)$ subject to $(\Omega, \mathcal{F},(\mathcal{F}_t)_{t\geq 0},\mathbb{P})$ and the same Wiener process $\WW$, we have $ \PP(\uu( t)=v(t),~\forall t \in[0, \tau\wedge \eta])=1 $. 	} \end{definition} \par \startnewsection{Assumptions and main results}{sec3} In this section, we summarize the assumptions and state the main result. We assume throughout that  $d=3$ and $p>3$.  We point out however that the results extend easily to any space dimension $d\geq2$ with the condition $p>d$. On the noise coefficient $\sigma$ in \eqref{FGHVBNFJHERTWERDFSDFGSDFGSDFGSDFGDFGWERTWERT01}, we assume \begin{enumerate} 	\item (sub-linear growth) 	\begin{align} 	\begin{split} 	\sum_{i=1}^3 \Vert \sigma_i(u)\Vert_{\mathbb{L}^{r}} 	\leq C (\Vert u\Vert_{r} + 1)         \comma r\in\{2,p, 3p\}         , 	\end{split} 	\label{FGHVBNFJHERTWERDFSDFGSDFGSDFGSDFGDFGWERTWERT24} 	\end{align} \item (Lipschitz continuity)  	\begin{align} 	\begin{split} 	\sum_{i=1}^3 \Vert \sigma_i(u)-\sigma_i(v)\Vert_{\mathbb{L}^{r}} 	\leq  	C \Vert u-v\Vert_{r} 	\comma  r\in\{2,p\}, 	\end{split} 	\label{FGHVBNFJHERTWERDFSDFGSDFGSDFGSDFGDFGWERTWERT25} 	\end{align} \item (preserved divergence and mean) $\sigma(W_{\rm sol}^{s,p})\subset \mathbb{W}_{\rm sol}^{s,p}$, and $\DFGDSFSDFGHSFGH_{\TT^3} \sigma(u)=0$ if $\DFGDSFSDFGHSFGH_{\TT^3} u=0$, and \item (sublinearity of the gradient) 	\begin{align} 	\begin{split} 	\sum_{i=1}^3 \Vert\nabla(\sigma_i(u))\Vert_{\mathbb{L}^{2}} 	\leq C \Vert u\Vert_{2}. 	\end{split} 	\label{FGHVBNFJHERTWERDFSDFGSDFGSDFGSDFGDFGWERTWERT22} 	\end{align} \end{enumerate} \par Under these conditions, we have the following statement. \cole\begin{Theorem} \label{T04} (Local strong solution up to a stopping time) Let $p> 3$ and $\uu_0\in L^p(\Omega; L^p(\TT^3))$. Then there exists a pathwise unique local strong solution  $(u,\tau)$ to \eqref{FGHVBNFJHERTWERDFSDFGSDFGSDFGSDFGDFGWERTWERT01}--\eqref{FGHVBNFJHERTWERDFSDFGSDFGSDFGSDFGDFGWERTWERT03} such that   \begin{align}   \begin{split}     \EE\left[      \sup_{0\leq s\leq \tau}\Vert\uu(s,\cdot)\Vert_p^p          +\DFGDSFSDFGHSFGH_0^{\tau}                   \sum_{j}    \DFGDSFSDFGHSFGH_{\TT^3} | \nabla (|\uu_j(s,x)|^{p/2})|^2 \,dx\,ds        \right]          \leq           C\EE\bigl[\Vert\uu_0\Vert_p^p                        +1              \bigr]    ,   \end{split}    \llabel{Wn Xzz2 qBba lv3snl 2j a vzU Snc pwh cG J0Di 3Lr3 rs6F23 6o b LtD vN9 KqA pO uold 3sec xqgSQN ZN f w5t BGX Pdv W0 k6G4 Byh9 V3IicO nR 2 obf x3j rwt 37 u82f wxwj SmOQq0 pq 4 qfv rN4 kFW hP HRmy lxBx 1zCUhs DN Y INv Ldt VDG 35 kTMT 0ChP EdjSG4 rW N 6v5 IIM TVB 5y cWuY OoU6 Sevyec OT f ZJv BjS ZZk M6 8vq4 NOpj X0oQ7r vM v myK ftb ioR l5 c4ID 72iF H0VbQz hj H U5Z 9EV MX8 1P GJss Wedm hBXKDA iq w UJV Gj2 rIS 92 AntB n1QP R3tTJr Z1 e lVo iKU stz A8 fCCg Mwfw 4jKbDb er B Rt6 T8O Zyn NO qXc5 3Pgf LK9oKe 1p P rYB BZY uui Cw XzA6 kaGb twGpmR Tm K viw HEz Rjh Te frip vLAX k3PkLN Dg 5 odc omQ j9L YI VawV mLpK rto0F6 Ns 7 Mmk cTL 9Tr 8f OT4u NNJv ZThOQw CO C RBH RTx hSB Na Iizz bKIB EcWSMY Eh D kRt PWG KtU mo 26ac LbBn I4t2P1 1e R iPP 99n j4q Q3 62UN AQaH JPPY1O gL h N8s ta9 eJz Pg mE4z QgB0 mlAWBa 4E m u7m nfY gbN Lz ddGp hhJV 9hyAOG CN j xJ8 3Hg 6CA UT nusW 9pQr Wv1DfV lG n WxM Bbe 9Ww Lt OdwD ERml xJ8LTq KW T tsR 0cD XAf hR X1zX lAUu wzqnO2 o7 r toi SMr OKL Cq joq1 tUGG iIxusp oi i tja NRn gtx S0 r98r wXF7 GNiepz Ef A O2s Ykt Idg H1 AGcR rd2w 89xoOK yN n LaL RU0 3su U3 JbS8 dok8 tw9NQS Y4 j XEQ26}   \end{align} where $C>0$ is a constant depending on $p$. \end{Theorem} \colb \par For the sake of completeness, we state two auxiliary results that were proven in~\cite{KXZ} (cf.~Theorem~4.1 and~Lemma~4.4 in~\cite{KXZ}). \par \cole \begin{Lemma} \label{T01} Let $2<p<\infty$ and $0<T<\infty$. Suppose that $\uu_0\in L^p(\Omega, L^p({\mathbb T}^{d}))$,  $\ff\in L^p(\Omega\times[0,T], W^{-1,q}({\mathbb T}^{d}))$,  and $g\in L^p(\Omega\times[0,T], \mathbb{L}^p({\mathbb T}^{d}))$ are $\mathbb{R}^{D}$-valued with $x$-mean zero $(\omega,t)$~a.s., and   \begin{equation}    \frac{\dd p}{p+\dd-2}    < q    \leq p     ,    \llabel{Y6 25K CcP Ly FRlS p759 DeVbY5 b6 9 jYO mdf b99 j1 5lvL vjsk K2gEwl Rx O tWL ytZ J1y Z5 Pit3 5SOi ivz4F8 tq M JIg QQi Oob Sp eprt 2vBV qhvzkL lf 7 HXA 4so MXj Wd MS7L eRDi ktUifL JH u kes trv rl7 mY cSOB 7nKW MD0xBq kb x FgT TNI wey VI G6Uy 3dL0 C3MzFx sB E 7zU hSe tBQ cX 7jn2 2rr0 yL1Erb pL R m3i da5 MdP ic dnMO iZCy Gd2MdK Ub x saI 9Tt nHX qA QBju N5I4 Q6zz4d SW Y Urh xTC uBg BU T992 uczE mkqK1o uC a HJB R0Q nv1 ar tFie kBu4 9ND9kK 9e K BOg PGz qfK J6 7NsK z3By wIwYxE oW Y f6A Kuy VPj 8B 9D6q uBkF CsKHUD Ck s DYK 3vs 0Ep 3g M2Ew lPGj RVX6cx lb V OfA ll7 g6y L9 PWyo 58h0 e07HO0 qz 8 kbe 85Z BVC YO KxNN La4a FZ7mw7 mo A CU1 q1l pfm E5 qXTA 0QqV MnRsbK zH o 5vX 1tp MVZ XC znmS OM73 CRHwQP Tl v VN7 lKX I06 KT 6MTj O3Yb 87pgoz ox y dVJ HPL 3k2 KR yx3b 0yPB sJmNjE TP J i4k m2f xMh 35 MtRo irNE 9bU7lM o4 b nj9 GgY A6v sE sONR tNmD FJej96 ST n 3lJ U2u 16o TE Xogv Mqwh D0BKr1 Ci s VYb A2w kfX 0n 4hD5 Lbr8 l7Erfu N8 O cUj qeq zCC yx 6hPA yMrL eB8Cwl kT h ixd Izv iEW uw I8qK a0VZ EqOroD UP G phf IOF SKZ 3i cda7 Vh3y wUSzkk W8 S fU1 yHN 0A1 4z nyPU Ll6h pzlkq7 SK N aFq g9Y hj2 hJ 3pWS mi9X gjEQ32}   \end{equation} provided $\dd\geq2$, or $1 <q \leq p$ if $d=1$. Then there exists a unique maximal solution $\uu\in L^p(\Omega; C([0,T], L^p))$ to 
 \begin{align} d\uu( t,x) &=\Delta \uu( t,x)\,dt +\ff( t,x)\,dt+g( t,x)\,d\WW(t), \llabel{apmM Z6 H V8y jig pSN lI 9T8e Lhc1 eRRgZ8 85 e NJ8 w3s ecl 5i lCdo zV1B oOIk9g DZ N Y5q gVQ cFe TD VxhP mwPh EU41Lq 35 g CzP tc2 oPu gV KOp5 Gsf7 DFBlek to b d2y uDt ElX xm j1us DJJ6 hj0HBV Fa n Tva bFA VwM 51 nUH6 0GvT 9fAjTO 4M Q VzN NAQ iwS lS xf2p Q8qv tdjnvu pL A TIw ym4 nEY ES fMav UgZo yehtoe 9R T N15 EI1 aKJ SC nr4M jiYh B0A7vn SA Y nZ1 cXO I1V 7y ja0R 9jCT wxMUiM I5 l 2sT XnN RnV i1 KczL G3Mg JoEktl Ko U 13t saq jrH YV zfb1 yyxu npbRA5 6b r W45 Iqh fKo 0z j04I cGrH irwyH2 tJ b Fr3 leR dcp st vXe2 yJle kGVFCe 2a D 4XP OuI mtV oa zCKO 3uRI m2KFjt m5 R GWC vko zi7 5Y WNsb hORn xzRzw9 9T r Fhj hKb fqL Ab e2v5 n9mD 2VpNzl Mn n toi FZB 2Zj XB hhsK 8K6c GiSbRk kw f WeY JXd RBB xy qjEV F5lr 3dFrxG lT c sby AEN cqA 98 1IQ4 UGpB k0gBeJ 6D n 9Jh kne 5f5 18 umOu LnIa spzcRf oC 0 StS y0D F8N Nz F2Up PtNG 50tqKT k2 e 51y Ubr szn Qb eIui Y5qa SGjcXi El 4 5B5 Pny Qtn UO MHis kTC2 KsWkjh a6 l oMf gZK G3n Hp h0gn NQ7q 0QxsQk gQ w Kwy hfP 5qF Ww NaHx SKTA 63ClhG Bg a ruj HnG Kf4 6F QtVt SPgE gTeY6f JG m B3q gXx tR8 RT CPB1 8kQa jtt6GD rK b 1VY LV3 RgW Ir AyZf 69V8 VM7jHO b7 z Lva XTT VI0 ON KEQ30} \\ \uu( 0,x)&= \uu_0 ( x) \quad{}\mbox{a.s.}\comma x\in\mathbb{T}^d . \llabel{MBA HOwO Z7dPky Cg U S74 Hln FZM Ha br8m lHbQ NSwwdo mO L 6q5 wvR exV ej vVHk CEdX m3cU54 ju Z SKn g8w cj6 hR 1FnZ Jbkm gKXJgF m5 q Z5S ubX vPK DB OCGf 4srh 1a5FL0 vY f RjJ wUm 2sf Co gRha bxyc 0Rgava Rb k jzl teR GEx bE MMhL Zbh3 axosCq u7 k Z1P t6Y 8zJ Xt vmvP vAr3 LSWDjb VP N 7eN u20 r8B w2 ivnk zMda 93zWWi UB H wQz ahU iji 2T rXI8 v2HN ShbTKL eK W 83W rQK O4T Zm 57yz oVYZ JytSg2 Wx 4 Yaf THA xS7 ka cIPQ JGYd Dk0531 u2 Q IKf REW YcM KM UT7f dT9E kIfUJ3 pM W 59Q LFm u02 YH Jaa2 Er6K SIwTBG DJ Y Zwv fSJ Qby 7f dFWd fT9z U27ws5 oU 5 MUT DJz KFN oj dXRy BaYy bTvnhh 2d V 77o FFl t4H 0R NZjV J5BJ pyIqAO WW c efd R27 nGk jm oEFH janX f1ONEc yt o INt D90 ONa nd awDR Ki2D JzAqYH GC T B0p zdB a3O ot Pq1Q VFva YNTVz2 sZ J 6ey Ig2 N7P gi lKLF 9Nzc rhuLeC eX w b6c MFE xfl JS E8Ev 9WHg Q1Brp7 RO M ACw vAn ATq GZ Hwkd HA5f bABXo6 EW H soW 6HQ Yvv jc ZgRk OWAb VA0zBf Ba W wlI V05 Z6E 2J QjOe HcZG Juq90a c5 J h9h 0rL KfI Ht l8tP rtRd qql8TZ GU g dNy SBH oNr QC sxtg zuGA wHvyNx pM m wKQ uJF Kjt Zr 6Y4H dmrC bnF52g A0 3 28a Vuz Ebp lX Zd7E JEEC 939HQt ha M sup Tcx VaZ 32 pPdb PIj2 x8Azxj YX S q8L sEQ31} \end{align} Moreover,  \begin{align} \begin{split} &\EE\left[\sup_{0\leq t\leq T}\Vert\uu(t,\cdot)\Vert_p^p +\sum_{j=1}^{D}\DFGDSFSDFGHSFGH_0^{T} \DFGDSFSDFGHSFGH_{\TT^\dd} | \nabla (|\uu_j(t,x)|^{p/2})|^2 \,dx dt\right] \\&\indeq \leq C \EE\left[ \Vert\uu_0\Vert_p^p +\DFGDSFSDFGHSFGH_0^{T}\Vert \ff(t,\cdot)\Vert_{-1,q}^p\,dt +\sum_{j=1}^D\DFGDSFSDFGHSFGH_0^{T}\DFGDSFSDFGHSFGH_{\TT^\dd}\Vert g_j(t,x)\Vert_{l^2( \mathcal{H},\mathcal{\RR})}^p\,dx\,dt \right] , \end{split} \label{FGHVBNFJHERTWERDFSDFGSDFGSDFGSDFGDFGWERTWERT39} \end{align} where $C>0$ depends on $T$ and $p$. \end{Lemma} \colb \par The following lemma is essential when passing to the limit in $\Vert |\un|^{p/2}\Vert_{H^1}$. \par \cole \begin{Lemma}\label{L02} Let $p\geq 2$. If    \begin{equation}    \uu_{n}\to\uu \text{~in~}  L^p(\Omega; L^{\infty}([0,T], L^p(\TT^d)))) \text{~as~} n\to \infty       \llabel{of qmg Sq jm8G 4wUb Q28LuA ab w I0c FWN fGn zp VzsU eHsL 9zoBLl g5 j XQX nR0 giR mC LErq lDIP YeYXdu UJ E 0Bs bkK bjp dc PLie k8NW rIjsfa pH h 4GY vMF bA6 7q yex7 sHgH G3GlW0 y1 W D35 mIo 5gE Ub Obrb knjg UQyko7 g2 y rEO fov QfA k6 UVDH Gl7G V3LvQm ra d EUO Jpu uzt BB nrme filt 1sGSf5 O0 a w2D c0h RaH Ga lEqI pfgP yNQoLH p2 L AIU p77 Fyg rj C8qB buxB kYX8NT mU v yT7 YnB gv5 K7 vq5N efB5 ye4TMu Cf m E2J F7h gqw I7 dmNx 2CqZ uLFthz Il B 1sj KA8 WGD Kc DKva bk9y p28TFP 0r g 0iA 9CB D36 c8 HLkZ nO2S 6Zoafv LX b 8go pYa 085 EM RbAb QjGt urIXlT E0 G z0t YSV Use Cj DvrQ 2bvf iIJCdf CA c WyI O7m lyc s5 Rjio IZt7 qyB7pL 9p y G8X DTz JxH s0 yhVV Ar8Z QRqsZC HH A DFT wvJ HeH OG vLJH uTfN a5j12Z kT v GqO yS8 826 D2 rj7r HDTL N7Ggmt 9M z cyg wxn j4J Je Qb7e MmwR nSuZLU 8q U NDL rdg C70 bh EPgp b7zk 5a32N1 Ib J hf8 XvG RmU Fd vIUk wPFb idJPLl NG e 1RQ RsK 2dV NP M7A3 Yhdh B1R6N5 MJ i 5S4 R49 8lw Y9 I8RH xQKL lAk8W3 Ts 7 WFU oNw I9K Wn ztPx rZLv NwZ28E YO n ouf xz6 ip9 aS WnNQ ASri wYC1sO tS q Xzo t8k 4KO z7 8LG6 GMNC ExoMh9 wl 5 vbs mnn q6H g6 WToJ un74 JxyNBX yV p vxN B0N 8wy mK 3reR eEzF xbK92xEQ36}   \end{equation} and    \begin{equation}    \nabla (|\uu_{n}(\omega,t,x)|^{p/2})      \text{~are~uniformly bounded in~}      L^2(\Omega\times[0,T], L^2(\TT^d))    ,    \llabel{ EL s 950 SNg Lmv iR C1bF HjDC ke3Sgt Ud C 4cO Nb4 EF2 4D 1VDB HlWA Tyswjy DO W ibT HqX t3a G6 mkfG JVWv 40lexP nI c y5c kRM D3o wV BdxQ m6Cv LaAgxi Jt E sSl ZFw DoY P2 nRYb CdXR z5HboV TU 8 NPg NVi WeX GV QZ7b jOy1 LRy9fa j9 n 2iE 1S0 mci 0Y D3Hg UxzL atb92M hC p ZKL JqH TSF RM n3KV kpcF LUcF0X 66 i vdq 01c Vqk oQ qu1u 2Cpi p5EV7A gM O Rcf ZjL x7L cv 9lXn 6rS8 WeK3zT LD P B61 JVW wMi KE uUZZ 4qiK 1iQ8N0 83 2 TS4 eLW 4ze Uy onzT Sofn a74RQV Ki u 9W3 kEa 3gH 8x diOh AcHs IQCsEt 0Q i 2IH w9v q9r NP lh1y 3wOR qrJcxU 4i 5 5ZH TOo GP0 zE qlB3 lkwG GRn7TO oK f GZu 5Bc zGK Fe oyIB tjNb 8xfQEK du O nJV OZh 8PU Va RonX BkIj BT9WWo r7 A 3Wf XxA 2f2 Vl XZS1 Ttsa b4n6R3 BK X 0XJ Tml kVt cW TMCs iFVy jfcrze Jk 5 MBx wR7 zzV On jlLz Uz5u LeqWjD ul 7 OnY ICG G9i Ry bTsY JXfr Rnub3p 16 J BQd 0zQ OkK ZK 6DeV gpXR ceOExL Y3 W KrX YyI e7d qM qanC CTjF W71LQ8 9m Q w1g Asw nYS Me WlHz 7ud7 xBwxF3 m8 u sa6 6yr 0nS ds Ywuq wXdD 0fRjFp eL O e0r csI uMG rS OqRE W5pl ybq3rF rk 7 YmL URU SSV YG ruD6 ksnL XBkvVS 2q 0 ljM PpI L27 Qd ZMUP baOo Lqt3bh n6 R X9h PAd QRp 9P I4fB kJ8u ILIArp Tl 4 E6j rUY wuF Xi FYaD EQ37}   \end{equation} then   \begin{align}       \liminf_{n\to \infty}         \EE\left[\DFGDSFSDFGHSFGH_0^{T} \DFGDSFSDFGHSFGH_{\TT^\dd} | \nabla (|\uu_{n}(\omega,t,x)|^{p/2})|^2 \,dx dt\right]         \geq        \EE\left[\DFGDSFSDFGHSFGH_0^{T} \DFGDSFSDFGHSFGH_{\TT^\dd} | \nabla (|\uu(\omega,t,x)|^{p/2})|^2 \,dx dt\right]     .    \label{FGHVBNFJHERTWERDFSDFGSDFGSDFGSDFGDFGWERTWERT38}   \end{align} \end{Lemma} \colb Note that in \eqref{FGHVBNFJHERTWERDFSDFGSDFGSDFGSDFGDFGWERTWERT39} and \eqref{FGHVBNFJHERTWERDFSDFGSDFGSDFGSDFGDFGWERTWERT38}, $T$ can be replaced by a stopping time by using indicator functions.  \par \startnewsection{Existence of a local strong solution}{sec4} Let $d=3$, and fix $p>3$; in particular, we allow  all constants to depend on $p$ without mention. Denote   \begin{equation*}    \mathcal{S}_n f(x)      =\sum_{|k_1|\leq n, |k_2|\leq n, |k_3|\leq n}\hat{f}(k)e^{2\pi i k\cdot x}    \comma f\in L^1(\TT^3)    \commaone n\in\mathbb{N}    ,   \end{equation*} i.e.,   \begin{equation}    \mathcal{S}_n=\mathcal{T}_{(n,n,n)}    \label{FGHVBNFJHERTWERDFSDFGSDFGSDFGSDFGDFGWERTWERT70}   \end{equation} (cf.~\eqref{FGHVBNFJHERTWERDFSDFGSDFGSDFGSDFGDFGWERTWERT40} and \eqref{FGHVBNFJHERTWERDFSDFGSDFGSDFGSDFGDFGWERTWERT70}). As pointed out in Section~\ref{sec2}, $\mathcal{S}_n$ is a continuous linear operator in all $L^q$ spaces ($1<q<\infty$). To construct a solution of \eqref{FGHVBNFJHERTWERDFSDFGSDFGSDFGSDFGDFGWERTWERT01}--\eqref{FGHVBNFJHERTWERDFSDFGSDFGSDFGSDFGDFGWERTWERT03}, we consider for $n\in\mathbb{N}$ a finite dimensional approximation  \begin{align} &d\un =\Delta \un\,dt  -\mathcal{S}_n\mathcal{P}\bigl(( \un\cdot \nabla)\un\bigr)\,dt+\mathcal{S}_n\sigma(\un)\,d\WW(t), \label{FGHVBNFJHERTWERDFSDFGSDFGSDFGSDFGDFGWERTWERT44}  \\& \nabla\cdot \un = 0 \comma  \label{FGHVBNFJHERTWERDFSDFGSDFGSDFGSDFGDFGWERTWERT45} \\& \un( 0,x)= \mathcal{S}_n\uu_0 ( x) \comma x\in\mathbb{T}^3.   \label{FGHVBNFJHERTWERDFSDFGSDFGSDFGSDFGDFGWERTWERT46} \end{align} Note that the cancellation property of the convective term remains true in $\mathcal{S}_nL^2$; namely, for $u,v \in (\mathcal{S}_nL^2(\TT^3))^3 \subset (H^1(\TT^3))^3$ with $\nabla\cdot u=\nabla\cdot v=0$, we have \begin{align} \begin{split} &\DFGDSFSDFGHSFGH_{\TT^\dd} u\cdot \mathcal{S}_n\mathcal{P}\bigl(( v\cdot \nabla)u\bigr) \,dx  = \DFGDSFSDFGHSFGH_{\TT^\dd} \bigl( u_j\mathcal{S}_n (v_i\DFGDSFSDFGHSFGI_i u_j) + u_j\mathcal{S}_nR_j  R_k \DFGDSFSDFGHSFGI_i( v_i u_k) \bigr) \,dx\\ &\indeq \indeq  = \DFGDSFSDFGHSFGH_{\TT^\dd} \bigl( u_jv_i\DFGDSFSDFGHSFGI_i u_j + u_jR_j  R_k \DFGDSFSDFGHSFGI_i ( v_i u_k) \bigr) \,dx\\ &\indeq \indeq = \DFGDSFSDFGHSFGH_{\TT^\dd} \left(-\frac{1}{2}( u_j u_j) \DFGDSFSDFGHSFGI_iv_i -
\DFGDSFSDFGHSFGI_ju_j R_k R_i ( v_i u_k) \right) \,dx =0,  \end{split}\label{can} \end{align} where $R_i$ are the Riesz transforms defined in \eqref{FGHVBNFJHERTWERDFSDFGSDFGSDFGSDFGDFGWERTWERT11}; in the third inequality, we used $R_j\DFGDSFSDFGHSFGI_{i}=\DFGDSFSDFGHSFGI_{j}R_i$. If $\DFGDSFSDFGHSFGH_{\TT^3} u=\DFGDSFSDFGHSFGH_{\TT^3} v=0$, then  \begin{align} \begin{split} \left\vert \DFGDSFSDFGHSFGH_{\TT^\dd} u\cdot\mathcal{S}_n\mathcal{P}\bigl(( v\cdot \nabla)w\bigr) \,dx \right\vert  \leq C \Vert u\Vert_6\Vert v\Vert_3\Vert \nabla w\Vert_2 \leq C \Vert \nabla u\Vert_2\Vert v\Vert_2^{1/2}\Vert \nabla v\Vert_2^{1/2}\Vert \nabla w\Vert_2. \end{split}\label{can2} \end{align}  Also, note that in $\mathcal{S}_nL^2(\TT^3)$ we have the norm equivalence \begin{equation}\llabel{VvrD b2zVpv Gg 6 zFY ojS bMB hr 4pW8 OwDN Uao2mh DT S cei 90K rsm wa BnNU sHe6 RpIq1h XF N Pm0 iVs nGk bC Jr8V megl 416tU2 nn o llO tcF UM7 c4 GC8C lasl J0N8Xf Cu R aR2 sYe fjV ri JNj1 f2ty vqJyQN X1 F YmT l5N 17t kb BTPu F471 AH0Fo7 1R E ILJ p4V sqi WT TtkA d5Rk kJH3Ri RN K ePe sR0 xqF qn QjGU IniV gLGCl2 He 7 kmq hEV 4PF dC dGpE P9nB mcvZ0p LY G idf n65 qEu Df Mz2v cq4D MzN6mB FR t QP0 yDD Fxj uZ iZPE 3Jj4 hVc2zr rc R OnF PeO P1p Zg nsHA MRK4 ETNF23 Kt f Gem 2kr 5gf 5u 8Ncu wfJC av6SvQ 2n 1 8P8 RcI kmM SD 0wrV R1PY x7kEkZ Js J 7Wb 6XI WDE 0U nqtZ PAqE ETS3Eq NN f 38D Ek6 NhX V9 c3se vM32 WACSj3 eN X uq9 GhP OPC hd 7v1T 6gqR inehWk 8w L oaa wHV vbU 49 02yO bCT6 zm2aNf 8x U wPO ilr R3v 8R cNWE k7Ev IAI8ok PA Y xPi UlZ 4mw zs Jo6r uPmY N6tylD Ee e oTm lBK mnV uB B7Hn U7qK n353Sn dt o L82 gDi fcm jL hHx3 gi0a kymhua FT z RnM ibF GU5 W5 x651 0NKi 85u8JT LY c bfO Mn0 auD 0t vNHw SAWz E3HWcY TI d 2Hh XML iGi yk AjHC nRX4 uJJlct Q3 y Loq i9j u7K j8 4EFU 49ud eA93xZ fZ C BW4 bSK pyc f6 nncm vnhK b0HjuK Wp 6 b88 pGC 3U7 km CO1e Y8jv Ebu59z mG Z sZh 93N wvJ Yb kEgD pJBj gQeQUH 9k C az6 ZGp cequiv} \frac{1}{C} \Vert f\Vert_2\leq \Vert \nabla f\Vert_2\leq C_n \Vert f\Vert_2, \end{equation} under the mean-zero condition, where the second constant depends on the dimension $n$ of the projection $\mathcal{S}_n$, and the first constant is universal. Hence, we conclude from \eqref{can2} that  \begin{align} \begin{split} \left\vert \DFGDSFSDFGHSFGH_{\TT^\dd} v\cdot\mathcal{S}_n\mathcal{P}\bigl(( v\cdot \nabla)w\bigr) \,dx \right\vert  &\leq C_n \Vert u\Vert_2 \Vert v\Vert_2 \Vert w\Vert_2 , \end{split}\label{can3} \end{align}  for $u,v,w\in \mathcal{S}_nL^2(\TT^3)$ with vanishing means. \par The next result extends Theorem~3.1 and Lemma~3.1 in~\cite{F} from an additive to the multiplicative noise. The statement  is proven by modifying the arguments in~\cite{F} and \cite{KXZ}.  For the sake of completeness and conciseness, we include a sketch of the proof. \par \cole \begin{Lemma}\label{L001} Let $T>0$,  and suppose $u_0\in L^2(\Omega; L^2)$ satisfies $\nabla\cdot u_0=0$ and $\DFGDSFSDFGHSFGH_{\TT^3} u_0=0$. Then for all $n\in\NNp$, the initial value problem \eqref{FGHVBNFJHERTWERDFSDFGSDFGSDFGSDFGDFGWERTWERT44}--\eqref{FGHVBNFJHERTWERDFSDFGSDFGSDFGSDFGDFGWERTWERT46} has a unique strong solution $\un\in L^2(\Omega; C([0,T], L^2))\cap L^2(\Omega; L^2([0,T], H^1))$. Moreover, \begin{equation}\label{FGHVBNFJHERTWERDFSDFGSDFGSDFGSDFGDFGWERTWERT63} \EE\left[\sup_{0\leq s\leq T}\Vert \un(s) \Vert_2^2   +\DFGDSFSDFGHSFGH_0^{T}  \Vert \nabla\un(s) \Vert_2^2 \,ds \right]\leq C , \end{equation} for some positive constant $C$  independent of $n$.  \end{Lemma}	 \par \colb \begin{proof}[Proof of Lemma~\ref{L001}] We fix $n\in\mathbb{N}$ and write $\mathcal{S}$ instead of $\mathcal{S}_n$ for simplicity. Let $u^{(0)}$ be the unique strong solution of the standard heat equation subject to the initial data $\mathcal{S}\uu_0$. Consider the iteration   \begin{align}    &du^{(k)}    =\Delta u^{(k)}\,dt     -(\phi_k^M)^2\mathcal{S}\mathcal{P}\bigl(( u^{(k)}\cdot \nabla)u^{(k)}\bigr)\,dt+(\phi_{k-1}^M)^2\mathcal{S}\sigma(u^{(k-1)})\,d\WW(t), \label{it-1}    \\&    \nabla\cdot u^{(k)} = 0    \comma \label{it-2}    \\&    u^{(k)}( 0,x)= \mathcal{S}\uu_0 ( x)    \comma x\in\mathbb{T}^3    ,    \label{it-3}  \end{align} where   \[    \phi_k^M=\phi^M(\Vert u^{(k)}(t,\cdot)\Vert_2)   \]  with $\phi^M\in C^{\infty}(\RR)$, $\phi^M(x)=1$ if $|x|\leq M/2$, $0\leq \phi^M(x)\leq 1$ if $M/2\leq |x|\leq M$, and $\phi^M(x)=0$ if $|x|\geq M$.  Since the noise is additive and the bilinear term  is globally Lipschitz in $u^{(k)}$,  due to the factor $(\phi_k^M)^2$, we can prove inductively as in~\cite{KXZ} that \eqref{it-1}--\eqref{it-3} has a pathwise unique strong solution in $L^2(\Omega; C([0,T], L^2))$ for all $k\in \NNp$. Furthermore, by \eqref{can}, we conclude that  \begin{equation}\label{bdd}  \EE\left[\sup_{0\leq s\leq T}\Vert u^{(k)}(s) \Vert_2^2  +\sum_{j}\DFGDSFSDFGHSFGH_0^{T}   \DFGDSFSDFGHSFGH_{\TT^3} | \nabla u^{(k)}_j(s,x)|^2 \,dx\,ds  \right]\leq C_{T,M}\big(\EE[\Vert u_0 \Vert_2^2 ]+1\big).  \end{equation} Denoting $v^{(k)}=u^{(k)}-u^{(k-1)}$ and applying the It\^{o} formula to the equation for the difference, we obtain  \begin{equation*}  \EE\left[\sup_{0\leq s\leq t}\Vert v^{(k)}(s) \Vert_2^2\right]  \leq C_{M}t\EE\left[ \sup_{0\leq s\leq t}\Vert v^{(k-1)}(s) \Vert_2^2\right],  \end{equation*} where we also used \eqref{can2}--\eqref{can3}. This implies that there exists a fixed point $u_M$ of \eqref{it-1}--\eqref{it-3} in $L^2(\Omega; C([0,t], L^2))$ for a sufficiently small $t\in(0,T]$. The square of the truncation function $\phi_k^M$ was needed for claiming that such an interval  is uniform with respect to $k$ and for previously establishing the existence of $u^{(k)}$.  Passing to the limit in \eqref{bdd} by using Lemma~\ref{L02}, we conclude that \eqref{bdd} holds for $u_M$ with $T$ being replaced by $t$. Utilizing this estimate, we can show that $u_M$ is a solution of the truncated finite-dimensional model  \begin{align}  &du  =\Delta u\,dt   -(\phi_u^M)^2\mathcal{S}\mathcal{P}\bigl(( u\cdot \nabla )u\bigr)\,dt+(\phi_u^M)^2\mathcal{S}\sigma(u)\,d\WW(t), \label{itt1}  \\&  \nabla\cdot u = 0  \comma \label{itt2}  \\&  u( 0,x)= \mathcal{S}\uu_0 ( x)  \comma x\in\mathbb{T}^3.  \label{itt3}  \end{align}  The pathwise uniqueness of the solution follows from a contraction argument similarly to above. The extension of the existence and uniqueness results from a small time interval to $[0,T]$ is obtained by using the pathwise uniqueness, and \eqref{bdd} still holds. \par    To address the non-truncated model \eqref{FGHVBNFJHERTWERDFSDFGSDFGSDFGSDFGDFGWERTWERT44}--\eqref{FGHVBNFJHERTWERDFSDFGSDFGSDFGSDFGDFGWERTWERT46}, we introduce the stopping time  \begin{equation*}  \eta_n^M =\inf\left\{t>0: \Vert u^{(n)}(s) \Vert_2 \geq \frac{M}{2}\right\}.  \end{equation*}  Clearly, \eqref{FGHVBNFJHERTWERDFSDFGSDFGSDFGSDFGDFGWERTWERT44}--\eqref{FGHVBNFJHERTWERDFSDFGSDFGSDFGSDFGDFGWERTWERT46} agrees with \eqref{itt1}--\eqref{itt3} on $[0, \eta_n^M\wedge T]$ and thus we have a pathwise unique strong solution at least up to $\eta_n^M\wedge T$. Applying the It\^{o} formula to \eqref{FGHVBNFJHERTWERDFSDFGSDFGSDFGSDFGDFGWERTWERT44} on $[0, \eta_n^M\wedge T]$ and using \eqref{can}, we obtain the energy estimate  \begin{align}\label{gron0}  \begin{split}  &\EE\left[\sup_{0\leq s\leq \eta_n^M\wedge T}\Vert u^{(n)}(s) \Vert_2^2  +\sum_{j}\DFGDSFSDFGHSFGH_0^{\eta_n^M\wedge T}   \DFGDSFSDFGHSFGH_{\TT^3} | \nabla u^{(n)}_j(s,x)|^2 \,dx\,ds  \right]\\  &\indeq \leq C\EE\left[\Vert u_0 \Vert_2^2 +\sum_{j}\DFGDSFSDFGHSFGH_0^{\eta_n^M\wedge T}    \Vert \mathcal{S}\sigma_j(\un)(s)\Vert_{\mathbb{L}^{2}}\,ds\right],  \end{split}  \end{align}  which by \eqref{FGHVBNFJHERTWERDFSDFGSDFGSDFGSDFGDFGWERTWERT24} and Lemma~\ref{L01} leads to   \begin{align}\label{gron}  \begin{split}  &\EE\left[\sup_{0\leq s\leq \eta_n^M\wedge T}\Vert u^{(n)}(s) \Vert_2^2    \right]  \leq C\left(\EE\left[  \Vert u_0 \Vert_2^2 +  \DFGDSFSDFGHSFGH_0^{\eta_n^M\wedge T}   (\Vert u^{(n)}(s) \Vert_2^2+1 ) \,ds  \right]\right)\\  &\indeq\indeq \leq C_T  \EE[\Vert u_0 \Vert_2^2 +1]  +  C\DFGDSFSDFGHSFGH_0^{T} \EE  \left[\sup_{0\leq s\leq \eta_n^M\wedge t}\Vert u^{(n)}(s) \Vert_2^2\right] \,dt.  \end{split}  \end{align}  Then Gr\"{o}nwall's lemma yields  \begin{align*}  \begin{split}  \EE\left[\sup_{0\leq s\leq  T}\Vert u^{(n)}(s\wedge\eta_n^M) \Vert_2^2    \right]=  \EE\left[\sup_{0\leq s\leq \eta_n^M\wedge T}\Vert u^{(n)}(s) \Vert_2^2    \right]   \leq C_T  \EE[\Vert u_0 \Vert_2^2 +1].  \end{split}  \end{align*}  Note that  \begin{equation*}  \EE\left[\indic_{\eta_n^M\leq  T}\Vert u^{(n)}(\eta_n^M) \Vert_2^2    \right]\leq \EE\left[\sup_{0\leq s\leq  T}\Vert u^{(n)}(s\wedge\eta_n^M) \Vert_2^2    \right].  \end{equation*}  Then, $\PP(\eta_n^M\leq  T)\leq C_T  \EE[\Vert u_0 \Vert_2^2 +1]/M^2$. Also, observe that $\eta_n^M$ is an increasing function of $M$ and that   \[  \PP(u_M=u_K,~ t\in[0, \eta_n^M])=1  \]  if $M<K$. Then defining $\eta^{\infty}=\lim_{M\to\infty} \eta_n^M$, we may uniquely define a process $u^{\infty}$ such that $u^{\infty}=u_M$ on $[0,\eta_n^M]$, and $u^{\infty}$ solves \eqref{FGHVBNFJHERTWERDFSDFGSDFGSDFGSDFGDFGWERTWERT44}--\eqref{FGHVBNFJHERTWERDFSDFGSDFGSDFGSDFGDFGWERTWERT46} on $[0,\eta^{\infty}]$. Since $\PP(\eta^{\infty}\leq  T)=\lim_{M\to\infty}\PP(\eta_n^M\leq  T)=0$ and $T$ is arbitrary, the solution is global. Note that \eqref{gron0} also implies   \begin{align*}  \begin{split}  &\EE\left[\sup_{0\leq s\leq \eta_n^M\wedge T}\Vert u^{(n)}(s) \Vert_2^2  +\sum_{j}\DFGDSFSDFGHSFGH_0^{\eta_n^M\wedge T}   \DFGDSFSDFGHSFGH_{\TT^3} | \nabla u^{(n)}_j(s,x)|^2 \,dx\,ds  \right]\\  &\indeq \leq C_T  \EE[\Vert u_0 \Vert_2^2 +1]  +  C\DFGDSFSDFGHSFGH_0^{T} \EE  \left[\sup_{0\leq s\leq \eta_n^M\wedge t}\Vert u^{(n)}(s) \Vert_2^2  +\sum_{j}\DFGDSFSDFGHSFGH_0^{\eta_n^M\wedge t} 
 \DFGDSFSDFGHSFGH_{\TT^3} | \nabla u^{(n)}_j(s,x)|^2 \,dx\,ds  \right] \,dt.  \end{split}  \end{align*}  Applying Gr\"{o}nwall's lemma and sending $M\to\infty$ in \eqref{gron}, we obtain \eqref{FGHVBNFJHERTWERDFSDFGSDFGSDFGSDFGDFGWERTWERT63}, concluding the proof. \end{proof} \par  We proceed by deriving an $L^p$ estimate of $\un$. For every $M>0$, we introduce stopping times relative to solutions $\{\un \}$ as  \begin{equation}\label{FGHVBNFJHERTWERDFSDFGSDFGSDFGSDFGDFGWERTWERT64} \tau_n^M =\inf\left\{t>0: \sup_{0\leq s\leq t}\Vert \un(s) \Vert_p  +\left(\DFGDSFSDFGHSFGH_0^{t}    \Vert \un(s) \Vert_{3p}^p \,ds\right)^{1/p}\geq M\right\}. \end{equation} Adapting the proof in~\cite[Lemma~3]{KZ}  (by contradiction and passing to the limit) and assuming $\DFGDSFSDFGHSFGH_{\TT^3} f=0$, we obtain the Gagliardo-Nirenberg inequality on $\TT^3$, \begin{equation}\llabel{pg rH r79I eQvT Idp35m wW m afR gjD vXS 7a FgmN IWmj vopqUu xF r BYm oa4 5jq kR gTBP PKLg oMLjiw IZ 2 I4F 91C 6x9 ae W7Tq 9CeM 62kef7 MU b ovx Wyx gID cL 8Xsz u2pZ TcbjaK 0f K zEy znV 0WF Yx bFOZ JYzB CXtQ4u xU 9 6Tn N0C GBh WE FZr6 0rIg w2f9x0 fW 3 kUB 4AO fct vL 5I0A NOLd w7h8zK 12 S TKy 2Zd ewo XY PZLV Vvtr aCxAJm N7 M rmI arJ tfT dd DWE9 At6m hMPCVN UO O SZY tGk Pvx ps GeRg uDvt WTHMHf 3V y r6W 3xv cpi 0z 2wfw Q1DL 1wHedT qX l yoj GIQ AdE EK v7Ta k7cA ilRfvr lm 8 2Nj Ng9 KDS vN oQiN hng2 tnBSVw d8 P 4o3 oLq rzP NH ZmkQ Itfj 61TcOQ PJ b lsB Yq3 Nul Nf rCon Z6kZ 2VbZ0p sQ A aUC iMa oRp FW fviT xmey zmc5Qs El 1 PNO Z4x otc iI nwc6 IFbp wsMeXx y8 l J4A 6OV 0qR zr St3P MbvR gOS5ob ka F U9p OdM Pdj Fz 1KRX RKDV UjveW3 d9 s hi3 jzK BTq Zk eSXq bzbo WTc5yR RM o BYQ PCa eZ2 3H Wk9x fdxJ YxHYuN MN G Y4X LVZ oPU Qx JAli DHOK ycMAcT pG H Ikt jlI V25 YY oRC7 4thS sJClD7 6y x M6B Rhg fS0 UH 4wXV F0x1 M6Ibem sT K SWl sG9 pk9 5k ZSdH U31c 5BpQeF x5 z a7h WPl LjD Yd KH1p OkMo 1Tvhxx z5 F LLu 71D UNe UX tDFC 7CZ2 473sjE Re b aYt 2sE pV9 wD J8RG UqQm boXwJn HK F Mps XBv AsX 8N YRZM wmZQ ctltsq of GN-1} \Vert |f|^{p/2}\Vert_{r}\leq C_0\Vert |f|^{p/2}\Vert_{2}^{1-\alpha}\Vert \nabla(|f|^{p/2})\Vert_{2}^{\alpha}, \end{equation} where $\alpha=3(1/2-1/r)$ and $C_0>0$ is independent of $f$.  In particular, setting $r=6$ yields  \begin{equation}\label{GN} \Vert f\Vert_{3p}^p\leq C_0 \Vert \nabla(|f|^{p/2})\Vert_{2}^2 , \end{equation} provided $f$ has $x$-mean zero. We need this inequality to prove the next lemma which asserts the boundedness of $\{\un\}$ in $L^p(\Omega; C([0,S\wedge \tau_n^M], L^p))$ for some deterministic value $S$ that is uniform with respect to $n$.  \cole \begin{Lemma}   \label{T002} Let $p>3$ and $K\geq 1$. Suppose that $\nabla\cdot \uu_0=0$, $\DFGDSFSDFGHSFGH_{\TT^3} u_0=0$, and $\Vert u_0\Vert_p\leq K~a.s$. Then there exist $M> K$ and $S>0$ such that     \begin{align}   \begin{split}   \EE\left[   \sup_{0\leq s\leq \tau_n^M\wedge S}\Vert\un(s,\cdot)\Vert_p^p   +\DFGDSFSDFGHSFGH_0^{\tau_n^M\wedge S}    \sum_{j}    \DFGDSFSDFGHSFGH_{\TT^3} | \nabla (|\un_j(s,x)|^{p/2})|^2 \,dx\,ds   \right]   \leq    C\EE\bigl[\Vert\uu_0\Vert_p^p   +1   \bigr].   \end{split}\label{FGHVBNFJHERTWERDFSDFGSDFGSDFGSDFGDFGWERTWERT65}   \end{align} Moreover,   \begin{align}   \begin{split}   \lim_{S\to 0} \sup_{n}\PP\left[   \sup_{0\leq s\leq \tau_n^M\wedge S}\Vert u^{(n)}(s,\cdot)\Vert_p^p   +\DFGDSFSDFGHSFGH_0^{\tau_n^M\wedge S}    \sum_{j}    \DFGDSFSDFGHSFGH_{\TT^3} | \nabla (|u^{(n)}_j(s,x)|^{p/2})|^2 \,dx\,ds \geq M^p   \right]   =0,   \end{split}\label{FGHVBNFJHERTWERDFSDFGSDFGSDFGSDFGDFGWERTWERT66}   \end{align} for any fixed $K> 0$ and the corresponding $M$. \end{Lemma} \colb \par \begin{proof}[Proof of Lemma~\ref{T002}] Let $M=2\sup_{n}\Vert \mathcal{S}_n u_0\Vert_p+1$ (cf.~\eqref{FGHVBNFJHERTWERDFSDFGSDFGSDFGSDFGDFGWERTWERT42}). The continuity of $\un$ implies that $\tau_n^M>0$. Since both $\mathcal{S}_n$ and $\mathcal{P}$ preserve the $x$-zero mean, we may apply Lemma~\ref{T01} to \eqref{FGHVBNFJHERTWERDFSDFGSDFGSDFGSDFGDFGWERTWERT44} on    \begin{align}   \begin{split}   &d \un_j   -\Delta \un_j \,dt    = -   \bigl(\mathcal{S}_n\mathcal{P}\bigl(\un\cdot \nabla) \un\bigr)_j \,dt   +\mathcal{S}_n\sigma_j(\un)\,d\WW_t   \commaone j=1,2,3   ,   \end{split}   \label{FGHVBNFJHERTWERDFSDFGSDFGSDFGSDFGDFGWERTWERT48}   \end{align} on $[0, S\wedge \tau_n^M]$ for a fixed $j\in\{1,2,3\}$, writing the first term on the right side of \eqref{FGHVBNFJHERTWERDFSDFGSDFGSDFGSDFGDFGWERTWERT48} as   \begin{equation}   - \sum_{i}\DFGDSFSDFGHSFGI_{i}   \mathcal{S}_n\bigl(\mathcal{P}\bigl(\un_i \un\bigr)\bigr)_j \,dt   .   \llabel{i 8wx n6I W8j c6 8ANB wz8f 4gWowk mZ P Wlw fKp M1f pd o0yT RIKH MDgTl3 BU B Wr6 vHU zFZ bq xnwK kdmJ 3lXzIw kw 7 Jku JcC kgv FZ 3lSo 0ljV Ku9Syb y4 6 zDj M6R XZI DP pHqE fkHt 9SVnVt Wd y YNw dmM m7S Pw mqhO 6FX8 tzwYaM vj z pBS NJ1 z36 89 00v2 i4y2 wQjZhw wF U jq0 UNm k8J 8d OOG3 QlDz p8AWpr uu 4 D9V Rlp VVz QQ g1ca Eqev P0sFPH cw t KI3 Z6n Y79 iQ abga 0i9m RVGbvl TA g V6P UV8 Eup PQ 6xvG bcn7 dQjV7C kw 5 7NP WUy 9Xn wF 9ele bZ8U YJDx3x CB Y CId PCE 2D8 eP 90u4 9NY9 Jxx9RI 4F e a0Q Cjs 5TL od JFph ykcz Bwoe97 Po h Tql 1LM s37 cK hsHO 5jZx qpkHtL bF D nvf Txj iyk LV hpwM qobq DM9A0f 1n 4 i5S Bc6 trq VX wgQB EgH8 lISLPL O5 2 EUv i1m yxk nL 0RBe bO2Y Ww8Jhf o1 l HlU Mie sst dW w4aS WrYv Osn5Wn 3w f wzH RHx Fg0 hK FuNV hjzX bg56HJ 9V t Uwa lOX fT8 oi FY1C sUCg CETCIv LR 0 AgT hCs 9Ta Zl 6ver 8hRt edkAUr kI n Sbc I8n yEj Zs VOSz tBbh 7WjBgf aA F t4J 6CT UCU 54 3rba vpOM yelWYW hV B RGo w5J Rh2 nM fUco BkBX UQ7UlO 5r Y fHD Mce Wou 3R oFWt baKh 70oHBZ n7 u nRp Rh3 SIp p0 Btqk 5vhX CU9BHJ Fx 7 qPx B55 a7R kO yHmS h5vw rDqt0n F7 t oPJ UGq HfY 5u At5k QLP6 ppnRjM Hk 3 HGq Z0O Bug FF xSnA SHBIEQ49}   \end{equation} For $q$ with $1/q\in[1/p, (p+1)/3p)$, we choose exponents $r$ and $l$ such that  the three conditions   \begin{equation}   \frac{1}{r} +\frac{1}{l}=\frac{1}{q}   \comma r\leq p   \commaone p<l<3p   \label{FGHVBNFJHERTWERDFSDFGSDFGSDFGSDFGDFGWERTWERT50}   \end{equation} hold.   Then,   \begin{align}   \begin{split}   &   \EE\left[\DFGDSFSDFGHSFGH_0^{S\wedge \tau_n^M }   \Vert    \mathcal{S}_n\mathcal{P}(\un_i \un)   \Vert_{q}^p \,dt   \right]   \leq    CM^p\EE   \left[\DFGDSFSDFGHSFGH_0^{S\wedge \tau_n^M}   \Vert    \un   \Vert_{l}^p    \,dt   \right]   \\&\indeq   \leq   M^p\EE\left[   \DFGDSFSDFGHSFGH_0^{S\wedge \tau_n^M }   (C_{\varepsilon}\Vert    \un   \Vert_{p}^p+   \varepsilon \Vert    \un   \Vert_{3p}^p)   \,dt
  \right].   \end{split}   \label{FGHVBNFJHERTWERDFSDFGSDFGSDFGSDFGDFGWERTWERT51}   \end{align} For the stochastic term in \eqref{FGHVBNFJHERTWERDFSDFGSDFGSDFGSDFGDFGWERTWERT48} (cf.~\eqref{FGHVBNFJHERTWERDFSDFGSDFGSDFGSDFGDFGWERTWERT39}), we also have   \begin{align}   \begin{split}   &\EE\left[   \DFGDSFSDFGHSFGH_0^{S\wedge \tau_n^M}\DFGDSFSDFGHSFGH_{\TT^\dd}   \Vert    \mathcal{S}_n\sigma(\un)   \Vert_{l^2( \mathcal{H},\mathcal{\RR}^\dd)}^p\,dx\,dt   \right]   \leq   C   \EE\DFGDSFSDFGHSFGH_0^{S\wedge \tau_n^M}   (\Vert    \un   \Vert_{p}^{p}+1)   \,dt   \end{split}   \label{FGHVBNFJHERTWERDFSDFGSDFGSDFGSDFGDFGWERTWERT52}   \end{align} due to \eqref{FGHVBNFJHERTWERDFSDFGSDFGSDFGSDFGDFGWERTWERT07} and the assumption \eqref{FGHVBNFJHERTWERDFSDFGSDFGSDFGSDFGDFGWERTWERT24}. From \eqref{FGHVBNFJHERTWERDFSDFGSDFGSDFGSDFGDFGWERTWERT39}, \eqref{FGHVBNFJHERTWERDFSDFGSDFGSDFGSDFGDFGWERTWERT51}, and \eqref{FGHVBNFJHERTWERDFSDFGSDFGSDFGSDFGDFGWERTWERT52}, we get   \begin{align}   \begin{split}   &\EE\left[\sup_{0\leq t\leq S\wedge \tau_n^M}\Vert\un(t)\Vert_p^p   +\sum_{j}\DFGDSFSDFGHSFGH_0^{S\wedge \tau_n^M} \DFGDSFSDFGHSFGH_{\TT^\dd} | \nabla (|\un_j(t,x)|^{p/2})|^2 \,dx dt\right]   \\&\indeq   \leq C   \EE\left[\Vert\uu_0\Vert_p^p+   \DFGDSFSDFGHSFGH_0^{S\wedge \tau_n^M}   (M^pC_{\varepsilon}\Vert    \un(t)   \Vert_{p}^{p}+M^p\varepsilon\Vert    \un(t)   \Vert_{3p}^{p}+1)\,dt   \right].   \end{split}\label{FGHVBNFJHERTWERDFSDFGSDFGSDFGSDFGDFGWERTWERT67}   \end{align} Note that the right hand side of \eqref{FGHVBNFJHERTWERDFSDFGSDFGSDFGSDFGDFGWERTWERT67} is finite by the definition of the stopping time, which implies finiteness of the left-hand side. Choosing $\varepsilon$ and then $S$ sufficiently small so that $CM^p\varepsilon\ll1$ and $CSM^pC_{\varepsilon}\ll1$, we arrive at \eqref{FGHVBNFJHERTWERDFSDFGSDFGSDFGSDFGDFGWERTWERT65}. The choice of $S$ depends on $M$ and the constants of embedding inequalities, but is independent of $n$. \par To obtain \eqref{FGHVBNFJHERTWERDFSDFGSDFGSDFGSDFGDFGWERTWERT66}, we  fix $j\in\{1,2,3\}$ and  utilize the trajectory It\^{o} expansion on $[0,\tau_n^M\wedge S]$,  \begin{align} \begin{split} \Vert\un_j(t)\Vert_{p}^{p} &=\Vert (\mathcal{S}_n\uu_0)_j\Vert_{p}^{p}-\frac{4(p-1)}{p}\DFGDSFSDFGHSFGH_0^t\DFGDSFSDFGHSFGH_{\TT^\dd} |\nabla (|\un_j|^{p/2})|^2\,dxdr \\&\indeq +p\DFGDSFSDFGHSFGH_0^t\DFGDSFSDFGHSFGH_{\TT^\dd} |\un_j|^{p-2}\un_j\mathcal{S}_n\bigl(\mathcal{P}\bigl(( \un\cdot \nabla)\un\bigr)_j\bigr)\,dxdr \\&\indeq + p\DFGDSFSDFGHSFGH_0^t\DFGDSFSDFGHSFGH_{\TT^\dd} |\un_j|^{p-2}\un_j \mathcal{S}_n\sigma_j(\un)\,dxd\WW_r \\&\indeq +\frac{p(p-1)}{2}\DFGDSFSDFGHSFGH_0^t\DFGDSFSDFGHSFGH_{\TT^\dd} |\un_j|^{p-2}\Vert \mathcal{S}_n\sigma_j(\un)\Vert_{l^2}^2\,dxdr. \end{split} \llabel{ 7agVfq wf g aAl eH9 DMn XQ QTAA QM8q z9trz8 6V R 2gO MMV uMg f6 tGLZ WEKq vkMEOg Uz M xgN 4Cb Q8f WY 9Tk7 3Gg9 0jy9dJ bO v ddV Zmq Jjb 5q Q5BS Ffl2 tNPRC8 6t I 0PI dLD UqX KO 1ulg XjPV lfDFkF h4 2 W0j wkk H8d xI kjy6 GDge M9mbTY tU S 4lt yAV uor 6w 7Inw Ch6G G9Km3Y oz b uVq tsX TNZ aq mwkz oKxE 9O0QBQ Xh x N5L qr6 x7S xm vRwT SBGJ Y5uo5w SN G p3h Ccf QNa fX Wjxe AFyC xUfM8c 0k K kwg psv wVe 4t FsGU IzoW FYfnQA UT 9 xcl Tfi mLC JR XFAm He7V bYOaFB Pj j eF6 xI3 CzO Vv imZ3 2pt5 uveTrh U6 y 8wj wAy IU3 G1 5HMy bdau GckOFn q6 a 5Ha R4D Ooj rN Ajdh SmhO tphQpc 9j X X2u 5rw PHz W0 32fi 2bz1 60Ka4F Dj d 1yV FSM TzS vF 1YkR zdzb YbI0qj KM N XBF tXo CZd j9 jD5A dSrN BdunlT DI a A4U jYS x6D K1 X16i 3yiQ uq4zoo Hv H qNg T2V kWG BV A4qe o8HH 70FflA qT D BKi 461 GvM gz d7Wr iqtF q24GYc yi f YkW Hv7 EI0 aq 5JKl fNDC NmWom3 Vy X JsN t4W P8y Gg AoAT OkVW Z4ODLt kz a 9Pa dGC GQ2 FC H6EQ ppks xFKMWA fY 0 Jda SYg o7h hG wHtt bb4z 5qrcdc 9C n Amx qY6 m8u Gf 7DZQ 6FBU PPiOxg sQ 0 CZl PYP Ba7 5O iV6t ZOBp fYuNcb j4 V Upb TKX ZRJ f3 6EA0 LDgA dfdOpS bg 1 ynC PUV oRW xe WQMK Smuh 3JHqX1 5A P JJX 2v0 W6l mEQ34} \end{align} As in the proof of \cite[Theorem~4.1]{KXZ}, we apply the Poincar\'{e}-type inequality, obtaining \begin{align*} \begin{split} &\DFGDSFSDFGHSFGH_0^t\DFGDSFSDFGHSFGH_{\TT^\dd} |\un_j|^{p-2}\un_j\mathcal{S}_n\bigl(\mathcal{P}\bigl(( \un\cdot \nabla)\un\bigr)_j\bigr)\,dxdr \\&\indeq \leq \varepsilon\DFGDSFSDFGHSFGH_0^t\DFGDSFSDFGHSFGH_{\TT^\dd}|\nabla (|\un_j|^{p/2})|^2 \,dxdr+ C_{M,p,\varepsilon}\DFGDSFSDFGHSFGH_0^t \Vert\un\Vert_{p}^{p}\,dr \end{split} \end{align*} for $q$ in \eqref{FGHVBNFJHERTWERDFSDFGSDFGSDFGSDFGDFGWERTWERT50} and for an arbitrarily small $\varepsilon$. Under the assumption \eqref{FGHVBNFJHERTWERDFSDFGSDFGSDFGSDFGDFGWERTWERT24}, \begin{align*} \begin{split} &\DFGDSFSDFGHSFGH_0^t\DFGDSFSDFGHSFGH_{\TT^\dd} |\un_j|^{p-2}\Vert \mathcal{S}_n\sigma_j(\un)\Vert_{l^2}^2\,dxdr \leq C \DFGDSFSDFGHSFGH_0^t \bigl(\Vert\un\Vert_{p}^{p}+1\bigr)\,dr. \end{split} \end{align*} Choosing a sufficiently small $\varepsilon$ and summing in $j$, we arrive at  \begin{align*} \begin{split} &\Vert\un(t)\Vert_{p}^{p}+\sum_j\DFGDSFSDFGHSFGH_0^t\DFGDSFSDFGHSFGH_{\TT^\dd} |\nabla (|\un_j|^{p/2})|\,dxdr -\Vert \mathcal{S}_n\uu_0\Vert_{p}^{p} \\&\indeq \leq C_{p,M}\DFGDSFSDFGHSFGH_0^t \bigl(\Vert\un\Vert_{p}^{p}+1\bigr)\,dr + C_p\sum_j\DFGDSFSDFGHSFGH_0^t\DFGDSFSDFGHSFGH_{\TT^\dd} |\un_j|^{p-2}\un_j \mathcal{S}_n\sigma_j(\un)\,dxd\WW_r \quad{} \PP\mbox{-a.s.} \end{split} \end{align*} Recall that $M$ was set to satisfy $M\geq 2\Vert \mathcal{S}_n u_0\Vert_p$ and that $p>3$. Hence, the inequality above implies   \begin{align} \begin{split} &\PP\left[ \sup_{0\leq t\leq \tau_n^M\wedge S}\Vert u^{(n)}(t,\cdot)\Vert_p^p +\sum_{j} \DFGDSFSDFGHSFGH_0^{\tau_n^M\wedge S}  \DFGDSFSDFGHSFGH_{\TT^3} | \nabla (|u^{(n)}_j(t,x)|^{p/2})|^2 \,dx dt  \geq  M^p \right] \\&\indeq \leq  \PP\left[ \DFGDSFSDFGHSFGH_0^{\tau_n^M\wedge S} \bigl(\Vert\un\Vert_{p}^{p}+1\bigr)\,dr \geq \frac{M^p}{3C_{p,M}} \right] \\&\indeq\indeq +
\PP\left[ \sum_j\sup_{0\leq t\leq \tau_n^M\wedge S} \left\vert \DFGDSFSDFGHSFGH_0^t\DFGDSFSDFGHSFGH_{\TT^\dd} |\un_j|^{p-2}\un_j \mathcal{S}_n\sigma_j(\un)\,dxd\WW_r\right\vert  \geq \frac{M^p}{3C_{p}} \right]. \end{split}\label{Cheb} \end{align} Clearly, \begin{equation*} \EE\left[ \DFGDSFSDFGHSFGH_0^{\tau_n^M\wedge S} \bigl(\Vert\un\Vert_{p}^{p}+1\bigr)\,dr \right]\leq (M^p+1)S. \end{equation*} Also, by the BDG inequality and Minkowski's integral inequality,   \begin{align*} \begin{split} &\EE\left[ \sup_{0\leq t\leq \tau_n^M\wedge S}\left\vert \DFGDSFSDFGHSFGH_0^t\DFGDSFSDFGHSFGH_{\TT^\dd} |\un_j|^{p-2}\un_j \mathcal{S}_n\sigma_j(\un)\,dxd\WW_r\right\vert \right] \\&\indeq \leq  C\EE\left[\left( \DFGDSFSDFGHSFGH_0^{\tau_n^M\wedge S}\left(\DFGDSFSDFGHSFGH_{\TT^\dd} |\un_j|^{p-1}\Vert \mathcal{S}_n\sigma_j(\un)\Vert_{l^2}\,dx\right)^2\,dr \right)^{1/2}\right] \\&\indeq \leq C\EE\left[ \sup_{0\leq t\leq\tau_n^M\wedge S}\Vert\un(t, \cdot)\Vert_{p}^{p-1} \left( \DFGDSFSDFGHSFGH_0^{\tau_n^M\wedge S} \DFGDSFSDFGHSFGH_{\TT^\dd}\Vert \mathcal{S}_n\sigma_j(\un)\Vert_{l^2}^p\,dx\,dr \right)^{1/p} \right] \\&\indeq \leq CM^{p-1}(M^p+1)^{1/p}S^{1/p} . \end{split} \end{align*} Using \eqref{Cheb} and Chebyshev's inequality, we get   \begin{align*} \begin{split} &\PP\left[ \sup_{0\leq t\leq \tau_n^M\wedge S}\Vert u^{(n)}(t,\cdot)\Vert_p^p +\sum_{j} \DFGDSFSDFGHSFGH_0^{\tau_n^M\wedge S}  \DFGDSFSDFGHSFGH_{\TT^3} | \nabla (|u^{(n)}_j(t,x)|^{p/2})|^2 \,dx dt  \geq M^p \right]\leq C_M (S+S^{1/p}). \end{split} \end{align*} Then, for the fixed $K$ and the associated $M$,   \begin{align}   \begin{split}   \lim_{S\to 0} \sup_{n}\PP\left[   \sup_{0\leq t\leq \tau_n^M\wedge S}\Vert u^{(n)}(t,\cdot)\Vert_p^p   +\sum_{j} \DFGDSFSDFGHSFGH_0^{\tau_n^M\wedge S}    \DFGDSFSDFGHSFGH_{\TT^3} | \nabla (|u^{(n)}_j(t,x)|^{p/2})|^2 \,dx dt    \geq M^p   \right]   =0,   \end{split}    \llabel{0 llC8 hlss 1NLWaN hR B Aqf Iuz kx2 sp 01oD rYsR ywFrNb z1 h Gpq 99F wUz lf cQkT sbCv GIIgmf Hh T rM1 ItD gCM zY ttQR jzFx XIgI7F MA p 1kl lwJ sGo dX AT2P goIp 9VonFk wZ V Qif q9C lAQ 4Y BwFR 4nCy RAg84M LJ u nx8 uKT F3F zl GEQt l32y 174wLX Zm 6 2xX 5xG oaC Hv gZFE myDI zj3q10 RZ r ssw ByA 2Wl OA DDDQ Vin8 PTFLGm wi 6 pgR ZQ6 A5T Ll mnFV tNiJ bnUkLy vq 9 zSB P6e JJq 7P 6RFa im6K XPWaxm 6W 7 fM8 3uK D6k Nj 7vhg 4ppZ 4ObMaS aP H 0oq xAB G8v qr qT6Q iRGH BCCN1Z bl T Y4z q8l FqL Ck ghxD UuZw 7MXCD4 ps Z cEX 9Rl Cwf 0C CG8b gFti Uv3mQe LW J oyF kv6 hcS nM mKbi QukL FpYAqo 5F j f9R RRt qS6 XW VoIY VDMl a5c7cW KJ L Uqc vti IOe VC U7xJ dC5W 5bk3fQ by Z jtU Dme gbg I1 79dl U3u3 cvWoAI ow b EZ0 xP2 FBM Sw azV1 XfzV i97mmy 5s T JK0 hz9 O6p Da Gcty tmHT DYxTUB AL N vQe fRQ uF2 Oy okVs LJwd qgDhTT Je R 7Cu Pcz NLV j1 HKml 8mwL Fr8Gz6 6n 4 uA9 YTt 9oi JG clm0 EckA 9zkElO B9 J s7G fwh qyg lc 2RQ9 d52a YQvC8A rK 7 aCL mEN PYd 27 XImG C6L9 gOfyL0 5H M tgR 65l BCs WG wFKG BIQi IRBiT9 5N 7 8wn cbk 7EF ei BRB2 16Si HoHJSk Ng x qup JmZ 1px Eb Wcwi JX5N fiYPGD 6u W sXT P94 uaF VD ZuhJ H2d0 PLOY24 3x M K4EQ73}   \end{align} concluding the proof. \end{proof}    \par We next show that $\{\un\}$ is Cauchy in $L^p_{\omega}L^{\infty}_t L_x^p\cap L^p_{\omega}L^{p}_t L_x^{3p}$ within a prescribed stopping time.  \par \cole \begin{Lemma} \label{L04} Let $p>3$ and $K\geq1$.  Assume that $\nabla\cdot \uu_0=0$, $\DFGDSFSDFGHSFGH_{\TT^3} u_0=0$, and $\Vert u_0\Vert_p\leq K~a.s$. Then there exist $M>K$ and a positive constant $S$ depending on $M$ such that   \begin{align}   \begin{split}   \lim_{m\to\infty} \sup_{n>m}\EE\left[   \sup_{0\leq t\leq \tau_n^M\wedge \tau_m^M\wedge S}\Vert u^{(n)}(t)-u^{(m)}(t)\Vert_p^p   +\DFGDSFSDFGHSFGH_0^{\tau_n^M\wedge \tau_m^M\wedge S}    \sum_{j}    \DFGDSFSDFGHSFGH_{\TT^3} | \nabla (|u^{(n)}_j-u^{(m)}_j|^{p/2})|^2 \,dx\,dt   \right]   =0. \label{FGHVBNFJHERTWERDFSDFGSDFGSDFGSDFGDFGWERTWERT74}   \end{split}   \end{align} \end{Lemma} \colb \par \begin{proof}[Proof of Lemma~\ref{L04}] We fix $M>0$ as in the proof of Lemma~\ref{T002}. Then for every $S>0$, $\tau_{nm}=\tau_n^M\wedge \tau_m^M\wedge S$ is a  positive stopping time almost surely. Denote $\wnm=\un-\um$ and $\snm=\mathcal{S}_n-\mathcal{S}_m$. Clearly on $[0,\tau_{nm}]$, the difference $\wnm$ satisfies   \begin{align}    \begin{split}    &d\wnm    -\Delta \wnm\,dt     =    \mathcal{S}_m\mathcal{P}\bigl(( \um\cdot \nabla)\um\bigr)-\mathcal{S}_n\mathcal{P}\bigl(( \un\cdot \nabla)\un\bigr)\,dt   \\&\indeq\indeq\indeq\indeq\indeq\indeq     +\bigl(            \mathcal{S}_n\sigma(\un)-\mathcal{S}_m\sigma(\um)      \bigr)\,d\WW_t,     \\&    \nabla\cdot \wnm = 0,    \\&    \wnm( 0)= \snm\uu_0 ( x)\quad{}\mbox{a.s.}    \end{split}    \llabel{7 VP6 FTy T3 5zpL xRC6 tN89as 3k u 8eG rdM KWo MI U946 FBjk sOTe0U xZ D 4av bTw 5mQ 3R y9Af JFjP gvLFKz 0o l fZd j3O 07E av pWfb M3rB GSyOiu xp I 4o8 2JJ 42X 1G Iux8 QFh3 PhRtY9 vj i SL6 x76 W9y 2Z z3YA SGRM p7kDhr gm a 8fW GG0 qKL sO 5oQr 42t1 jP1crM 2f C lRb ETd qra 5l VG1l Kitb XqbdPK ca U V0l v4L alo 8V TXcl aUqh 5GWCzA nR n lNN cmw aF8 Er bwX3 2rji Hleb4g XS j LRO JgG 2yb 8O CAxN 4uy4 RsLQjD 7U 7 enw cYC nZx iK dju7 4vpj BKKjRR l3 6 kXX zvn X2J rD 8aPD UWGs tgb8CT WY n HRs 6y6 JCp 8L x1jz CI1m tG26y5 zr J 1nF hX6 7wC zq F8uZ QIS0 dnYxPe XD y jBz 1aY wzD Xa xaMI ZzJ3 C3QRra hp w 8sW Lxr AsS qZ P5Wv v1QF 7JPAVQ wu W u69 YLw NHU PJ 0wjs 7RSi VaPrEG gx Y aVm Sk3 Yo1 wL n0q0 PVeX rzoCIH 7v x q5z tOm q6m p4 drAp dzhw SOlRPD ps C lr8 FoZ UG7 vD UYhb ScJ6 gJb8Q8 em G 2JG 9Oj a83 ow Ywjo zLa3 DB500s iG j EHo lPu qe4 p7 T1kQ JmU6 cHnOo2 9o r oOz Ta3 j31 n8 mDL7 CIvC pKZUs0 jV r b7v HIH 7NT tY Y7JK vVdG LhA1ON CW o QW1 fvj mlH 7l SlIm 8T1Q SdUWhT iM P KDZ mm4 V7o fR W1dn lqg0 Ah1QRj dt K ZVz EBN E1e Xi RRSL LQPE SEDeXb iM M Ffx C5F I1z vi yNsY HPsG xfGiIu hD P Di0 OIH uBT TH OCHy CTkA BxuEQ54}   \end{align} We rewrite the first equation as   \begin{align}    \begin{split}    &d{w}_j^{(n,m)}    -\Delta {w}_j^{(n,m)}\,dt      = \sum_{i}\DFGDSFSDFGHSFGI_{i}    f_{ij} \,dt     + g_j      \,d\WW_t    \comma j=1,2,3    \,,   \end{split}    \llabel{Cjg OZ s 965 wfe Fwv fR pNLL T3Ev gKgkO9 jy y vot RRl pDT dn 9H5Z nqwW r4OUkI lx t sk0 RZd ODn so Yid6 ctgw wQrxQk 1S 8 ajp PiZ Jlp 5p IAT1 t482 KxtvQ6 D1 T VzQ 7F3 xoz 6H w2ph WDlC Jg7VcE ix 6 XFI dlO lcN bg ODKp 86tC HVGrzE cV n Bk9 9sq 5XG d1 DNFA Negg JYjfBW jA b JSc hyE uVl EN awP0 DWoZ WKuP4I Pt v Zbm nRL 047 2K 3bBQ IH5S pPxtXy 5N J joW ceA 7Fe T7 Iwpi vQdq LaeZE0 Qf i MW1 Koz kdU tR sGH6 ryob MpDbfL t0 Z 2FA XbR 3QQ wu Iizg ZFQ4 Gh4lY5 pt 9 RMT ieq BIk dX I979 BGU2 yYtJSa nO M sDL Wyd CQf ol xJWb bIdb EggZLB Kb F mKX oRM cUy M8 NlGn WyuE RUtbAs 4Z R PHd IWt lbJ Rt Qwod dmlZ hI3I8A 9K 8 Syf lGz cVj Cq GkZn aZrx HNxIcM ae G QdX XxG HFi 6A eYBA lo4Q 9HZIjJ jt O hl4 VLm Vvc ph mMES M8lt xHQQUH jJ h Yyf 5Nd c0i 8m HOTN S7yx 5hNrJC yJ 1 ZFj 4Qe Iom 7w czw9 8Bn6 SxxoqP tn X p4F yiE b2M Cy j2AH aB8F ejdIRh qQ V fR8 rEt z0m q5 4IZt bSlX dBmEvC uv A f5b YxZ 3LE sJ YEX8 eNmo tV2IHl hJ E 70c s45 KVw JR 1riF MPEs P3srHa 8p q wVN AHu soh YI rkNw ekfR bDVLm2 ax u 6ca KkT Xrg Bg nQhU A1z8 X6Mtqv ks U fAF VLg Tmq Pn trgI ggjf JfMGfC uB y BS7 njW fYR Nh pHsj FCzM 4f6cRD gj P Zkb SUH QBn zQ wEEQ55}   \end{align} where    \begin{align}    \begin{split}    f_{ij} & = \mathcal{S}_m(\mathcal{P}\um_i\um)_j-\mathcal{S}_n(\mathcal{P}\un_i\un)_j    \\&
   =    -\mathcal{S}_m(\mathcal{P}\um_i\wnm)_j-\mathcal{S}_m(\mathcal{P}\wnm_i\un)_j-  \snm(\mathcal{P}\un_i\un)_j    \\ &    =     f_{ij}^{(1)}+f_{ij}^{(2)}+f_{ij}^{(3)}    \end{split}    \llabel{nS 9CxS fn00xm Af w lTv 4HI ZIZ Ay XIs4 hPOP jQ3v93 iT L 0Jt NJ8 baB BW cY18 vifU iGKvSQ 4g E kZ1 0yS 5lX Cw I4oX 2gPB isFp7T jK u pgV n5o i4u xK t2QP 4kbr ChS5Zn uW X Wep 0mO jW1 r2 IaXv Hle8 ksF2XQ 52 9 gTL s3u vAO f6 4HOV Iqrb LoG5I2 n0 X skv cKY FIV 8y P9tf MEVP R7F0ip Da q wgQ xro 5Et IW r3tE aSs5 CjzfRR AL g vmy MhI ztV Kj StP7 44RC 0TTPQp n8 g LVt zpL zEQ e2 Rck9 WuM7 XHGA7O 7K G wfm ZHL hJR NU DEQe Brqf KIt0Y4 RW 4 9GK EHY ptg LH 4F8r ZfYC vcf1pO yj k 8iT ES0 ujR vF pipc wIvL DgikPu qq k 9RE dH9 YjR UM kr9b yFJK LBex0S gD J 2gB IeC X2C UZ yyRt GNY3 eGOaDp 3m w QyV 1Aj tGL gS C1dD pQCB cocMSM 4j q bSW bvx 6aS nu MtD0 5qpw NDlW0t Z1 c bjz wU5 bUd CG AghC w0nI CDFKHR kp h btA 6nY ld6 c5 TSkD q3Qx o2jhDx Qb m b8n Pq3 zNZ QF JJyu Vm1C 6rzRDC B1 m eQy 4Tt Yr5 jQ VWoO fbrY Q6qakZ ep H b2b 5w4 KN3 mE HtQK AXsI ycbaky ID 9 O8Y CmR lEW 7f GISs 6xaz bM6PSB N2 B jtb 65z z2N uY o4kU lpIq JVBC4D zu Z ZN6 Zkz 0oo mm nswe bstF mlxkKE QE L 6bs oYz xx0 8I Q5Ma 7Inf dXLQ9j eH S Tmi gtt k4v P7 778H p1o6 7atRbf cr S 2CW zwQ 9j0 Rj r0VL 9vlv kkk6J9 bM 1 Xgi Yla y8Z Eq 39Z5 3jRn Xh5mKP Pa 5 tFw 7EEQ75}   \end{align} and    \begin{align}    g_{j} &= \mathcal{S}_n\bigl(\sigma_j(\un )-\sigma_j(\um )\bigr)+\snm\sigma_j(\um )=g_j^{(1)}+g_{j}^{(2)}    .    \llabel{0 nE7 Cu FIoV lFxg uxB1hq lH e OLd b7R Kfl 0S KJiY ekpv RSYnNF f7 U VOW Bvw pN9 mt gGwh 2NJC Y53IdJ XP p YAZ 1B1 AgS xn 61oQ Vtg7 W7QcPC 42 e cSA 5jG 4K5 H1 tQs6 TNph OKTBId Gk F SGm V0k zAx av Qzje XGbi Sjg3kY Z5 L xzF 3JN Hkn rm y4sm J70w hEtBeX kS T WEu jcA uS0 Nk Hloa 7wYg Ma5j8g 4g i 7WZ 77D s5M ZZ MtN5 iJEa CfHJ0s D6 z VuX 06B P99 Fg a9Gg YMv6 YFVOBE Ry 3 Xw2 SBY ZDx ix xWHr rlxj KA3fok Ph 9 Y75 8fG XEh gb Bw82 C4JC StUeoz Jf I uGj Ppw p7U xC E5ah G5EG JF3nRL M8 C Qc0 0Tc mXI SI yZNJ WKMI zkF5u1 nv D 8GW YqB t2l Nx dvzb Xj00 EEpUTc w3 z vyf ab6 yQo Rj HWRF JzPB uZ61G8 w0 S Abz pNL IVj WH kWfj ylXj 6VZvjs Tw O 3Uz Bos Q7e rX yGsd vcKr YzZGQe AM 1 u1T Nky bHc U7 1Kmp yaht wKEj7O u0 A 7ep b7v 4Fd qS AD7c 02cG vsiW44 4p F eh8 Odj wM7 ol sSQo eyZX ota8wX r6 N SG2 sFo GBe l3 PvMo Ggam q3Ykaa tL i dTQ 84L YKF fA F15v lZae TTxvru 2x l M2g FBb V80 UJ Qvke bsTq FRfmCS Ve 3 4YV HOu Kok FX YI2M TZj8 BZX0Eu D1 d Imo cM9 3Nj ZP lPHq Ell4 z66IvF 3T O Mb7 xuV RYj lV EBGe PNUg LqSd4O YN e Xud aDQ 6Bj KU rIpc r5n8 QTNztB ho 3 LC3 rc3 0it 5C N2Tm N88X YeTdqT LP l S97 uLM w0N As MphO uPNi sXNIlW EQ76}   \end{align} Now, choosing the exponents as in \eqref{FGHVBNFJHERTWERDFSDFGSDFGSDFGSDFGDFGWERTWERT50}--\eqref{FGHVBNFJHERTWERDFSDFGSDFGSDFGSDFGDFGWERTWERT51}, we obtain  \begin{align} \begin{split} &\EE\left[\DFGDSFSDFGHSFGH_0^{\tau_{nm}}(\Vert  f_{ij}^{(1)} \Vert_{q}^p+ \Vert  f_{ij}^{(2)} \Vert_{q}^p )\,dt \right] \leq  \EE\left[\DFGDSFSDFGHSFGH_0^{\tau_{nm}} \Vert  \wnm  \Vert_{l}^p  (\Vert  \um_i \Vert_{p}^p  +\Vert  \un_i \Vert_{p}^p)  \,dt \right] \\ &\indeq \leq  C_M\EE\left[ \DFGDSFSDFGHSFGH_0^{\tau_{nm}} (\varepsilon\Vert  \wnm  \Vert_{3p}^p+C_{\varepsilon} \Vert  \wnm  \Vert_{p}^p ) \,dt \right].    \end{split} \label{FGHVBNFJHERTWERDFSDFGSDFGSDFGSDFGDFGWERTWERT68} \end{align} Recall that $1/q\in[1/p, (p+1)/3p)$. For $0<\delta\ll1$ and $0<\theta\ll1$ such that $q+\delta<4p/3$  and  $1/q=\theta/(1+\delta)+(1-\theta)/(q+\delta)$, we have  \def\minn#1{#1}  \begin{align*} \begin{split} &\EE\left[\DFGDSFSDFGHSFGH_0^{\tau_{nm}}\Vert  f_{ij}^{(3)} \Vert_{q}^p \,dt \right] \leq C \EE\left[\DFGDSFSDFGHSFGH_0^{\tau_{nm}} \Vert  \snm\un_i\un \Vert_{1+\delta}^{\theta p} \Vert  \snm\un_i\un \Vert_{q+\delta}^{(1-\theta)p}   \,dt \right] \\ &\indeq  \leq  \frac{C}{\minn{m}^{\theta\alpha p}\wedge \minn{n}^{\theta\alpha p}} \EE\left[\DFGDSFSDFGHSFGH_0^{\tau_{nm}} \Vert  \nabla(\un_i\un) \Vert_{1+\delta}^{\theta p}  \Vert  \un_i\un \Vert_{q+\delta}^{(1-\theta)p}  \,dt \right] \\ &\indeq  \leq  \frac{C}{\minn{m}^{\theta\alpha p}\wedge \minn{n}^{\theta\alpha p}} \EE\left[\DFGDSFSDFGHSFGH_0^{\tau_{nm}} \Vert  \nabla\un \Vert_{2}^{\theta p}  \Vert  \un \Vert_{\tilde{l}}^{\theta p}  \Vert  \un \Vert_{2(q+\delta)}^{2(1-\theta) p}  \,dt \right] , \end{split} \end{align*}  where $1/(1+\delta)=1/2+1/\tilde{l}$ and $\alpha$ is determined by Lemma~\ref{T02}.  Note that $\delta$ can be sufficiently small so that $\theta p<2$ and $2<\tilde{l}<p$. By \eqref{FGHVBNFJHERTWERDFSDFGSDFGSDFGSDFGDFGWERTWERT63} and \eqref{FGHVBNFJHERTWERDFSDFGSDFGSDFGSDFGDFGWERTWERT65}, we conclude   \begin{align}   \begin{split}    &\EE\left[\DFGDSFSDFGHSFGH_0^{\tau_{nm}}\Vert     f_{ij}^{(3)}    \Vert_{q}^p \,dt    \right]    \leq    \frac{C_M}{\minn{m}^{\theta\alpha p}\wedge \minn{n}^{\theta\alpha p}}    \EE\left[\DFGDSFSDFGHSFGH_0^{\tau_{nm}}    (\Vert     \nabla\un    \Vert_{2}^{2}     +    \Vert     \un    \Vert_{3p}^{ p})    \,dt    \right]    \leq     \frac{C_M}{\minn{m}^{\theta\alpha p}\wedge \minn{n}^{\theta\alpha p}}    .   \end{split}   \llabel{fX B Gc2 hxy kg5 0Q TN75 t5JN wZR3NH 1M n VRZ j2P rUY ve HPEl jGaT Ix4sCF zK B 0qp 3Pl eK6 8p 85w4 4l5z Zl07br v6 1 Kki AuT SA5 dk wYS3 F3YF 3e1xKE JW o AvV OZV bwN Yg F7CK bSi9 2R0rlW h2 a khC oEp pr6 O2 PZJD ZN8Z ZD4IhH PT M vSD TgO y1l Z0 Y86n 9aMg kWdeuO Zj O i2F g3z iYa SR Cjlz XdQK bcnb5p KT q rJp 1P6 oGy xc 9vZZ RZeF r5TsSZ zG l 7HW uIG M0y Re YDw3 lMux gAdF6d pp 8 ZVR cl7 uqH 8O BMbz L6dK BflWCW dl V hyc V5n Epv 2J SkD0 ccMp oIR38Q pe Z j9j 0Zo Pmq XR TxBs 8w9Q 5epR3t N5 j bvb rbS K7U 4W 4PJ0 ovnB 0opRpC YN P so8 34P wtS Rq vir4 DRqu jaJq32 QU T G1P gbp 6nJ M2 CUnE NdJC r3ZGBH Eg B tds Td8 4gM 22 gKBN 7Qnm RtJgKU IG E eKx 64y AGK Ge zeJN mpeQ kLR389 HH 9 fXL BcE 6T4 Gj VZLI dLQI iQtkBk 9G 9 FzH WIG m91 M7 SW02 9tzN UX3HLr OU t vG5 QZn Dqy M6 ESTx foUV ylEQ99 nT C SkH A8s fxr ON eFp9 QLDn hLBPib iu j cJc 8Qz Z2K zD oDHg 252c lhDcaQ continuous n xG9 aJl jFq mA DsfD FA0w DO3CZr Q1 a 2IG tqK bjc iq zRSd 0fjS JA1rsi e9 i qOr 5xg Vlj y6 afNu ooOy IVlT21 vJ W fKU deL bcq 1M wF9N R9xQ np6Tqg El S k50 p43 Hsd Cl 7VKk Zd12 Ijx43v I7 2 QyQ vUm 77B V2 3a6W h6IX dP9n67 St l Zll bRi DyG NEQ57}   \end{align} Also, by Lemma~\ref{L01} and the assumption \eqref{FGHVBNFJHERTWERDFSDFGSDFGSDFGSDFGDFGWERTWERT25}, we have   \begin{align}   \begin{split}   \EE\left[   \DFGDSFSDFGHSFGH_0^{\tau_{nm}}\DFGDSFSDFGHSFGH_{\TT^\dd}\Vert {g_1}\Vert_{l^2}^p\,dx\,dt   \right]   &\leq   C \,S\,   \EE\left[\sup_{s\in[0,\tau_{nm}]} \Vert \wnm\Vert_{p}^{p}    \right]. \end{split} \label{FGHVBNFJHERTWERDFSDFGSDFGSDFGSDFGDFGWERTWERT78} \end{align} Applying Lemmas~\ref{T02} and~\ref{L01}, together with Fubini's theorem, we obtain  \begin{align} \begin{split}   &\EE\left[   \DFGDSFSDFGHSFGH_0^{\tau_{nm}}\DFGDSFSDFGHSFGH_{\TT^\dd}\Vert {g_2}\Vert_{l^2}^p\,dx\,dt   \right]   \leq C\EE\left[   \DFGDSFSDFGHSFGH_0^{\tau_{nm}}   \left(\DFGDSFSDFGHSFGH_{\TT^\dd}\Vert {g_2}\Vert_{l^2}^2\,dx   \right)^{1/2}   \left(\DFGDSFSDFGHSFGH_{\TT^\dd}\Vert {g_2}\Vert_{l^2}^{2(p-1)}\,dx   \right)^{1/2}   \,dt   \right]    \\ & \indeq\indeq    \leq\frac{C}{\minn{m}\wedge \minn{n}}    \EE\left[   \DFGDSFSDFGHSFGH_0^{\tau_{nm}}   \Vert \nabla(\sigma(\um ))\Vert_{\mathbb{L}^2}   \Vert \sigma(\um )\Vert_{\mathbb{L}^{2(p-1)}}^{p-1}   \,dt   \right]    \\ & \indeq\indeq   \leq\frac{C}{\minn{m}\wedge \minn{n}}   \EE\left[   \DFGDSFSDFGHSFGH_0^{\tau_{nm}}   \Vert \um \Vert_{L^2}   \Vert \sigma(\um )\Vert_{\mathbb{L}^{2(p-1)}}^{p-1}   \,dt   \right]   \\ & \indeq\indeq   \leq\frac{C_{M,S}}{\minn{m}\wedge \minn{n}}   ,   \end{split}   \label{FGHVBNFJHERTWERDFSDFGSDFGSDFGSDFGDFGWERTWERT5-8}   \end{align} where the third inequality follows from the assumptions \eqref{FGHVBNFJHERTWERDFSDFGSDFGSDFGSDFGDFGWERTWERT24} and \eqref{FGHVBNFJHERTWERDFSDFGSDFGSDFGSDFGDFGWERTWERT22}. If $\varepsilon C_M\ll1$ in \eqref{FGHVBNFJHERTWERDFSDFGSDFGSDFGSDFGDFGWERTWERT68} and if $S>0$ is sufficiently small, then \eqref{FGHVBNFJHERTWERDFSDFGSDFGSDFGSDFGDFGWERTWERT68}--\eqref{FGHVBNFJHERTWERDFSDFGSDFGSDFGSDFGDFGWERTWERT5-8} lead to   \begin{align}   \begin{split}   & \EE\left[\sup_{0\leq t\leq \tau_{nm}}\Vert\wnm(t,\cdot)\Vert_p^p   +\DFGDSFSDFGHSFGH_0^{\tau_{nm}}    \sum_{j,k}    \DFGDSFSDFGHSFGH_{\TT^3} | \DFGDSFSDFGHSFGI_{k} (|\wnm_j|^{p/2})|^2 \,dx dt   \right]   \\&\indeq   \leq    C\EE\left[\Vert \snm\uu_0 ( x)\Vert_p^p   \right]   +   \frac{C_M}{\minn{m}^{\theta\alpha p}\wedge \minn{n}^{\theta\alpha p}}   +\frac{C_{M,S}}{\minn{m}\wedge \minn{n}},   \end{split}    \llabel{r 0g9S 4AHA Vga0Xo fk X FZw gGt sW2 J4 92NC 7FAd 8AVzIE 0S w EaN EI8 v9e le 8EfN Yg3u WVH3JM gi 7 vGf 4N0 akx mB AIjp x4dX lxQRGJ Ze r TMz BxY 9JA tm ZCjH 9064 Q4uzKx gm p CQg 8x0 6NY x0 2vkn EtYX 5O2vgP 3g c spG swF qhX 3a pbPW sf1Y OzHivD ia 1 eOD MIL TC2 mP ojef mEVB 9hWwMa Td I Gjm 9Pd pHV WG V4hX kfK5 Rtci05 ek z j0L 8Tm e2J PX pDI8 Ebcq V4Fdxv rH I eP8 CdO RJp Ti MVEb AunS GsUMWP ts 4 uBv 2QS iXI b7 B8zo 7bp9 voEwNR uX J 4Zx uRZ Yhc 1h 339T HRXV Fw5XVW 8g a B39 mFS v6M ze znkb LHrt Z73hUu aq L vPh gTl NnV po 1Zgg mnRA qM3X31 OR Y Sj8 Rkt S8V GO jrz1 iblt 3uOuEs 8Q 3 xJ1 cA2 NKo F8 o6U3 mW2H q5y6jp os x Jgw WZ4 Exd 79 Jvlc wauo RDCYZz mp a bV0 9jg ume bz cbug patf 9yU9iB Ey v 3Uh S79 XdI mP NEhN 64Rs 9iHQ84 7j X UCA ufF msn Uu dD4S g3FM LMWbcB Ys 4 JFy Yzl rSf nk xPjO Hhsq lbV5eB ld 5 H6A sVt rHg CN Yn5a C028 FEqoWa KS s 9uu 8xH rbn 1e RIp7 sL8J rFQJat og Z c54 yHZ vPx Pk nqRq Gw7h lG6oBk zl E dJS Eig f0Q 1B oCMa nS1u LzlQ3H nA u qHG Plc Iad FL Rkdj aLg0 VAPAn7 c8 D qoV 8bR CvO zq k5e0 Zh3t zJBWBO RS w Zs9 CgF bGo 1E FAK7 EesL XYWaOP F4 n XFo GQl h3p G7 oNtG 4mpT MwEqV4 pO 8 fMEQ79}   \end{align} which completes the proof. \end{proof} \par The following result asserts positivity of the stopping time. We fix $K>1$ and choose a constant $M>K$ that fulfills the requirements in Lemmas \ref{T002} and \ref{L04}.  \par \cole \begin{Lemma}  \label{L05}  Let $p>3$ and $K\geq1$.    Assume that $\nabla\cdot \uu_0=0$, $\DFGDSFSDFGHSFGH_{\TT^3} u_0=0$, and $\Vert u_0\Vert_p\leq K~a.s$. Then there exist $M>K$, a stopping time $\tau_M$ with $\PP(\tau_M>0)=1$, and a subsequence $\{u^{(n_k)}\}$ so that      \begin{align}  \begin{split}  \lim_{k\to\infty}   \left(  \sup_{0\leq t\leq \tau_M}\Vert u^{(n_k)}(t)-u(t)\Vert_p^p  +\DFGDSFSDFGHSFGH_0^{\tau_M}   \Vert u^{(n_k)}(t)-u(t)\Vert_{3p}^p\,dt   \right)  =0\quad{} \PP\mbox{-a.s.}    ,  \end{split}    \llabel{F jfg ktn kw IB8N P60f wfEhjA DF 3 bMq EPV 9U0 o7 fcGq UUL1 0f65lT hL W yoX N4v uSY es 96Sc 2HbJ 0hugJM eB 5 hVa EdL TXr No 2L78 fJme hCMd6L SW q ktp Mgs kNJ q6 tvZO kgp1 GBBqG4 mA 7 tMV p8F n60 El QGMx joGW CrvQUY V1 K YKL pPz Vhh uX VnWa UVqL xeS9ef sA i 7Lm HXC ARg 4Y JnvB e46D UuQYkd jd z 5Mf PLH oWI TM jUYM 7Qry u7W8Er 0O g j2f KqX Scl Gm IgqX Tam7 J8UHFq zv b Vvx Niu j6I h7 lxbJ gMQY j5qtga xb M Hwb JT2 tlB si b8i7 zj6F MTLbwJ qH V IiQ 3O0 LNn Ly pZCT VUM1 bcuVYT ej G 3bf hcX 0BV Ql 6Dc1 xiWV K4S4RW 5P y ZEV W8A Yt9 dN VSXa OkkG KiLHhz FY Y K1q NGG EEU 4F xdja S2NR REnhHm B8 V y44 6a3 VCe Ck wjCM e3DG fMiFop vl z Lp5 r0z dXr rB DZQv 9HQ7 XJMJog kJ n sDx WzI N7F Uf veeL 0ljk 83TxrJ FD T vEX LZY pEq 5e mBaw Z8VA zvvzOv CK m K2Q ngM MBA Wc UH8F jSJt hocw4l 9q J TVG sq8 yRw 5z qVSp d9Ar UfVDcD l8 B 1o5 iyU R4K Nq b84i OkIQ GIczg2 nc t txd WfL QlN ns g3BB jX2E TiPrpq ig M OSw 4Cg dGP fi G2HN ZhLe aQwyws ii A WrD jo4 LDb jB ZFDr LMuY dt6k6H n9 w p4V k7t ddF rz CKid QPfC RKUedz V8 z ISv ntB qpu 3c p5q7 J4Fg Bq59pS Md E onG 7PQ CzM cW lVR0 iNJh WHVugW PY d IMg tXB 2ZS ax azHe Wp7r fhkEQ80}    \end{align} for some adapted process $u\in L^p(\Omega, C([0,\tau_M], L^p))\cap L^p(\Omega, L^{p}([0,\tau_M], L^{3p}))$.   \end{Lemma} \colb \par \begin{proof}[Proof of Lemma~\ref{L05}]  Using \eqref{GN} and \eqref{FGHVBNFJHERTWERDFSDFGSDFGSDFGSDFGDFGWERTWERT66}, there exists $M$ be sufficiently large relative to $K$ so that    \begin{align} \begin{split} \lim_{\epsilon\to 0} \sup_{n}\PP\left[ \sup_{0\leq s\leq \tau_n^M\wedge \epsilon}\Vert u^{(n)}(s,\cdot)\Vert_p +\left(\DFGDSFSDFGHSFGH_0^{\tau_n^M\wedge \epsilon}  \Vert u^{(n)}(s,\cdot)\Vert_{3p}^p \, ds\right)^{1/p} \geq \frac{M}{2} \right] =0. \end{split}\label{ee} \end{align} Fix a corresponding value for the constant $S$ in \eqref{FGHVBNFJHERTWERDFSDFGSDFGSDFGSDFGDFGWERTWERT65} and \eqref{FGHVBNFJHERTWERDFSDFGSDFGSDFGSDFGDFGWERTWERT74}. From \eqref{FGHVBNFJHERTWERDFSDFGSDFGSDFGSDFGDFGWERTWERT74} we can infer  the existence of a subsequence $\{n_k\}$ for which   \begin{align}   \begin{split}   \EE\left[\sup_{0\leq r\leq \tau_{n_k}^M\wedge \tau_{n_{k+1}}^M\wedge S}\Vert w^{(n_k, n_{k+1})}(r,\cdot)\Vert_p^p   +\DFGDSFSDFGHSFGH_0^{\tau_{n_k}^M\wedge \tau_{n_{k+1}}^M\wedge S}     \Vert w^{(n_k, n_{k+1})}(r,\cdot)\Vert_{3p}^p\,dr   \right]   \leq 4^{-kp}   .   \end{split}   \label{FGHVBNFJHERTWERDFSDFGSDFGSDFGSDFGDFGWERTWERT81}   \end{align} Now we need to prove that there exists a uniform time interval where this estimate can be applied. Inspired by \cite{GZ}, we introduce stopping times   \begin{equation}\llabel{4qr Ab J FFG 0li i9M WI l44j s9gN lu46Cf P3 H vS8 vQx Yw9 cE yGYX i3wi 41aIuU eQ X EjG 3XZ IUl 8V SPJV gCJ3 ZOliZQ LO R zOF VKq lyz 8D 4NB6 M5TQ onmBvi kY 8 8TJ ONa DfE 2u zbcv fL67 bnJUz8 Sd 7 yx5 jWr oXd Jp 0lSy mIK8 bkKzql jN n 4Kx luF hYL g0 FrO6 yRzt wFTK7Q RN 0 1O2 1Zc HNK gR M7GZ 9nB1 Etq8sq lA s fxo tsl 927 c6 Y8IY 8T4x 0DRhoh 07 1 8MZ Joo 1oe hV Lr8A EaLK hyw6Sn Dt h g2H Mt9 D1j UF 5b4w cjll AvvOSh tK 8 06u jYa 0TY O4 pcVX hkOO JVtHN9 8Q q q0J 1Hk Ncm LS 3MAp Q75A lAkdnM yJ M qAC erD l5y Py s44a 7cY7 sEp6Lq mG 3 V53 pBs 2uP NU M7pX 6sy9 5vSv7i IS 8 VGJ 08Q KhA S3 jIDN TJsf bhIiUN fe H 9Xf 8We Cxm BL gzJT IN5N LhvdBO zP m opx YqM 4Vh ky btYg a3XV TTqLyA Hy q Yqo fKP 58n 8q R9AY rRRe tBFxHG g7 p duM 8gm 1Td pl RKIW 9gi5 ZxEEAH De A sfP 5hb xAx bW CvpW k9ca qNibi5 A5 N Y5I lVA S3a hA aB8z zUTu yK55gl DL 5 XO9 CpO RXw rE V1IJ G7wE gpOag9 zb J iGe T6H Emc Ma QpDf yDxh eTNjwf wM x 2Ci pkQ eUj RU VhCf NMo5 DZ4h2a dE j ZTk Ox9 46E eU IZv7 rFL6 dj2dwg Rx g bOb qJs Yms Dq QAss n9g2 kCb1Ms gK f x0Y jK0 Glr XO 7xI5 WmQH ozMPfC XT m Dk2 Tl0 oRr nZ vAsF r7wY EJHCd1 xz C vMm jeR 4ct k7 cSEQ81-1}   \eta_k= \inf \left\{t>0: \sup_{0\leq r\leq t}\Vert u^{(n_k) }(r,\cdot)\Vert_p   +\left(   \DFGDSFSDFGHSFGH_0^{t}    \Vert u^{(n_k) }(r,\cdot)\Vert_{3p}^p\,dr   \right)^{1/p}   >\frac{M}{2}+2^{-k}  \right\}   \end{equation}  and probability events   \begin{equation*}    \Omega_N= \bigcap_{k=N}^{\infty} \left\{\omega: \sup_{0\leq r\leq \eta_k\wedge \eta_{k+1}\wedge S}\Vert w^{(n_k, n_{k+1})}(r,\omega)\Vert_p    +\left(\DFGDSFSDFGHSFGH_0^{\eta_k\wedge \eta_{k+1}\wedge S}     \Vert w^{(n_k, n_{k+1})}(r,\omega)\Vert_{3p}^{p}dr\right)^{\frac{1}{p}} <2^{-k-2} \right\}    .   \end{equation*}  Note that $\eta_k\leq \tau_{n_k}^M$ (cf.~\eqref{FGHVBNFJHERTWERDFSDFGSDFGSDFGSDFGDFGWERTWERT64}). Then, by Chebyshev's inequality,   \begin{align*}   \begin{split}   &\PP \left(   \sup_{0\leq r\leq \eta_k\wedge \eta_{k+1}\wedge S}   \Vert w^{(n_k, n_{k+1})}(r,\omega)\Vert_p   +\left(   \DFGDSFSDFGHSFGH_0^{\eta_k\wedge \eta_{k+1}\wedge S}    \Vert w^{(n_k, n_{k+1})}(r,\omega)\Vert_{3p}^{p}   \,dr   \right)^{\frac{1}{p}} \geq 2^{-k-2}    \right)\\   &\indeq \leq    \PP \left(   \sup_{0\leq r\leq \eta_k\wedge \eta_{k+1}\wedge S}   \Vert w^{(n_k, n_{k+1})}(r,\omega)\Vert_p^p   \geq 2^{-kp-3p}    \right)\\   &\indeq\indeq\indeq   +\PP\left(   \DFGDSFSDFGHSFGH_0^{\eta_k\wedge \eta_{k+1}\wedge S}    \Vert w^{(n_k, n_{k+1})}(r,\omega)\Vert_{3p}^{p}   \,dr   \geq 2^{-kp-3p}    \right)\\   &\indeq\leq 2^{kp+3p+1}4^{-kp}=2^{-kp+3p+1}.   \end{split}   \end{align*}   Next, by the Borel-Cantelli lemma,    \begin{equation*}   \PP \left(   \bigcap_{N=1}^{\infty}\bigcup_{k=N}^{\infty}\left\{ \sup_{0\leq r\leq \eta_k\wedge \eta_{k+1}\wedge S}\Vert w^{(n_k, n_{k+1})}(r,\omega)\Vert_p   +\left(\DFGDSFSDFGHSFGH_0^{\eta_k\wedge \eta_{k+1}\wedge S}    \Vert w^{(n_k, n_{k+1})}(r,\omega)\Vert_{3p}^{p}dr\right)^{\frac{1}{p}} \geq 2^{-k-2}  \right\}   \right)=0.   \end{equation*} This shows that $\PP(\bigcup_{N=1}^{\infty}\Omega_N)=1$. Note that $\eta_k\geq \eta_{k+1}\wedge S$ in $\Omega_N$ if $N\leq k$, because  \begin{equation*} \Omega_N\cap \{\eta_k< \eta_{k+1}\wedge S\}\subseteq \left\{\omega: \sup_{0\leq r\leq \eta_{k}}\Vert w^{(n_k, n_{k+1})}(r,\omega)\Vert_p +\left(\DFGDSFSDFGHSFGH_0^{\eta_{k}}  \Vert w^{(n_k, n_{k+1})}(r,\omega)\Vert_{3p}^{p}dr\right)^{\frac{1}{p}} <2^{-k-2} \right\}, \end{equation*}  and almost surely in $\Omega_N\cap \{\eta_k< \eta_{k+1}\wedge S\}$,  \begin{align*} \begin{split}
&\sup_{0\leq r\leq \eta_{k}} \Vert u^{(n_{k+1}) }(r,\cdot)\Vert_p +\left( \DFGDSFSDFGHSFGH_0^{\eta_{k}}  \Vert u^{(n_{k+1}) }(r,\cdot)\Vert_{3p}^p\,dr \right)^{1/p}\\ &\indeq\indeq \geq  \sup_{0\leq r\leq \eta_{k}} \Vert u^{(n_k) }(r,\cdot)\Vert_p -\sup_{0\leq r\leq \eta_{k}}\Vert w^{(n_k, n_{k+1})}(r,\cdot)\Vert_p \\ &\indeq \indeq\indeq +\left( \DFGDSFSDFGHSFGH_0^{\eta_{k}}  \Vert u^{(n_k) }(r,\cdot)\Vert_{3p}^p\,dr \right)^{1/p} -\left(\DFGDSFSDFGHSFGH_0^{\eta_{k}}  \Vert w^{(n_k, n_{k+1})}(r,\cdot)\Vert_{3p}^{p}dr\right)^{\frac{1}{p}} \\ &\indeq\indeq \geq  \frac{M}{2}+2^{-k} -2^{-k-2} > \frac{M}{2}+2^{-k-1}, \end{split} \end{align*} which contradicts $\eta_k< \eta_{k+1}\wedge S$. Then, $\Omega_N\cap \{\eta_k< \eta_{k+1}\wedge S\}=\emptyset$, and $\{\eta_k(\omega)\wedge S\}$ is a non-increasing sequence in $\Omega_N$. Also note that $\Omega_N$ monotonically expands to the whole probability space, and then $\tau_M=\lim_{k\to\infty}\eta_k\wedge S$ is well-defined almost everywhere in $\Omega$. Furthermore,   \begin{align*}   \begin{split}   &\PP(\tau_M<\epsilon)=\PP\left(\bigcup_{l=1}^{\infty}\bigcap_{k=l}^{\infty}\{\eta_k\wedge S <\epsilon\}\right)    = \sup_l    \PP\left(\bigcap_{k=l}^{\infty}\{\eta_k\wedge S <\epsilon\}\cap \Omega_l\right)    \\ &\indeq    \leq \sup_l    \PP\big(\eta_l\wedge S <\epsilon\big)    \leq \sup_l     \PP\left(    \sup_{0\leq r\leq \epsilon}\Vert u^{(n_l) }(r,\cdot)\Vert_p   +\left(\DFGDSFSDFGHSFGH_0^{\epsilon}    \Vert u^{(n_l)}(r,\cdot)\Vert_{3p}^{p} dr\right)^{1/p}   >\frac{M}{2}+2^{-l}\right)   \end{split}    \label{FGHVBNFJHERTWERDFSDFGSDFGSDFGSDFGDFGWERTWERT82}   \end{align*}
if $\epsilon<S$, which by \eqref{ee} yields   \begin{equation*}    \PP(\tau_M=0)=\PP\left(\bigcap_{m=1}^{\infty}\{\tau_M<1/m\}\right)\leq\lim_{m\to\infty} \PP(\tau_M<1/m)=0.   \label{FGHVBNFJHERTWERDFSDFGSDFGSDFGSDFGDFGWERTWERT83}   \end{equation*} \par   It remains to show that $\{u^{(n_k)}\}$ has an $\omega$-pointwise limit and the limit belongs to $L^p(\Omega, C([0,\tau_M], L^p))\cap L^p(\Omega, L^{p}([0,\tau_M], L^{3p}))$. Indeed, by \eqref{FGHVBNFJHERTWERDFSDFGSDFGSDFGSDFGDFGWERTWERT81} and \cite[Lemma~5.2]{KXZ},  \begin{equation*} u^{(n_k)}\indic_{\Omega_N} \xrightarrow{k \to \infty} u_N \mbox{ in }C([0,\tau_M], L^p)\cap L^{p}([0,\tau_M], L^{3p}) \comma \PP\text{-a.e.} \end{equation*} for some adapted process $u_{N}$ and for all $N\in \NNp$. Then, $\{u^{(n_k)}\}$ converges in the same manner to $u=\lim_{N\to\infty} u_{N}$ on $\Omega$. Using \eqref{FGHVBNFJHERTWERDFSDFGSDFGSDFGSDFGDFGWERTWERT65}, we obtain that $\{u^{(n_k)}\indic_{\Omega_N}\}$ are uniformly bounded in $L^p(\Omega, L^{\infty}([0,\tau_M], L^p))\cap L^p(\Omega, L^{p}([0,\tau_M], L^{3p}))$. Thus, the same conclusion holds for $u_{N}$ and $u$, which completes the proof. \end{proof} \par Next, we prove the pathwise uniqueness of local strong solutions.  \par \cole \begin{Lemma}  \label{L06}  Let $p>3$.    Assume that $\nabla\cdot \uu_0=0$, $\DFGDSFSDFGHSFGH_{\TT^3} u_0=0$, and $u_0\in L^p(\Omega, L^p)$. Then for any pair of local strong solutions $(v^{(1)}, \tau)$ and $(v^{(2)},\tau)$ of \eqref{FGHVBNFJHERTWERDFSDFGSDFGSDFGSDFGDFGWERTWERT01}--\eqref{FGHVBNFJHERTWERDFSDFGSDFGSDFGSDFGDFGWERTWERT03} that satisfy  \begin{equation}\label{eng}  \EE\left[\sup_{0\leq s\leq \tau}\Vert v(s)\Vert_p^p  +\DFGDSFSDFGHSFGH_0^{\tau} \sum_j\DFGDSFSDFGHSFGH_{\TT^3}\vert \nabla (|v_j(s,x)| ^{p/2})\vert^2 \,dx\,ds\right]\leq C\EE[\Vert u_0\Vert^p_p+1],  \end{equation} we have $  \PP(v^{(1)}(t)=v^{(2)}(t), ~\forall t\in [0,\tau])=1 $. \end{Lemma} \par \colb \begin{proof}[Proof of Lemma~\ref{L06}]   Let $M>0$ and introduce stopping times \begin{align} \begin{split} \eta_i^M=\inf\left\{ t>0:  \sup_{0\leq s\leq t}\Vert v(s)\Vert_p^p +\DFGDSFSDFGHSFGH_0^{t} \sum_j\DFGDSFSDFGHSFGH_{\TT^3}\vert \nabla (|v_j(s,x)| ^{p/2})\vert^2 \,dx\,ds \geq M \right\} . \end{split} \llabel{2f ncvf aN6AO2 nI h 6nk VkN 8tT 8a Jdb7 08jZ ZqvL1Z uT 5 lSW Go0 8cL J1 q3Tm AZF8 qhxaoY JC 6 FWR uXH Mx3 Dc w8uJ 87Q4 kXVac6 OO P DZ4 vRt sP0 1h KUkd aCLB iPSAtL u9 W Loy xMa Bvi xH yadn qQSJ WgSCkF 7l H aO2 yGR IlK 3a FZen CWqO 9EyRof Yb k idH Qh1 G2v oh cMPo EUzp 6f14Ni oa r vW8 OUc 426 Ar sSo7 HiBU KdVs7c Oj a V9K EUt Kne 4V IPuZ c4bP RFB9AB fq c lU2 ct6 PDQ ud t4VO zMMU NrnzJX px k E2N B8p fJi M4 UNg4 Oi1g chfOU6 2a v Nrp cc8 IJm 2W nVXL D672 ltZTf8 RD w qTv BXE WuH 2c JtO1 INQU lOmEPv j3 O OvQ SHx iKc 8R vNnJ NNCC 3KXp3J 8w 5 0Ws OTX HHh vL 5kBp Kr5u rqvVFv 8u p qgP RPQ bjC xm e33u JUFh YHBhYM Od 0 1Jt 7yS fVp F0 z6nC K8gr RahMJ6 XH o LGu 4v2 o9Q xO NVY8 8aum 7cZHRN XH p G1a 8KY XMa yT xXIk O5vV 5PSkCp 8P B oBv 9dB mep ms 7DDU aicX Y8Lx8I Bj F Btk e2y ShN GE 7a0o EMFy AUUFkR WW h eDb HhA M6U h3 73Lz TTTx xm6ybD Bs I IIA PHh i83 7r a970 Fam4 O7afXU Gr f 0vW e52 e8E Py BFZ0 wxBz ptJf8L iZ k dTZ SSP pSz rb GEpx b4KX LHLg1V Pa f 7ys vYs FJb 8r DpAM Knzq Dg7g2H wC r uQN DBz Z5S NM ayKB 6RIe PFIHFQ aw r RHA x38 CHh oB GVIR vxSM Yf0g8h ac i bKG 3Cu Sl5 jT Kl42 o6gA OYYHUB 2S V O3R c4EQ26} \end{align} If $\Vert u_0(\omega)\Vert_p< M$, then $\eta_i^M(\omega)>0$, otherwise $\eta_i^M(\omega)=0$ for $i=1,2$. Define $\eta^M=\eta_1^M\wedge \eta_2^M\wedge \tau$. Due to \eqref{eng}, $\lim_{M\to\infty}\PP(\eta^M= \tau)=1$. Let $S>0$ and denote $w=v^{(1)}-v^{(2)}$. On $[0,\eta^M\wedge S]$, $w$ satisfies \begin{align*} \begin{split} dw -\Delta w\,dt  &= -\mathcal{P}\bigl(( w\cdot \nabla)v^{(2)}\bigr)-\mathcal{P}\bigl(( v^{(1)}\cdot \nabla)w\bigr)\,dt +\left( \sigma(v^{(1)})-\sigma(v^{(2)}) \right)\,d\WW_t,  \\ \nabla\cdot w &= 0, \\ w( 0)&= 0\quad{}\mbox{a.s.} \end{split} \end{align*} The It\^{o} formula yields, as in \eqref{FGHVBNFJHERTWERDFSDFGSDFGSDFGSDFGDFGWERTWERT68} and \eqref{FGHVBNFJHERTWERDFSDFGSDFGSDFGSDFGDFGWERTWERT78},  \begin{align*} \begin{split} \EE\left[ \sup_{0\leq s\leq \eta^M\wedge S}\Vert w(s)\Vert_p^p +\DFGDSFSDFGHSFGH_0^{\eta^M\wedge S} \sum_j\DFGDSFSDFGHSFGH_{\TT^3}\vert \nabla (|w_j(s,x)| ^{p/2})\vert^2 \,dx\,ds \right]  \leq  \EE\left[ \DFGDSFSDFGHSFGH_0^{\eta^{M}\wedge S} (\varepsilon\Vert  w  \Vert_{3p}^p+C_{M,\varepsilon} \Vert  w  \Vert_{p}^p ) \,dt \right] , \end{split} \end{align*} which can be further simplified to  \begin{align*} \begin{split} \EE\left[ \sup_{0\leq s\leq \eta^M\wedge S}\Vert w(s)\Vert_p^p \right] \leq  C_{M}\DFGDSFSDFGHSFGH_0^{S} \EE\left[ \sup_{0\leq s\leq \eta^M\wedge t} \Vert  w  \Vert_{p}^p \right] \,dt. \end{split} \end{align*} By Gr\"{o}nwall's lemma, we conclude that $w\equiv 0$ a.s.~on $[0, \eta^M\wedge S]$. Since $S$ is an arbitrary and positive constant independent of $M$, by sending both $M$ and $S$ to the positive infinity, we obtain the pathwise uniqueness on $[0,\tau]$.  \end{proof} \par \begin{proof}[Proof of Theorem~\ref{T04}]  Using the notation in the proof of Lemma~\ref{L05}, we first impose $\Vert u_0\Vert_p\leq K$~a.s.\ and show that $u$ is a strong solution to \eqref{FGHVBNFJHERTWERDFSDFGSDFGSDFGSDFGDFGWERTWERT01}--\eqref{FGHVBNFJHERTWERDFSDFGSDFGSDFGSDFGDFGWERTWERT03} on $[0,\tau_M]$. Denote $\zeta_k=\inf_{l\geq k} \eta_l\wedge S$. By right-continuity of the filtration $\mathcal{F}_t$, $\{\zeta_k\}_{k\in \NNp}$ are also stopping times. Note that $\zeta_k\leq \zeta_l$ if $k< l$ and $\zeta_l\leq \eta_l\wedge S$.  \par  Since each $\{(u^{(n_l)}, \eta_l)\}$ is a local solution to an approximating equation, then componentwise,   \begin{align}   \begin{split}    &\big(\indic_{[0,\zeta_m]}(s)u^{(n_l)} ( s),\phi\big)    = \DFGDSFSDFGHSFGH_0^s \indic_{[0,\zeta_m]}(r)    \big(u^{(n_l)} ,\Delta\phi\big)    \,dr+    \sum_{j}\DFGDSFSDFGHSFGH_0^s \indic_{[0,\zeta_m]}(r)\bigl(\mathcal{S}_{n_l}\mathcal{P}\bigl( u^{(n_l)}_j u^{(n_l)} \bigr),\DFGDSFSDFGHSFGI_{j}\phi\bigr)\,dr    \\&\indeq\indeq    +\DFGDSFSDFGHSFGH_0^s \indic_{[0,\zeta_m]}(r)    \big(\mathcal{S}_{n_l}\sigma(u^{(n_l)} ),\phi\big)    \,d\WW_r    +(\mathcal{S}_{n_l}\uu_0,\phi)    \comma  (s,\omega)\text{-a.e.}    ,   \end{split}   \label{FGHVBNFJHERTWERDFSDFGSDFGSDFGSDFGDFGWERTWERT58}   \end{align} for all $\phi\in C^{\infty}(\TT^3)$, $s\in [0, S]$, and $m\leq l$. Also, \eqref{FGHVBNFJHERTWERDFSDFGSDFGSDFGSDFGDFGWERTWERT81} holds on $[0,\zeta_m]$ for $\{u_l\}_{l\geq m}$. Utilizing Lemma~\ref{L05} and the boundedness of $\{u^{(n_k)}\}$ in $L^p_{\omega}L^{\infty}_t L_x^p$ on $[0,\zeta_m]$, we may pass to the limit in \eqref{FGHVBNFJHERTWERDFSDFGSDFGSDFGSDFGDFGWERTWERT58} and obtain   \begin{align}   \begin{split}
   &\DFGDSFSDFGHSFGH_0^s \indic_{[0,\zeta_m]}(r)\big(u^{(n_l)} ,\Delta\phi\big)\,dr    +    \sum_{j}\DFGDSFSDFGHSFGH_0^s \indic_{[0,\zeta_m]}(r)\big(\mathcal{S}_{n_l}\mathcal{P}( u^{(n_l)}_j u^{(n_l)} ),\DFGDSFSDFGHSFGI_{j}\phi\big)\,dr    \\&\indeq     \to\DFGDSFSDFGHSFGH_0^s\indic_{[0,\zeta_m]}(r) \bigl((\uu,\Delta\phi)+(\mathcal{P}(u_j\uu),\DFGDSFSDFGHSFGI_{j}\phi)\bigr)\,dr   ,   \end{split}    \label{FGHVBNFJHERTWERDFSDFGSDFGSDFGSDFGDFGWERTWERT59}   \end{align} for a.e.~$(s,\omega)$ as $l\to \infty$. Also, by the BDG inequality,    \begin{align*}   \begin{split}    &\EE\left[\sup_{s\in[0,S]}\left|\DFGDSFSDFGHSFGH_0^s     \indic_{[0,\zeta_m]}(r)    (\mathcal{S}_{n_l}\sigma(u^{(n_l)} )-\sigma(\uu),\phi)\,d\WW_r\right|\right]     \\ \leq &    C\EE\left[    \left(    \DFGDSFSDFGHSFGH_0^{\zeta_m}     \bigl\Vert    \bigl(    \mathcal{S}_{n_l}(\sigma(u^{(n_l)} )-\sigma(\uu)),\phi    \bigr)    \bigr\Vert_{ l^2}^2\, dr    \right)^{1/2}\right]    +    C\EE\left[\left(    \DFGDSFSDFGHSFGH_0^{\zeta_m} \bigl\Vert\bigl(    (\mathcal{S}_{n_l}-I )\sigma(\uu),\phi    \bigr)\bigr\Vert_{ l^2}^2\, dr    \right)^{1/2}\right],   \end{split}   \end{align*}   where     \begin{align*}   \begin{split}   &\EE\left[   \left(   \DFGDSFSDFGHSFGH_0^{\zeta_m}    \bigl\Vert   \bigl(   \mathcal{S}_{n_l}(\sigma(u^{(n_l)} )-\sigma(\uu)),\phi   \bigr)   \bigr\Vert_{ l^2}^2\, dr   \right)^{1/2}\right]       \\ &\indeq\leq  C   \EE\left[   \left(   \DFGDSFSDFGHSFGH_0^{\zeta_m} \left(   \DFGDSFSDFGHSFGH_{\TT^3}   \Vert   \mathcal{S}_{n_l}(\sigma(u^{(n_l)} )-\sigma(\uu))   \Vert_{ l^2}^2\,dx   \right)   \Vert \phi\Vert_2^2    \, dr\right)^{1/2}   \right]   \\ &\indeq\leq  C   \Vert \phi\Vert_2   \EE\left[   \left(   \DFGDSFSDFGHSFGH_0^{\zeta_m} \left(   \DFGDSFSDFGHSFGH_{\TT^3}   \Vert   \mathcal{S}_{n_l}(\sigma(u^{(n_l)} )-\sigma(\uu))   \Vert_{ l^2}^p\,dx   \right)^{2/p}   \, dr\right)^{1/2}   \right]   \\&\indeq    \leq C \Vert \phi\Vert_2 \EE\left[\DFGDSFSDFGHSFGH_0^{\zeta_m} \Vert   \mathcal{S}_{n_l}(\sigma(u^{(n_l)} )-\sigma(\uu))  \Vert_{ \mathbb{L}^p}^p\, dr\right]   \end{split}   \end{align*} and   \begin{align*} \begin{split} &\EE\left[\left( \DFGDSFSDFGHSFGH_0^{\zeta_m} \bigl\Vert\bigl( (\mathcal{S}_{n_l}-I )\sigma(\uu),\phi \bigr)\bigr\Vert_{ l^2}^2\, dr \right)^{1/2}\right] \leq C \Vert \phi\Vert_2  \EE\left[\DFGDSFSDFGHSFGH_0^{\zeta_m} \Vert (\mathcal{S}_{n_l}-I )\sigma(\uu) \Vert_{ \mathbb{L}^p}^p\, dr\right]. \end{split} \end{align*}   Then by Lemma~\ref{L01} and the assumptions on $\sigma$,   \begin{align*} \begin{split} &\EE\left[\sup_{s\in[0,S]}\left|\DFGDSFSDFGHSFGH_0^s  \indic_{[0,\zeta_m]}(r) (\mathcal{S}_{n_l}\sigma(u^{(n_l)} )-\sigma(\uu),\phi)\,d\WW_r\right|\right] \\&\indeq \leq CS\Vert \phi\Vert_2 \EE\left[\sup_{s\in[0,\zeta_m]} \Vert u^{(n_l)} -u\Vert_{ p}^p\right] +  C \Vert \phi\Vert_2 \EE\left[\DFGDSFSDFGHSFGH_0^{\zeta_m} \Vert (\mathcal{S}_{n_l}-I )\sigma(\uu) \Vert_{ \mathbb{L}^p}^p\, dr\right]. \end{split} \end{align*}   The right-hand side approaches to zero as $l\to \infty$. Hence, we may infer the existence of a further subsequence, which for simplicity we still denote by $\{n_l\}$, such that    \begin{equation}    \DFGDSFSDFGHSFGH_0^s \indic_{[0,\zeta_m]}(r)(\mathcal{S}_{n_l}\sigma(u^{(n_l)}),\phi)\,d\WW_r         \xrightarrow{l \to \infty}        \DFGDSFSDFGHSFGH_0^s \indic_{[0,\zeta_m]}(r)(\sigma(\uu),\phi)\,d\WW_r    \comma (s,\omega)\text{-a.e.}    \label{FGHVBNFJHERTWERDFSDFGSDFGSDFGSDFGDFGWERTWERT61}   \end{equation} Combining \eqref{FGHVBNFJHERTWERDFSDFGSDFGSDFGSDFGDFGWERTWERT59} and \eqref{FGHVBNFJHERTWERDFSDFGSDFGSDFGSDFGDFGWERTWERT61}, we obtain  \begin{align}  \begin{split}  &\indic_{[0,\zeta_m]}(s)\big(u ( s),\phi\big)  = (\uu_0,\phi)  +\indic_{[0,\zeta_m]}(s)\DFGDSFSDFGHSFGH_0^s   \big(u ,\Delta\phi\big)  \,dr+\indic_{[0,\zeta_m]}(s)\DFGDSFSDFGHSFGH_0^s   \big(\sigma(u ),\phi\big)  \,d\WW_r  \\&\indeq\indeq  +\sum_{j}\indic_{[0,\zeta_m]}(s)\DFGDSFSDFGHSFGH_0^s \bigl(\mathcal{P}\bigl( u_j u \bigr),\DFGDSFSDFGHSFGI_{j}\phi\bigr)\,dr  \comma  (s,\omega)\text{-a.e.}  ,  \end{split}  \label{FGHVBNFJHERTWERDFSDFGSDFGSDFGSDFGDFGWERTWERT76}  \end{align}  for all $m\in\NNp$. Sending $m\to \infty$ in \eqref{FGHVBNFJHERTWERDFSDFGSDFGSDFGSDFGDFGWERTWERT76} and noting that $\lim_{m\to \infty}\indic_{[0,\zeta_m]}(t)=\indic_{[0,\tau_M]}(t)$ $\PP$-a.s., we conclude that $\uu$ is indeed a strong solution to \eqref{FGHVBNFJHERTWERDFSDFGSDFGSDFGSDFGDFGWERTWERT01}--\eqref{FGHVBNFJHERTWERDFSDFGSDFGSDFGSDFGDFGWERTWERT03} on $[0, \tau_M]$. Moreover, due to \eqref{FGHVBNFJHERTWERDFSDFGSDFGSDFGSDFGDFGWERTWERT65}, we have \begin{align*} \begin{split} \EE\left[ \indic_{\Omega_N}\sup_{0\leq s\leq \tau_M}\Vert u^{(n_l)}(s,\cdot)\Vert_p^p +\indic_{\Omega_N}\DFGDSFSDFGHSFGH_0^{\tau_M} 
\sum_{j}    \DFGDSFSDFGHSFGH_{\TT^3} | \nabla (|u^{(n_l)}_j(s,x)|^{p/2})|^2 \,dx\,ds \right] \leq  C\EE\bigl[\Vert\uu_0\Vert_p^p +1 \bigr] \end{split} \end{align*} if $N\leq l$. Then we use Lemmas \ref{L02} and \ref{L04}, send $l\to \infty$ in above inequality, and send $N\to \infty$, arriving at \begin{align} \begin{split} \EE\left[ \sup_{0\leq s\leq \tau_M}\Vert u(s,\cdot)\Vert_p^p +\DFGDSFSDFGHSFGH_0^{\tau_M}  \sum_{j}    \DFGDSFSDFGHSFGH_{\TT^3} | \nabla (|u_j(s,x)|^{p/2})|^2 \,dx\,ds \right] \leq  C\EE\bigl[\Vert\uu_0\Vert_p^p +1 \bigr].\label{FGHVBNFJHERTWERDFSDFGSDFGSDFGSDFGDFGWERTWERT84} \end{split} \end{align} \par To remove the condition $\Vert u_0\Vert_p\leq K$~a.s.,  we denote the local strong solution corresponding to the initial data $u_0\indic_{k\leq \Vert u_0\Vert_p< k+1}$ by $(u_{(k)},\tau_k)$, i.e., \begin{align} \begin{split} &\indic_{[0,\tau_k]}(s)\big(u_{(k)} ( s),\phi\big) = (\uu_0\indic_{k\leq \Vert u_0\Vert_p< k+1},\phi) +\indic_{[0,\tau_k]}(s)\DFGDSFSDFGHSFGH_0^s  \big(u_{(k)} ,\Delta\phi\big) \,dr+\indic_{[0,\tau_k]}(s)\DFGDSFSDFGHSFGH_0^s  \big(\sigma(u_{(k)} ),\phi\big) \,d\WW_r \\&\indeq\indeq +\sum_{j}\indic_{[0,\tau_k]}(s)\DFGDSFSDFGHSFGH_0^s \bigl(\mathcal{P}\bigl( (u_{(k)})_j u_{(k)} \bigr),\DFGDSFSDFGHSFGI_{j}\phi\bigr)\,dr \comma  (s,\omega)\text{-a.e.} , \end{split}\label{FGHVBNFJHERTWERDFSDFGSDFGSDFGSDFGDFGWERTWERT90} \end{align}  for all $k\in \NNp$. Define  \[ u=\sum_{k=0}^{\infty} u_{(k)} \indic_{k\leq \Vert u_0\Vert_p< k+1}\comma \tau=\sum_{k=0}^{\infty} \tau_k \indic_{k\leq \Vert u_0\Vert_p< k+1}. \] Since $\PP(\tau_k>0)=1$ for all $\tau_k$, we have \[ \PP(\tau>0)=\sum_{k=0}^{\infty}\PP(\tau_k>0| k\leq \Vert u_0\Vert_p< k+1)\PP(k\leq \Vert u_0\Vert_p< k+1)=\sum_{k=0}^{\infty}\PP(k\leq \Vert u_0\Vert_p< k+1)=1. \] Next, note $\indic_{[0,\tau_k]}\indic_{k\leq \Vert u_0\Vert_p< k+1}=\indic_{[0,\tau]}\indic_{k\leq \Vert u_0\Vert_p< k+1}$. Also,  $\sum_{k=0}^{\infty}\sigma(u_{(k)} )\indic_{k\leq \Vert u_0\Vert_p< k+1}$ agrees with $\sigma(u )$ in $\mathbb{L}^p$. Multiplying both sides of \eqref{FGHVBNFJHERTWERDFSDFGSDFGSDFGSDFGDFGWERTWERT90} by $\indic_{k\leq \Vert u_0\Vert_p< k+1}$ and summing over $k$, we obtain  \begin{align*} \begin{split} &\indic_{[0,\tau]}(s)\big(u ( s),\phi\big) = (\uu_0,\phi) +\indic_{[0,\tau]}(s)\DFGDSFSDFGHSFGH_0^s  \big(u ,\Delta\phi\big) \,dr+\indic_{[0,\tau]}(s)\DFGDSFSDFGHSFGH_0^s  \big(\sigma(u ),\phi\big) \,d\WW_r \\&\indeq\indeq +\sum_{j}\indic_{[0,\tau]}(s)\DFGDSFSDFGHSFGH_0^s \bigl(\mathcal{P}\bigl( u_j u\bigr),\DFGDSFSDFGHSFGI_{j}\phi\bigr)\,dr \comma  (s,\omega)\text{-a.e.}, \end{split} \end{align*}  namely, $(u,\tau)$ is a local solution associated with a general initial data $u_0\in L^p(\Omega, L^p(\TT^3))$.  \par Since $u_{(k)}\in C([0,\tau_k], L^p)$ almost surely, we have $u\in C([0,\tau], L^p)$ almost surely. In addition, using \eqref{FGHVBNFJHERTWERDFSDFGSDFGSDFGSDFGDFGWERTWERT84} and the pathwise uniqueness, we obtain \begin{align} \begin{split} &\EE\left[\sup_{0\leq s\leq \tau}\Vert\uu(s,\cdot)\Vert_p^p +\sum_{j}\DFGDSFSDFGHSFGH_0^{\tau} \DFGDSFSDFGHSFGH_{\TT^3} | \nabla (|\uu_j(s,x)|^{p/2})|^2 \,dx\,ds \right] \\&\indeq =\lim_{k\to\infty} \EE\left[\indic_{0\leq \Vert u_0\Vert_p< k+1}\left( \sup_{0\leq s\leq \tau}\Vert\uu(s,\cdot)\Vert_p^p +\sum_{j}\DFGDSFSDFGHSFGH_0^{\tau} \DFGDSFSDFGHSFGH_{\TT^3} | \nabla (|\uu_j(s,x)|^{p/2})|^2 \,dx\,ds \right)\right] \\&\indeq \leq \lim_{k\to\infty} C\EE\bigl[\indic_{0\leq \Vert u_0\Vert_p< k+1}\Vert\uu_0\Vert_p^p\bigr]+C\leq  C\EE\bigl[\Vert\uu_0\Vert_p^p\bigr]+C , \end{split} \llabel{w hR8 pw krrA NA4j 7MfcEM al 4 HwK PTg ZaZ 9G 8sev uwIA hkhR8W ga f zJA 0FV NmS Cw UB0Q JDgR jCSVSr sG M bWA Bxv zOM My cNSO ylZz wFiNPc mf Q hwZ Pan Lp0 1E UVHM A2dE 0nLuNV xK x co7 opb QzR aA lowo Vtor qU5eUX tE l 1qh 1IP CPE uV Pxcn TwZv KJTpHZ pq x XOF aGs rQN Pn uqc9 CMD8 mJZoMO Cy 6 XHj WAf EqI 95 Zjgc PdV8 maWwkH lM 3 0Vw DjX lx1 Qf gyLF xe5i wJDnJI rU K L6e CVt h3g X8 RLAC CC2q 5kMYXo 8N s DfA n3O sap ZT 8U3F dMIE DKqMo0 Yv b wCG MR6 6XT yI OtLU uC6c mcOFsv pW T niQ mu0 PeH EF 9Imo lIuT hHWHwh 8J z 4hC 0rK 2Gd Nz LXiE Y7Vu QfRbXp iQ n Pps 9gM A8m Wk yXsY FLoi Rtl2Kl 2p I 9bS nyi 07m UZ qhEs BOCg I4F5AF Fd j X3w f0W u2Y qd dp2Z Ukje FMAxnD ls u t9q zby RgD Wr HldN Zewz EK1cSw WJ Z ywl oSo f6z VD AB6e r0o2 HZY1tr Zh B uL5 zYz rAU dM KXVK GWKI HOqqx1 zj 8 tlp xuU D83 eL Uerj xfHN MZlaqZ Vh T 6Jk u15 FdL vd eo08 7AsG C8WdoM nf 4 dTo Hw4 7hg lT qjKt AlwR 9ufOhL KT D gWZ hxH FX5 gU 5uN2 S6es PlxKpX zB m gyW Uy5 D01 WD 88a4 YmWR fdmev1 dB v HOm hTB qur Ag TC6y rrRB Pn9QfZ 9T 4 mwI h3x jAt ki MlAl Td6f SQ5iQB BY 6 OEr T3g f0D Ke ECnj gcTX AL8grK Bp f cJv q4f pIh WG FSdh 6LOq g0ao9A EQ69} \end{align} concluding  the proof. \end{proof} \par \section*{Acknowledgments} \rm IK was supported in part by the NSF grant DMS-1907992. \par  \end{document}